\documentclass[11pt,reqno]{amsart}

\usepackage{amsfonts}
\usepackage{eurosym}
\usepackage{amssymb}
\usepackage{amsthm}
\usepackage{amsmath}
\usepackage{amsaddr}
\usepackage{bm}
\usepackage{cite}
\usepackage{mathrsfs}
\usepackage{xcolor}
\usepackage[OT1]{fontenc}
\usepackage[left=2.2cm, right=2.2cm, top=3cm]{geometry}
\usepackage{hyperref}
\hypersetup{colorlinks=true, linkcolor=blue, citecolor=red, urlcolor=blue}

\usepackage{times}

\usepackage{xpatch}
\makeatletter
\AtBeginDocument{\xpatchcmd{\@thm}{\thm@headpunct{.}}{\thm@headpunct{}}{}{}}
\makeatother

\allowdisplaybreaks
\newtheorem{theorem}{Theorem}[section]

\newtheorem{corollary}[theorem]{Corollary}

\newtheorem{lemma}[theorem]{Lemma}
\newtheorem{proposition}[theorem]{Proposition}
\theoremstyle{remark}
\newtheorem{remark}[theorem]{Remark}

\def\d{{\rm d}}
\def \l {\langle}
\def \r {\rangle}

\def\H{\mathbf{H}}
\def\W{\mathbf{W}}

\def\L{\mathbf{L}}
\def\f{\textbf{\textit{f}}}
\def\uu{\textbf{\textit{u}}}
\def\UU{\textbf{\textit{U}}}
\def\vv{\textbf{\textit{v}}}
\def\ww{\textbf{\textit{w}}}

\def\ddt{\frac{\d}{\d t}}

\def\n{\textbf{\textit{n}}}

\def\div{\mathrm{div}}
\def\J{\mathbf{J}}

\def\TT{\mathcal{T}}
\def\wTT{\widetilde{\mathcal{T}}}

\newcommand{\Hone}{H^1_{(0)}}

\newcommand{\dom}{\operatorname{dom}}
\newcommand{\habil}[1]{}

\newcommand{\dist}{\operatorname{dist}}

\newcommand{\loc}{\operatorname{loc}}
\newcommand{\uloc}{\operatorname{uloc}}

\newcommand{\supp}{\operatorname{supp}}
\newcommand{\eps}{\ensuremath{\varepsilon}}

\newcommand{\N}{\ensuremath{\mathbb{N}}}
\newcommand{\R}{\ensuremath{\mathbb{R}}}

\def \au {\rm}
\def \ti {\it}
\def \jou {\rm}
\def \bk {\it}
\def \no#1#2#3 {{\bf #1} (#3), #2.}
\def \eds#1#2#3 {#1, #2, #3.}

\begin{document}

\title[Regularity and stabilization for the incompressible NSCH model with unmatched densities]
{Global regularity and asymptotic stabilization for the incompressible Navier-Stokes-Cahn-Hilliard model with unmatched densities}
\author[Helmut Abels, Harald Garcke \& Andrea Giorgini]{Helmut Abels, Harald Garcke \& Andrea Giorgini}

\address{Fakult\"{a}t f\"{u}r Mathematik\\ 
Universit\"{a}t Regensburg \\
93040 Regensburg, Germany\\
\href{mailto:helmut.abels@ur.de}{helmut.abels@ur.de}\\
\href{mailto:harald.garcke@ur.de}{harald.garcke@ur.de}
}

\address{Politecnico di Milano\\
Dipartimento di Matematica\\
Via E. Bonardi 9, 20133 Milano, Italy\\
\href{mailto:andrea.giorgini@polimi.it}{andrea.giorgini@polimi.it}}



%

\maketitle

\begin{abstract}
We study an initial-boundary value problem for the incompressible Navier-Stokes-Cahn-Hilliard system with non-constant density proposed by Abels, Garcke and Gr\"{u}n in 2012. This model arises in the diffuse interface theory for binary mixtures of viscous incompressible fluids. This system is a generalization of the well-known model H in the case of fluids with unmatched densities. In three dimensions, we prove that any global weak solution (for which uniqueness is not known) exhibits a propagation of regularity in time and stabilizes towards an equilibrium state as $t \rightarrow \infty$. More precisely, the concentration function $\phi$ is a {\it strong} solution of the Cahn-Hilliard equation for (arbitrary) positive times, whereas the velocity field $\uu$ becomes a {\it strong} solution of the momentum equation for large times. Our analysis hinges upon the following key points: a novel global regularity result (with explicit bounds) for the Cahn-Hilliard equation with divergence-free velocity belonging only to $L^2(0,\infty; \H^1_{0,\sigma}(\Omega))$, the energy dissipation of the system, the separation property for large times, a {\it weak-strong} uniqueness type result, and the Lojasiewicz-Simon inequality. Additionally, in two dimensions, we show the existence and uniqueness of global {\it strong} solutions for the full system. 
Finally, we discuss the existence of global weak solutions for the case of the double obstacle potential. 
\end{abstract}

\section{Introduction}
\label{1}

We consider the initial-boundary value problem for the diffuse interface model describing a two-phase flow of incompressible viscous Newtonian fluids in the case of general densities, which was introduced by Abels, Garcke, and Gr\"{u}n in \cite{AGG2012} (also called AGG-model nowadays). This model leads to the system 
\begin{equation}
 \label{AGG}
\begin{cases}
\partial_t \left( \rho(\phi)\uu\right) + \div \left( \uu \otimes \left( \rho(\phi) \uu + \widetilde{\J} \right) \right) - \div \left( \nu(\phi) D \uu \right) + \nabla P= - \div \left( \nabla \phi \otimes \nabla \phi \right),\\
\div \, \uu=0,\\
\partial_t \phi +\uu\cdot \nabla \phi = \div \left(m(\phi)\nabla \mu \right),\\
\mu= -\Delta \phi+\Psi'(\phi),
\end{cases}
\end{equation}
in $\Omega \times (0,\infty)$, where $\Omega$ is a bounded domain in $\mathbb{R}^d$, with $d=2$ and $d=3$. The system \eqref{AGG} is 
completed with the following boundary and initial conditions
\begin{equation}
\label{AGG-bc}
\begin{cases}
\uu=\mathbf{0}, \quad \partial_\n\phi=\partial_\n \mu=0 \quad &\text{on }  \partial \Omega \times (0,T),\\
\uu|_{t=0}=\uu_0, \quad \phi|_{t=0}=\phi_0 \quad &\text{in } \Omega.
\end{cases}
\end{equation}
Here, $\n$ is the unit outward normal vector on $\partial \Omega$, and
 $\partial_\n$ denotes the outer normal derivative on $\partial \Omega$.
The state variables of the system are the volume averaged velocity $\uu\colon \Omega\times [0,\infty)\to \R^d, (x,t)\mapsto \uu(x,t)$, the pressure of the mixture $P\colon \Omega \times [0,\infty)\to \R, (x,t)\mapsto P(x,t)$, and the difference of the fluids volume fractions $\phi\colon \Omega \times [0,\infty)\to [-1,1], (x,t)\mapsto \phi(x,t)$. Here $D \uu= \frac12 (\nabla \uu +(\nabla \uu)^T)$ is the symmetrized gradient of $\uu$. The flux term $\widetilde{\J}$, the density $\rho$ and the viscosity $\nu$ of the mixture are given by
\begin{equation}
\label{Jrhonu}
\widetilde{\J}= -\frac{\rho_1-\rho_2}{2}\nabla \mu, \quad \rho(\phi)= \rho_1 \frac{1+\phi}{2}+ \rho_2 \frac{1-\phi}{2}, \quad
\nu(\phi)=\nu_1 \frac{1+\phi}{2}+ \nu_2 \frac{1-\phi}{2},
\end{equation}
where $\rho_1$, $\rho_2$ and $\nu_1$, $\nu_2$ are the positive homogeneous density and viscosity parameters of the two fluids, respectively. Moreover, $m\colon [-1,1]\to [0,\infty)$ is a mobility coefficient, which in general might depend on $\phi$.  The homogeneous free energy density $\Psi$ is the Flory-Huggins potential 
\begin{equation}  
\label{Log}
\Psi(s)=F(s)-\frac{\theta_0}{2}s^2=\frac{\theta}{2}\bigg[ (1+s)\log(1+s)+(1-s)\log(1-s)\bigg]-\frac{\theta_0}{2} s^2, \quad s \in [-1,1],
\end{equation}
where $\theta$ and $\theta_0$ are constant positive parameters.

As mentioned before, system \eqref{AGG}-\eqref{AGG-bc} is a diffuse interface model for the flow of two incompressible and viscous Newtonian fluids, in which the macroscopically immiscible fluids are considered to be (partly) miscible on a small length scale $\eps>0$. Here the parameter $\eps$ is for simplicity set to $1$. This model is thermodynamically consistent as shown in \cite{AGG2012}. The total energy associated to system \eqref{AGG} is
\begin{equation}
\label{Energy-def}
E(\uu,\phi)= E_{\text{kin}}(\uu, \phi) + E_{\text{free}}(\phi)= 
\int_{\Omega}  \frac12 \rho(\phi) |\uu|^2 \, \d x + \int_{\Omega} \frac12 |\nabla \phi|^2 + \Psi(\phi)   \, \d x,
\end{equation}
and the corresponding energy balance reads as
 \begin{equation}
\label{EE}
\ddt E(\uu, \phi) +\int_{\Omega} \nu(\phi) |D \uu|^2 \, \d x+ 
\int_{\Omega} m(\phi)|\nabla \mu|^2 \, \d x =0
\end{equation}
for sufficiently smooth solutions $(\uu, P, \phi)$. 
Let us note that in the case $\rho_1=\rho_2$, the flux term $\widetilde{\J}=\mathbf{0}$ and the model simplifies to the well-known ``Model H'' (cf.~\cite{HH77,GPV96}). An alternative and thermodynamically consistent model was derived before by Lowengrub and Truskinovski in \cite{LT98}. From the mathematical viewpoint, the latter has the disadvantage that the mass averaged (barycentric) velocity is not divergence free, which is the case in the present model based on a volume averaged velocity. {Moreover, the pressure enters the equation for the chemical potential, which leads to additional difficulties for the construction of weak solutions, cf.~\cite{LTWeakSolutions}, and a strong coupling of the linearized systems, cf.~\cite{LTStrongSolutions}, where existence of strong solutions locally in time was shown. To the best of our knowledge, there are no results on existence of strong solutions for large times (in two space dimensions), regularity for weak solutions or qualitative behavior for large times for this model.} We also refer the interested reader to \cite{B2002, DSS2007, HMR2012, SSBZ2017} for further (different) diffuse interface systems modeling two-phase flows with unmatched densities. {In the recent work \cite{UnifiedNSCH}, the reader can find a unified derivation and comparison of the known diffuse interface models from the physical point of view.}

The existence of global weak solutions to the model \eqref{AGG}-\eqref{AGG-bc} was proved in \cite[Theorem 3.4]{ADG2013} in the case of a strictly positive mobility $m$. The corresponding result in the case of a degenerate mobility in the Cahn-Hilliard equation \eqref{AGG}$_{3-4}$ was shown in \cite{ADG2013-2}. The convergence of a fully discrete numerical scheme to weak solutions was shown by Gr\"un et al. in \cite{GGM2016}. The existence of strong solutions for regular free energy densities $\Psi$ and small times was proved by Weber~\cite{WeberThesis} (cf. also {Abels} and Weber~\cite{AW2020}). The existence of weak solutions in the case of dynamic boundary conditions was proved by  Gal et al.~\cite{GalGW2019}. Results on well-posedness of this system with the free energy density \eqref{Log} in two-space dimensions, locally in time for bounded domains and globally in time in the case of periodic boundary conditions, were achieved in \cite{Gior2021}. In three space dimensions, the existence of strong solutions locally in time was obtained in \cite{G2021}. The existence of weak solutions for a variant of \eqref{AGG} with a non-local Cahn-Hilliard equation was shown by Frigeri~\cite{FrigeriNonlocalAGG} and \cite{FrigeriNonlocalDegAGG} for non-degenerate and degenerate mobility, respectively, and for non-Newtonian fluids in {Abels} and Breit~\cite{AB2018}. Finally, a model for a two-phase flow with magnetic fluids, where \eqref{AGG} is coupled to a gradient flow of the magnetization vector was studied by Kalousek et al.~\cite{KMS2021}, where the existence of weak solutions was proven.

The main goal of the present contribution is to show regularity properties of weak solutions of \eqref{AGG}-\eqref{AGG-bc} and to study the convergence for large times in the case of constant positive mobility.  Let us first recall the result on existence of weak solutions shown in \cite{ADG2013} (in the specific case $a(\cdot)\equiv 1$). We refer to the end of this introduction for the notation and, in particular, for the function spaces used in the following theorem. 
\begin{theorem}[Global existence of weak solutions]
\label{WEAK-SOL}
Let $\Omega$ be a bounded domain in $\mathbb{R}^d$, $d=2,3$, with boundary $\partial \Omega$ of class $C^2$ and $m\colon [-1,1]\to (0,\infty)$ be continuous. Assume that $\uu_0 \in \L_\sigma^2(\Omega)$, $\phi_0 \in H^1(\Omega)$ with $\| \phi_0\|_{L^\infty(\Omega)}\leq 1$ and $\left|\overline{\phi_0} \right|<1$. Then, there exists a global weak solution $(\uu,\phi)$ to \eqref{AGG}-\eqref{AGG-bc} defined on $\Omega\times [0,\infty)$ such that
\begin{equation}
\label{Weak-reg}
\begin{split}
&\uu \in BC_{\mathrm{w}} ([0,\infty); \L_\sigma^2(\Omega)) 
\cap L^2(0,\infty; \H_{0,\sigma}^1(\Omega)),\\
&\phi \in BC_{\mathrm{w}}( [0,\infty); H^1(\Omega))\cap L_{\uloc}^2([0,\infty); H^2(\Omega)), \quad \Psi'(\phi)\in L^2_{\uloc}([0,\infty); L^2(\Omega)), \\
&\phi \in L^\infty(\Omega \times (0,\infty)) \text{ with }  |\phi(x,t)| <1 \ \text{a.e. in } \Omega \times (0,\infty),\\
&\mu \in L^2_{\uloc}([0,\infty); H^1(\Omega)), \quad \nabla \mu \in L^2(0,\infty; L^2(\Omega)),
\end{split}
\end{equation}
which satisfies
\begin{equation}
\label{uu-weak}
\int_0^\infty -\left( \rho \uu, \partial_t \ww \right) + \left( \div(\rho \uu\otimes \uu), \ww \right) 
- \left( (\uu \otimes \widetilde{\J}), \nabla \ww \right)
+ \left( \nu(\phi) D \uu, D \ww \right) \, \d t
= \int_0^\infty \left( \mu \nabla \phi, \ww \right) \, \d t
\end{equation}
for all $\ww \in C_{0,\sigma}^\infty(\Omega \times (0,\infty))$, and
\begin{equation}
\label{phi-weak}
\int_0^\infty - (\phi, \partial_t v) + (\uu \cdot \nabla \phi, v) \, \d t= \int_0^\infty -(m(\phi)\nabla \mu, \nabla v) \, \d t 
\end{equation}
for all $v \in C_0^\infty((0,\infty); C^1(\overline{\Omega}))$, where
\begin{equation}
\label{mu-def}
\mu= -\Delta \phi + \Psi'(\phi) \quad \text{a.e. in } \Omega \times (0,\infty)
\end{equation}
as well as $\uu(\cdot, 0)= \uu_0(\cdot)$ and $\phi(\cdot, 0)= \phi_0$ in $\Omega$.
In addition, the energy inequality
\begin{equation}
\label{energy-ineq}
E(\uu(t), \phi(t))+ \int_s^t \left\| \sqrt{\nu(\phi(\tau))} D \uu(\tau) \right\|_{L^2(\Omega)}^2 + \left\| \sqrt{m(\phi(\tau))}\nabla \mu(\tau) \right\|_{L^2(\Omega)}^2 \, \d \tau \leq  E(\uu(s),\phi(s))
\end{equation}
holds for all $t \in [s, \infty)$ and almost all $s \in [0,\infty)$ (including $s=0$).
\end{theorem}

\begin{remark} 
\label{rem-weak-sol}
The regularity properties \eqref{Weak-reg} and the weak formulation \eqref{phi-weak}, combined with a density argument, entail that $\partial_t \phi \in L^2(0, \infty; H^1(\Omega)')$. As a consequence, \eqref{phi-weak} is equivalent to 
$$
\l \partial_t \phi,v\r+ ( \uu\cdot \nabla \phi,v) 
+ (m(\phi) \nabla \mu, \nabla v )=0, 
\quad \forall \, v \in H^1(\Omega), \ \text{a.e. in } (0,\infty).
$$
Furthermore, arguing as in \cite{Abels2009} and \cite{GMT2019} (cf. also \cite{GiGrWu2018}), it follows from \eqref{mu-def} that 
$\phi \in L^2_{\uloc} ([0,\infty); W^{2,p}(\Omega) ) \cap L^4_{\uloc} ([0,\infty); H^2(\Omega))$ and $\Psi'(\phi) \in L^2_{\uloc} ([0,\infty); L^{p}(\Omega) )$, for any $p \in [2,\infty)$ if $d=2$ and $p=6$ if $d=3$. Finally, it holds that $\partial_\n \phi=0$ almost everywhere on $\partial \Omega \times (0,\infty)$.
\end{remark}

Throughout the manuscript we assume $m\equiv 1$ for simplicity.
Because of the energy dissipation \eqref{EE}, it is inherently expected that, as time $t$ goes to infinity, the velocity $\uu(t)$ tends to zero and $\phi(t)$ converges to an equilibrium of the Cahn-Hilliard equation/a critical point of the free energy $E_{\text{free}}$. Moreover, one predicts the solution to become regular for sufficiently large times. In the case of matched densities, i.e., $\rho_1=\rho_2$, such a result was shown in \cite{Abels2009}. However, in the case of non-matched densities $\rho_1\neq \rho_2$, this result for \eqref{AGG} was unknown so far. It is the purpose of this contribution to provide such a result. Our main result describes the global regularity features and the large time behavior of each weak solution given by Theorem \ref{WEAK-SOL} as follows.

\begin{theorem}[Regularity and asymptotic behavior of weak solutions]
\label{MAIN}
Let $\Omega$ be a bounded domain in $\mathbb{R}^d$, $d=2,3$, with boundary $\partial \Omega$ of class $C^3$ and $m\equiv 1$. Consider a global weak solution $(\uu,\phi)$ given by Theorem \ref{WEAK-SOL}. Then, the following results hold:
\begin{itemize}
\item[(i)] Global regularity of the concentration: for any $\tau >0$, we have
\begin{equation}
\label{REG-CONC}
\begin{split}
&\phi \in L^\infty( \tau,\infty; W^{2,p}(\Omega)), \quad \partial_t \phi \in L^2( \tau,\infty;H^1(\Omega)),\\
&\mu \in L^{\infty}(\tau,\infty; H^1(\Omega))\cap L_{\uloc}^2([\tau,\infty);H^3(\Omega)),
\quad F'(\phi) \in L^\infty(\tau,\infty; L^p(\Omega)),
\end{split}
\end{equation}
for any $2 \leq p <\infty$ if $d=2$ and $p=6$ if $d=3$. The equations \eqref{AGG}$_{3,4}$ are satisfied almost everywhere in $\Omega \times (0,\infty)$ and  the boundary condition $\partial_\n \mu=0$ holds almost everywhere on 
$\partial\Omega\times(0,\infty)$. 

\item[(ii)] Separation property:  there exist $T_{SP}>0$ and $\delta>0$ such that 
\begin{equation}
\label{separation}
|\phi(x,t)|\leq 1-\delta, \quad \forall \, (x,t) \in \overline{\Omega}\times [T_{SP},\infty).
\end{equation}

\item[(iii)] Large time regularity of the velocity: 
If $\Omega$ is a $C^4$ domain, then there exists $T_R>0$ (possibly larger than $T_{SP}$) such that 
\begin{equation}
\label{u-long-reg}
\uu \in L^\infty(T_R,\infty; \H^1_{0,\sigma}(\Omega)) \cap L^2(T_R,\infty; \H^2(\Omega))\cap H^1( T_R,\infty; \L^2_\sigma(\Omega)).
\end{equation}

\item[(iv)] Convergence to a stationary solution: $(\uu(t),\phi(t)) \rightarrow (\mathbf{0}, \phi_\infty)$ in $\L^2_\sigma(\Omega) \times W^{2-\eps, p}(\Omega)$ as $t \rightarrow \infty$, for any $\eps>0$, with any $2 \leq p <\infty$ if $d=2$ and $p=6$ if $d=3$, where $\phi_\infty \in W^{2,p}(\Omega)$ is a solution to the stationary Cahn-Hilliard equation
\begin{equation}
\label{Stat-CH}
\begin{split}
 -\Delta \phi_\infty + \Psi'(\phi_\infty)&=\mu_\infty \qquad \text{in } \Omega,\\
 \partial_{\n} \phi_\infty&=0 \qquad \quad \text{on } \partial\Omega,\\
   \frac{1}{|\Omega|}\int_\Omega \phi_\infty(x) \d x &= \overline{\phi_0},
    \end{split}
  \end{equation}
where $\mu_\infty \in \mathbb{R}$.
\end{itemize}
\end{theorem}

The structure of this contribution is as follows: in Section~\ref{sec:CH} we show a result on improved regularity for the solutions to the Cahn-Hilliard equation \eqref{AGG}$_{3-4}$ with given velocity $\uu \in L^2(0,\infty; \H^1_{0,\sigma}(\Omega))$, which implies \eqref{REG-CONC} and is the basis for the following analysis. Next, we study the $\omega$-limit set of a weak solution to \eqref{AGG} in Section~\ref{sec:LargeTimeBehaviour}. In particular, it is shown that $\uu(t)\to_{t\to \infty} \mathbf{0}$ in $\L^2(\Omega)$ and that the $\omega$-limit set of the concentration $\phi$ consists of stationary solutions of the Cahn-Hilliard equation (cf. \eqref{Stat-CH}). Moreover, we prove the strict separation property \eqref{separation}. In order to achieve  the regularity for large times \eqref{u-long-reg}, we first prove a result on weak-strong uniqueness in Section~\ref{weak-strong-U}. Combining this with the local existence result of strong solutions from \cite{G2021}, we obtain that every weak solutions becomes a strong solution for sufficiently large times in Section~\ref{Reg-3D}. The convergence to an equilibrium of the system is shown with the aid of the Lojasiewicz-Simon inequality in the same Section~\ref{Conv-eq}. Furthermore, we prove the global regularity in two space dimensions in Section~\ref{Reg-2D}. Finally, we study the limit $\theta\to 0$ in \eqref{Log} and show that weak solutions converge (for a suitable subsequence) to weak solutions to \eqref{AGG} in the case of a double obstacle potential in Section~\ref{obstacle}.
\medskip

\textbf{Notation}. 
In the sequel, we will use the following notation
$$
\rho_\ast=\min \lbrace \rho_1,\rho_2\rbrace,
\quad \rho^\ast=\max \lbrace \rho_1,\rho_2\rbrace,
\quad \nu_\ast =\min \lbrace \nu_1,\nu_2 \rbrace,
\quad \nu^\ast =\max \lbrace \nu_1,\nu_2 \rbrace.
$$
We denote 
$a\otimes b = (a_i b_j)_{i,j=1}^d$ for $a,b\in \R^d$. 
If $X$ is a Banach space and $X'$ is its dual, then 
\begin{equation*}
  \l f,g\r \equiv \l f,g\r_{X',X} = f(g), \qquad f\in X', g\in X,
\end{equation*}
denotes the duality product.

For a measurable set $M\subseteq \R^d$ and $1\leq q \leq \infty$, $L^q(M)$ denotes the Lebesgue space and $\| \cdot \|_{L^q(M)}$ its norm. We define 
\begin{equation*}
 (f,g)= \int_\Omega f(x)g(x) \, \d x \quad \text{for all } \ f\in L^q(\Omega), \ g \in L^{q'}(\Omega) \ \text{ with } \ \frac1q+\frac1{q'}=1.
\end{equation*}
If $q=2$, $(\cdot, \cdot)$ stands for the inner product in $L^2(M)$. 
For any $f \in L^1(M)$ with $|M|<\infty$, the total mass is defined as $\overline{f}=\frac{1}{|M|} \int_{M} f(x) \, \d x$.
The space $\L^q(M):= L^q(M; \mathbb{R}^d)$ consists of all $q$-integrable/essentially bounded vector-fields. For simplicity of notation, we denote the norm in $\L^q(M)$ by $\|\cdot\|_{L^q(M)}$ and the inner product in $\L^2(M)$ by $(\cdot, \cdot)$. Let $X$ be a Banach space, $L^q(M;X)$ denotes the set of all strongly measurable
$q$-integrable functions/essentially bounded functions with values in $X$. If $M=(a,b)$, we write for simplicity $L^q(a,b;X)$ and $L^q(a,b)$.
Furthermore, $f\in L^q_{\loc}([0,\infty);X)$ if and only if $f\in L^q(0,T;X)$ for every $T>0$ and $L^q_{\uloc}([0,\infty); X)$ denotes the \emph{uniformly
  local} variant of $L^q(0,\infty;X)$ consisting of all strongly measurable $f\colon
[0,\infty)\to X$ such that
\begin{equation*}
  \|f\|_{L^q_{\uloc}([0,\infty); X)}= \sup_{t\geq 0}\|f\|_{L^q(t,t+1;X)} <\infty.
\end{equation*}
For $T<\infty$, we set $L^q_{\uloc}([0,T); X) := L^q(0,T;X)$.

Let $\Omega \subset \R^d$ be an open bounded and connected domain with Lipschitz boundary. 
The set $W^{m,q}(\Omega)$, $m\in \mathbb{R}_+$, $1\leq q\leq \infty$, denotes the $L^q$-Sobolev space, and $\W^{m,q}(\Omega)= W^m_q(\Omega; \mathbb{R}^d)$ is the corresponding space for vector-fields. In both cases, the corresponding norms are denoted by $\| \cdot\|_{W^{m,q}(\Omega)}$. The space
$W^{m,q}_{0}(\Omega)$ is the closure of $C^\infty_0(\Omega)$ in $W^{m,q}(\Omega)$,
$W^{-m,q}(\Omega)= W^{m,q'}_{0}(\Omega)'$ (where $q'$ is such that $\frac 1q +\frac 1{q'}=1$).
As usual, $H^m(\Omega)= W^{m,2}(\Omega)$ and $\H^m(\Omega)= \W^{m,2}(\Omega)$.
Next, $\L^2_\sigma(\Omega)$ and $\H^1_{0,\sigma}(\Omega)$ are the closure of $C^\infty_{0,\sigma}(\Omega; \mathbb{R}^d)=\{\boldsymbol\varphi \in C_0^\infty(\Omega; \mathbb{R}^d ): \operatorname{div} \boldsymbol\varphi =0\}$ 
in $\L^2(\Omega)$ and $\H^1(\Omega)$, respectively. 


Let $I=[0,T]$ with $0<T< \infty$ or let $I=[0,\infty)$ if $T=\infty$ and $X$ be a Banach space. The set $BC(I;X)$ is the Banach space of all bounded and continuous $f\colon I\to X$ equipped with the supremum norm, and $BUC(I;X)$ is the subspace of all bounded and uniformly continuous functions . We define $BC_{\mathrm{w}}(I;X)$ as the topological vector space of all bounded and weakly continuous functions $f\colon I\to X$. We denote by $C^\infty_0(0,T;X)$ the vector space of all smooth functions $f\colon (0,T)\to X$ with $\supp
f \overset{c}{\subset} (0,T)$. Moreover, for $1\leq p <\infty$, $W^{1,p}(0,T;X)$ is the space of all $f\in L^p(0,T;X)$ with $\partial_t f \in L^p(0,T;X)$, where $\partial_t$ denotes the vector-valued
distributional derivative of $f$. The set $W^{1,p}_{\uloc}([0,\infty);X)$ is defined in the same way, replacing $L^p(0,T;X)$ by $L^p_{\uloc}([0,\infty);X)$. Lastly, we set $H^1(0,T;X)= W^{1,2}(0,T;X)$ as well as $H^1_{\uloc}([0,\infty);X) := W^{1,2}_{\uloc}([0,\infty);X)$.

\section{On the Cahn-Hilliard equation with divergence-free drift}\label{sec:CH}
\setcounter{equation}{0}

In this section we prove new regularity results for 
the Cahn-Hilliard equation with divergence-free drift
\begin{equation}
\begin{cases}
\label{CH}
\partial_t \phi+ \uu \cdot \nabla \phi = \Delta \mu \\
\mu= -\Delta \phi+ \Psi'(\phi)
\end{cases}
\quad \text{ in } \Omega\times (0,\infty),
\end{equation}
subject to the boundary and initial conditions
\begin{equation}
\label{bcic}
\begin{cases}
\partial_\n \phi=\partial_\n \mu=0 \quad &\text{ on } \partial \Omega \times (0,\infty),\\
\phi|_{t=0}=\phi_0\quad &\text{ in } \Omega.
\end{cases}
\end{equation}
We first report the following well-posedness result proved in \cite[Theorem 6]{Abels2009}. 

\begin{theorem}
\label{well-pos}
Let $\Omega$ be a bounded domain in $\mathbb{R}^d$, $d=2,3$ with $C^3$ boundary. Assume that $\phi_0\in H^1(\Omega)\cap L^{\infty}(\Omega)$ with $\| \phi_0\|_{L^\infty(\Omega)}\leq 1$ and $\left|\overline{\phi_0}\right|<1$, and  $\uu \in L^\infty(0,\infty; \L^2_\sigma(\Omega))\cap L^2(0,\infty;\H^1_{0,\sigma}(\Omega))$.
Then, there exists a unique global weak solution $\phi$ to 
\eqref{CH}-\eqref{bcic} such that:
\medskip

\begin{enumerate}
\item The weak solution satisfies
\begin{align*}
&\phi \in  BC( [0,\infty); H^1(\Omega))\cap L_{\uloc}^4([0,\infty);H^2(\Omega))\cap L_{\uloc}^2([0,\infty);W^{2,p}(\Omega)),\\ 
&\phi \in L^{\infty}(\Omega\times (0,\infty)) \text{ such that }  |\phi(x,t)|<1  \ \text{a.e. in }\Omega\times (0,\infty),\\
&\partial_t \phi \in L^2(0,\infty; H^1(\Omega)'),\\
& \mu \in L_{\uloc}^2([0,\infty);H^1(\Omega)), \quad F'(\phi)\in L_{\uloc}^2([0,\infty);L^p(\Omega)),
\end{align*}
for any $2\leq p <\infty$ if $d=2$ and $p=6$ if $d=3$.
\medskip

\item The weak solution solves \eqref{CH} in a variational sense as follows: 
\begin{align}
\label{e2}
\l \partial_t \phi,v\r+ ( \uu\cdot \nabla \phi,v) 
+ (\nabla \mu, \nabla v )=0, 
\quad \forall \, v \in H^1(\Omega), \ \text{a.e. in } (0,\infty),
\end{align}
where $\mu$ is given by
$
\mu=-\Delta \phi +\Psi'(\phi).
$
Moreover, $\partial_\n \phi=0$ almost everywhere 
on $\partial\Omega\times(0,\infty)$, and
$\phi(\cdot,0)=\phi_0$ in $\Omega$.
\medskip

\item The weak solution satisfies the energy equality
\begin{equation}
\label{EI}
E_{\mathrm{free}}(\phi(t))+ \int_\tau^t 
\| \nabla \mu(s)\|_{L^2(\Omega)}^2 \, \d s =  E_{\mathrm{free}}(\phi(\tau)) - \int_{\tau}^t (\uu \cdot \nabla \phi, \mu) \, \d s,
\end{equation}
for every $0\leq \tau< t \leq \infty$. 
\medskip

\item Let $Q= \Omega \times (0,\infty)$. The following estimates holds
\begin{align}
\label{E1}
&\| \phi\|_{L^\infty(0,\infty;H^1(\Omega))}^2 + \| \partial_t \phi\|_{L^2(0,\infty; H^1(\Omega)')}^2 + \| \nabla \mu\|_{L^2(Q)}^2 
\leq 
C \left( 1+ E_{\mathrm{free}}(\phi_0) +  \| \uu\|_{L^2(Q)}^2 \right),
\\
&\| \phi\|_{L_{\uloc}^2([0,\infty); W^{2,p}(\Omega)) }^2 + \| F'(\phi)\|_{L_{\uloc}^2([0,\infty); L^p(\Omega))}^2 
\leq 
C_p \left( 1+ E_{\mathrm{free}}(\phi_0) +  \| \uu\|_{L^2(Q)}^2 \right),
\\
&\| \phi\|_{L_{\uloc}^4([0,\infty); H^2(\Omega))}^4 
\leq 
C \left( 1+ E_{\mathrm{free}}(\phi_0) +  \| \uu\|_{L^2(Q)}^2 \right)^2,
\end{align}
for $2\leq p <\infty$ if $d=2$ and $p=6$ if $d=3$.
The constants $C$ and $C_p$ are independent of $\theta$, $\uu$ and $\phi_0$.
\end{enumerate}
\end{theorem}

\begin{remark}
\label{remark-wellposs} 
The following comments  concerning Theorem \ref{well-pos} are in order: 
\begin{itemize}
\item[(i)] A closer look at the proof of \cite[Theorem 6]{Abels2009} reveals that $\uu \in L^2(0,\infty; \H^1_{0,\sigma}(\Omega))$ is sufficient to show the desired claim in Theorem \ref{well-pos}.

\item[(ii)] The regularity $\phi \in L_{\uloc}^4(0,\infty;H^2(\Omega ))$ is not shown in \cite{Abels2009}, but the proof can be found in \cite{GiGrWu2018}.
\end{itemize}
\end{remark}

The propagation of regularity of the weak solutions and the existence of strong solutions have been first shown in \cite[Lemma 3]{Abels2009}. We report it here below for clarity of presentation.
\begin{theorem}
\label{A-reg}
Let the assumptions of Theorem \ref{well-pos} hold. Assume that $\kappa \equiv 1$ if $\phi_0 \in H^2(\Omega)$, $\mu_0=-\Delta \phi_0+\Psi'(\phi_0) \in H^1(\Omega)$ and 
$\partial_\n \phi_0=0$ on $\partial \Omega$, whereas $\kappa(t)= \left( \frac{t}{1+t}\right)^\frac12$ otherwise. 
\begin{itemize}
\item[1.] If $\partial_t \uu \in L^1_{\uloc}([0,\infty); \L^2(\Omega))$, then the weak solution to \eqref{CH}-\eqref{bcic} satisfies
\begin{align*}
& \kappa \partial_t \phi \in L^\infty(0,\infty; H^1(\Omega)')\cap L^2_{\uloc}([0,\infty); H^1(\Omega)),\\
& \kappa \phi \in L^\infty(0,\infty; W^{2,p}(\Omega)), \quad \kappa F'(\phi) \in L^\infty(0,\infty; L^p(\Omega)),\\
& \kappa \mu \in L^\infty(0,T; H^1(\Omega)),
\end{align*}
for any $2 \leq p <\infty$ if $d=2$ and $p=6$ if $d=3$.
\medskip

\item[2.] If $\uu \in B^\alpha_{\frac43 \, \infty, \uloc}([0,\infty); \H^s(\Omega) )\cap BC_{\mathrm{w}}([0,\infty); \L_\sigma^2(\Omega))$ for some $-\frac12 < s \leq 0$ and $\alpha \in (0,1)$, then the weak solution to \eqref{CH}-\eqref{bcic} fulfills
$$
\kappa \phi \in C^\alpha([0,\infty); H^1(\Omega)')\cap B^\alpha_{2 \, \infty, \uloc}([0,\infty);H^1(\Omega)).
$$ 
\end{itemize}
\end{theorem}
\noindent
Theorem \ref{A-reg} provides two regularity results for the weak solutions to the Cahn-Hilliard equation with divergence-free drift by requiring that either $\partial_t \uu \in L_{\uloc}^1([0,\infty); \L^2(\Omega))$ or $\uu \in B^\alpha_{\frac43 \infty, \uloc}([0,\infty);\H^s(\Omega))$ for some $-\frac12 < s \leq 0$ and $\alpha \in (0,1)$, in addition to $\uu \in L^\infty(0,\infty; \L^2_\sigma(\Omega))\cap L^2(0,\infty;\H^1_{0,\sigma}(\Omega))$. 
{Although the assumption in the second part of Theorem \ref{A-reg} can be proven for weak solutions of \eqref{AGG} in the case of matched densities (namely for the Model H) as in \cite{Abels2009}, it does not seem possible for the case with unmatched densities since the weak formulation only gives some control of $\partial_t (\rho \uu)|_{\H^1_{0,\sigma}(\Omega)}$. To overcome this issue,} we will now show that the condition $\uu \in L^2(0,\infty;\H^1_{0,\sigma}(\Omega))$ is sufficient to gain full regularity for the solutions to \eqref{CH}-\eqref{bcic}. {We expect that such result will be useful for other diffuse interface models with hydrodynamics.}

\begin{theorem}
\label{CH-strong}
Let $\Omega$ be a bounded domain in $\mathbb{R}^d$, $d=2,3$, with $C^3$ boundary and the initial condition $\phi_0 \in H^2(\Omega)$ be such that 
$\| \phi_0\|_{L^\infty(\Omega)}\leq 1$, $\left|\overline{\phi_0}\right|<1$, $\mu_0=-\Delta \phi_0+\Psi'(\phi_0) \in H^1(\Omega)$ and 
$\partial_\n \phi_0=0$ on $\partial \Omega$. Assume that $\uu \in L^2(0,\infty;\H^1_{0,\sigma}(\Omega))$.
Then, there exists a unique global (strong) solution to \eqref{CH}-\eqref{bcic} such that 
\begin{equation}
\label{REG}
\begin{split}
&\phi \in L^\infty(0,\infty;W^{2,p}(\Omega)), \quad \partial_t \phi \in L^2(0,\infty;H^1(\Omega)),\\
&\phi \in L^{\infty}(\Omega\times (0,\infty)) \text{ with } |\phi(x,t)|<1 
\ \text{a.e. in } \Omega\times (0,\infty),\\
&\mu \in L^{\infty}(0,\infty; H^1(\Omega))\cap L_{\uloc}^2([0,\infty);H^3(\Omega)),
\quad F'(\phi) \in L^\infty(0,\infty; L^p(\Omega)),
\end{split}
\end{equation}
for any $2\leq p <\infty$ if $d=2$ and $p=6$ if $d=3$. The strong solution satisfies \eqref{CH} almost everywhere in $\Omega \times (0,\infty)$ and  $\partial_\n \phi=\partial_\n \mu=0$ almost everywhere on 
$\partial\Omega\times(0,\infty)$. Moreover, there exists a positive constant $C$ depending only on $\Omega$, $\theta$, $\theta_0$,
 and $\overline{\phi_0}$ such that
\begin{equation}
\label{N-mu}
\begin{split}
\| \nabla \mu \|_{L^\infty(0,\infty;L^2(\Omega))} 
& \leq \left(  4 \left\| \nabla \left( -\Delta \phi_{0} +\Psi' \left(\phi_{0} \right) \right)\right\|_{L^2(\Omega)}^2+  4C \int_0^\infty \| \nabla \uu(s)\|_{L^2(\Omega)}^2 \, \d s\right)^\frac12 \\
& \quad \times
\mathrm{exp}\left( C \int_0^\infty \| \nabla \uu(s)\|_{L^2(\Omega)}^2 \, \d s \right)
\end{split}
\end{equation}
and
 \begin{equation}
\label{N-phit}
\begin{split}
\int_0^\infty \| \nabla \partial_t \phi(s)\|_{L^2(\Omega)}^2 \, \d s
&\leq 
6 \left(  \left\| \nabla \left(-\Delta \phi_{0} +\Psi'(\phi_{0}) \right) \right\|_{L^2(\Omega)}^2+ C \int_0^\infty \| \nabla \uu(s)\|_{L^2(\Omega)}^2 \, \d s\right) \\
&\quad \times \left( 1+\left( \int_0^\infty \| \nabla \uu(s)\|_{L^2(\Omega)}^2 \, \d s \right) \ \mathrm{exp}\left( 2C \int_0^\infty \| \nabla \uu(s)\|_{L^2(\Omega)}^2 \, \d s \right) \right). 
\end{split}
\end{equation}
In particular, the constant $C$ is bounded whenever $\theta$ is restricted to a bounded interval.

In addition, if $\uu \in  L^\infty(0, \infty;\L^2_\sigma)\cap L^2(0,\infty;\H^1_{0,\sigma}(\Omega))$, then $\partial_t \phi \in L^\infty(0,\infty; H^1(\Omega)')$.
\end{theorem} 

\begin{proof}
Let us assume first that $\uu \in {C_0^\infty}(0,T; \H^1_{0,\sigma}(\Omega) \cap \H^2(\Omega))$ and the initial condition $\phi_0$ is such that
\begin{equation}
\label{phi0-1}
\phi_0 \in H^3(\Omega) \ \text{ with } \ \| \phi_0\|_{L^\infty(\Omega)}<1 \ \text{ and } \
\partial_\n \phi_0=0 \  \text{ on } \partial \Omega.
\end{equation}
For any $\alpha \in (0,1)$, we consider the viscous Cahn-Hilliard system with divergence-free drift
\begin{equation}
\label{vcCH}
\begin{cases}
\partial_t \phi +\uu \cdot \nabla \phi = \Delta \mu\\
\mu= \alpha \partial_t \phi -\Delta \phi+F'(\phi)-\theta_0 \phi
\end{cases}
\quad \text{in }\ \Omega \times (0,T),
\end{equation}
which is equipped with the boundary and initial conditions
\begin{equation}
\label{vcCH-c}
\partial_\n \phi=\partial_\n \mu=0 \quad \text{on} \ \partial \Omega \times (0,T), \quad \phi|_{t=0}=\phi_{0} \quad \text{in }\ \Omega.
\end{equation}
Thanks to \cite[Theorem A.1]{G2021}, there exists a unique solution  such that
\begin{align}
\label{vCH-reg}
\begin{split}
&\phi \in L^\infty(0,T;H^3(\Omega))  \ \text{ with }\  \max_{(x,t)\in \overline{\Omega} \times [0,T]} |\phi(x,t)| <1, \\
& \partial_t \phi \in L^\infty(0,T;H^1(\Omega))\cap L^2(0,T;H^2(\Omega))\cap W^{1,2}(0,T;L^2(\Omega)),\\
&\mu \in L^\infty(0,T;H^2(\Omega))\cap W^{1,2}(0,T;L^2(\Omega)).
\end{split}
\end{align}
The pair $(\phi,\mu)$ satisfies \eqref{vcCH} almost everywhere in $\Omega \times (0,T)$, the boundary conditions $\partial_\n \phi=\partial_\n \mu=0$ almost everywhere on $\partial \Omega \times (0,T)$ and $\phi (\cdot, 0)= \phi_0(\cdot)$ in $\Omega$.

We now proceed with the conservation of mass and the first energy estimate. Integrating \eqref{vcCH}$_1$ over $\Omega$, exploiting the incompressibility and the no-slip boundary condition of the velocity field $\uu$, we infer that
\begin{equation}
\label{cons-mass}
\int_{\Omega} \phi(t) \, \d x= \int_{\Omega} \phi_{0} \, \d x, \quad \forall \, t \in [0,T].
\end{equation}
Multiplying \eqref{vcCH}$_1$ by $\mu$, integrating over $\Omega$ and exploiting the definition of $\mu$, we find for almost every $t\in (0,T)$ 
\begin{equation}
\label{CH-mu}
\ddt \left( \int_{\Omega} \frac12 |\nabla \phi|^2 + \Psi(\phi) \, \d x\right)
+ \int_{\Omega}  |\nabla \mu|^2 +\alpha  |\partial_t \phi|^2 \, \d x + \int_{\Omega} \uu \cdot \nabla \phi \, \mu \, \d x=0.
\end{equation}
Since $\uu(t)$ belongs to $\H^1_{0,\sigma}(\Omega)$, we have that $\int_{\Omega} \uu \cdot \nabla \phi \, \mu \, \d x=- \int_{\Omega} \uu \cdot \nabla \mu \, \phi \, \d x$. Then, thanks to the $L^\infty$ bound of $\phi$ in \eqref{vCH-reg}, we easily reach
\begin{equation}
\label{EE-CH}
\ddt \left( \int_{\Omega} \frac12 |\nabla \phi|^2 + \Psi(\phi) \, \d x\right)
+ \frac12 \int_{\Omega}  |\nabla \mu|^2 \, \d x+ \int_{\Omega}\alpha  |\partial_t \phi|^2 \, \d x \leq  \frac12 \| \uu\|_{L^2(\Omega)}^2.
\end{equation}
An integration in time on $[0,t]$, with $0<t\leq T$, yields
\begin{equation}
\label{EE-2}
\begin{split}
\sup_{t \in [0,T]} \|\nabla \phi(t)\|_{L^2(\Omega)}^2 
&+ \int_0^T \|\nabla \mu(s)\|_{L^2(\Omega)}^2 \, \d s
+2 \alpha  \int_0^T \|\partial_t \phi(s)\|_{L^2(\Omega)}^2 \, \d s \\
&\leq 
\theta_0 |\Omega|
+ \|\nabla \phi_0\|_{L^2(\Omega)}^2 
+ 2\int_{\Omega} \Psi(\phi_0) \, \d x + \int_0^T \| \uu(s)\|_{L^2(\Omega)}^2 \, \d s.
\end{split}
\end{equation}
By using \eqref{cons-mass}, we obtain
\begin{equation}
\label{EE-3}
\| \phi \|_{L^\infty(0,T;H^1(\Omega))}\leq C_0, \quad
\| \nabla \mu\|_{L^2(0,T;L^2(\Omega))}\leq C_0, \quad
\sqrt{\alpha}\| \partial_t \phi \|_{L^2(0,T;L^2(\Omega))}\leq C_0,
\end{equation}
where the constant $C_0$ depends only on $E_{\text{free}}(\phi_0)$, $|\overline{\phi_0}|$, $\theta_0$, $\Omega $ and $\| \uu \|_{L^2(0,T; L^2(\Omega))}$, but is independent of $\alpha$ and depends on $\theta$ only through $E_{\text{free}}(\phi_0)$.

Next, we derive some preliminary estimates which will play a crucial role for the subsequent part. We recall the well-known inequality (see, for instance, \cite{MZ})
\begin{equation}
\label{F'-L1}
\int_{\Omega} \left| F'(\phi) \right| \, \d x \leq C_1 \int_{\Omega} F'(\phi) \left( \phi-\overline{\phi_{0}} \right) \, \d x+ C_2,
\end{equation}
where the positive constant $C_1$ only depends only on  $\overline{\phi_{0}}$ and $C_2$ only depends on $\theta$ and $\overline{\phi_{0}}$. We also observe that $C_2$ can be chosen to depend only on $\overline{\phi_{0}}$ if we 
restrict $\theta$ to lie in a bounded interval. This can be seen if we consider \eqref{F'-L1} for $\theta=1$ and then multiply the equation by $\theta$.
Multiplying \eqref{vcCH}$_2$ by $\phi - \overline{\phi_{0}}$ (cf. \eqref{cons-mass}), we find
\begin{align*}
\int_{\Omega} |\nabla \phi|^2 \, \d x&+ 
\int_{\Omega} F'(\phi) \left( \phi -\overline{\phi_{0}} \right) \, \d x \\
&= - \alpha \int_{\Omega} \partial_t \phi \left(\phi-\overline{\phi_{0}}\right)\, \d x+ \int_{\Omega} (\mu-\overline{\mu}) \phi \, \d x + \theta_0 \int_{\Omega} \phi \left( \phi -\overline{\phi_{0}} \right) \, \d x.
\end{align*}
By the generalized Poincar\'{e} inequality and the $L^\infty$ bound of $\phi$, we reach
\begin{equation}
\label{F'-L1e}
 \int_{\Omega} \left| F'(\phi) \right| \, \d x  \leq 
C \left(1+ \| \nabla \mu\|_{L^2(\Omega)} + \alpha \| \partial_t \phi\|_{L^2(\Omega)}\right),
\end{equation}
where $C$ only depends on the Poincar\'{e} constant, $\theta_0$, $C_1$, $C_2$ and $\Omega$.
Then, since $\overline{\mu}= \overline{F'(\phi)}- \theta_0 \overline{\phi_{0}}$, we infer from \eqref{F'-L1} and \eqref{F'-L1e} that
$$
|\overline{\mu}|\leq C\left(1+ \| \nabla \mu\|_{L^2(\Omega)} + \alpha \| \partial_t \phi\|_{L^2(\Omega)}\right).
$$ 
As a consequence, we deduce that 
\begin{equation}
\label{mu-H1e}
\| \mu\|_{H^1(\Omega)}\leq 
C \left( 1+ \| \nabla \mu\|_{L^2(\Omega)}+ \alpha \| \partial_t \phi\|_{L^2(\Omega)} \right).
\end{equation}
Besides, multiplying \eqref{vcCH}$_2$ by $|F'(\phi)|^{p-2}F'(\phi)$, with $p\geq 2$, and integrating over $\Omega$, we find
\begin{align*}
\int_{\Omega} (p-1) |F'(\phi)|^{p-2}F''(\phi) |\nabla \phi|^2 \, \d x + \| F'(\phi)\|_{L^p(\Omega)}^p 
=  \int_{\Omega} \left( -\alpha \partial_t \phi + \mu + \theta_0 \phi \right) \,   |F'(\phi)|^{p-2}F'(\phi) \, \d x.
\end{align*}
Here these computations are justified since  $\phi$ is separated from the pure phases (cf. \eqref{vCH-reg}). Then, by the H\"{o}lder inequality, it follows that
\begin{equation}
\label{F'-p}
\| F'(\phi)\|_{L^p(\Omega)}\leq C\left( 1+ \alpha \| \partial_t \phi\|_{L^p(\Omega)} + \| \mu\|_{L^p(\Omega)} \right).
\end{equation}
Combining \eqref{F'-p} with \eqref{vcCH}$_2$, the elliptic regularity of the Neumann problem yields that
\begin{equation}
\label{phi2-p}
\| \phi\|_{W^{2,p}(\Omega)}\leq C\left( 1+ \alpha \| \partial_t \phi\|_{L^p(\Omega)} + \| \mu\|_{L^p(\Omega)} \right).
\end{equation}
In addition, we also find by comparison in \eqref{vcCH}$_1$ that
\begin{equation}
\label{phit}
\| \partial_t \phi\|_{H^1(\Omega)'}\leq C \left( \|\nabla \mu \|_{L^2(\Omega)}+ \| \uu \|_{L^2(\Omega)} \right).
\end{equation}
The positive constants $C$ in \eqref{mu-H1e}-\eqref{phit} may vary from line to line, but they only depend on $\overline{\phi_0}$, $\theta$, $\theta_0$, $\Omega$ and $C_1, C_2$. In particular, they are all independent of $\alpha$ and stay bounded for $\theta$
belonging to a bounded interval.

We now carry out the higher-order Sobolev energy estimates. Multiplying \eqref{vcCH}$_1$ by $\partial_t \mu$ and integrating over $\Omega$, we find
$$
\frac12 \ddt \| \nabla \mu\|_{L^2(\Omega)}^2 + \int_{\Omega} \partial_t \phi \, \partial_t \mu \,\d x + \int_{\Omega} \uu \cdot \nabla \phi \, \partial_t \mu \, \d x=0.
$$
By definition of $\mu$, we observe that 
$$
\int_{\Omega} \partial_t \phi \, \partial_t \mu \,\d x 
= \frac{\alpha}{2} \ddt \| \partial_t \phi\|_{L^2(\Omega)}^2 
+ \int_{\Omega} |\nabla \partial_t \phi|^2 \, \d x + \int_{\Omega} F''(\phi) |\partial_t \phi |^2 \, \d x- \theta_0 \int_{\Omega}  |\partial_t \phi |^2 \, \d x.
$$
Similarly, by exploiting the incompressibilty and the no-slip boundary condition of the velocity field, we notice that
\begin{align*}
\int_{\Omega} \uu \cdot \nabla \phi \, \partial_t \mu \, \d x
&= \int_{\Omega} \uu \cdot \nabla \phi \, \left( \alpha \partial^2_t \phi -\Delta \partial_t \phi + F''(\phi)\partial_t \phi -\theta_0 \partial_t  \phi\right) \, \d x\\
&= \ddt \left( \alpha \int_\Omega \uu \cdot \nabla \phi \, \partial_t \phi \, \d x\right)
- \alpha \int_\Omega \partial_t \uu \cdot \nabla \phi \, \partial_t \phi \, \d x
- \underbrace{\alpha \int_{\Omega} \uu \cdot \nabla \partial_t \phi \, \partial_t \phi \, \d x}_{=0} \\
&\quad + \int_{\Omega} \nabla \left( \uu \cdot \nabla \phi\right) \cdot \nabla \partial_t \phi \, \d x
- \underbrace{\int_{\partial \Omega} (\uu \cdot \nabla \phi) (\nabla \partial_t \phi \cdot \n) \, \d \sigma}_{=0}\\
&\quad + \int_{\Omega} \uu \cdot \left( F''(\phi) \nabla \phi \right) \partial_t \phi \, \d x - \theta_0 \int_{\Omega} \uu \cdot \nabla \phi \, \partial_t \phi \, \d x\\
&= \ddt \left( \alpha \int_\Omega \uu \cdot \nabla \phi \, \partial_t \phi \, \d x\right)
+ \alpha \int_\Omega \partial_t \uu \cdot \nabla \partial_t \phi \,  \phi \, \d x
+ \int_{\Omega} \left( \nabla \uu^T \nabla \phi \right) \cdot \nabla \partial_t \phi \, \d x
\\
&\quad
+\int_{\Omega} \left( \nabla^2 \phi \, \uu \right) \cdot \nabla \partial_t \phi \, \d x 
+ \int_{\Omega} \uu \cdot \nabla \left( F'(\phi) \right) \partial_t \phi \, \d x  - \theta_0 \int_{\Omega} \uu \cdot \nabla \phi \, \partial_t \phi \, \d x\\
&= \ddt \left( \alpha \int_\Omega \uu \cdot \nabla \phi \, \partial_t \phi \, \d x\right)
+ \alpha \int_\Omega \partial_t \uu \cdot \nabla \partial_t \phi \,  \phi \, \d x
 + \int_{\Omega} \left( \nabla \uu^T \nabla \phi \right) \cdot \nabla \partial_t \phi \, \d x
\\ 
&\quad +
\int_{\Omega} \left( \nabla^2 \phi \, \uu \right) \cdot \nabla \partial_t \phi \, \d x  -
\int_{\Omega} \uu \cdot \nabla \partial_t \phi  \, F'(\phi) \, \d x + \theta_0 \int_{\Omega} \uu \cdot \nabla \partial_t \phi \, \phi \, \d x.
\end{align*}
Then, we arrive at
\begin{equation}
\label{diff-eq}
\begin{split}
& \ddt \left( \frac12 \| \nabla \mu\|_{L^2(\Omega)}^2 +\frac{\alpha}{2} \| \partial_t \phi\|_{L^2(\Omega)}^2 + \alpha \int_\Omega \uu \cdot \nabla \phi \, \partial_t \phi \, \d x \right)+ \int_{\Omega} |\nabla \partial_t \phi|^2 \, \d x + \int_{\Omega} F''(\phi) |\partial_t \phi |^2 \, \d x \\
&\quad =\theta_0 \int_{\Omega}  |\partial_t \phi |^2 \, \d x 
- \alpha \int_{\Omega} \partial_t \uu \cdot \nabla \partial_t \phi \, \phi \, \d x 
- \int_{\Omega} \left( \nabla \uu^T \nabla \phi \right) \cdot \nabla \partial_t \phi \, \d x-
\int_{\Omega} \left( \nabla^2 \phi \, \uu \right) \cdot \nabla \partial_t \phi \, \d x\\ 
&\quad \quad +
\int_{\Omega} \uu \cdot \nabla \partial_t \phi  \, F'(\phi) \, \d x
- \theta_0 \int_{\Omega} \uu \cdot \nabla \partial_t \phi \, \phi \, \d x.
\end{split}
\end{equation}
In order to estimate the terms on the right-hand side of \eqref{diff-eq}, combining \eqref{F'-p} and \eqref{phi2-p} with \eqref{mu-H1e} through the Sobolev embedding theorem, we have
\begin{equation}
\label{phi2}
\| \phi\|_{W^{2,p}(\Omega)}+ \| F'(\phi)\|_{L^p(\Omega)}
\leq  \alpha C \| \partial_t \phi\|_{L^p(\Omega)} + C_p \left( 1+\| \nabla \mu \|_{L^2(\Omega)} \right), 
\end{equation}
where $2\leq p <\infty$ if $d=2$ and $p=6$ if $d=3$.
Here the positive constants $C$ and $C_p$ are independent of $\alpha$ and remains bounded for $\theta$ in a bounded interval. Recalling that $\partial_t \phi$ is mean-free, thanks to \eqref{phit} and the generalized Poincar\'{e} inequality, we obtain
\begin{align*}
\theta_0 \int_{\Omega}  |\partial_t \phi |^2 \, \d x 
&\leq C \| \nabla \partial_t \phi\|_{L^2(\Omega)} \| \partial_t \phi\|_{H^1(\Omega)'}\\
&\leq \frac{1}{12} \| \nabla \partial_t \phi\|_{L^2(\Omega)}^2 + C \left( \| \uu\|_{L^2(\Omega)}^2+ \| \nabla \mu\|_{L^2(\Omega)}^2 \right).
\end{align*}
By using the $L^\infty$ bound of $\phi$ in \eqref{vCH-reg}, we also get
\begin{align*}
\left| \alpha \int_{\Omega} \partial_t\uu \cdot \nabla \partial_t \phi \, \phi \, \d x  \right| 
&\leq \alpha \| \partial_t \uu\|_{L^2(\Omega)} \| \nabla \partial_t \phi\|_{L^2(\Omega)} \| \phi\|_{L^\infty(\Omega)}\\
&\leq \frac{1}{12} \| \nabla \partial_t \phi\|_{L^2(\Omega)}^2 + \alpha^2 C \| \partial_t \uu\|_{L^2(\Omega)}^2.
\end{align*}
Exploiting \eqref{phi2}, the Gagliardo-Nirenberg interpolation inequality and the Sobolev embedding theorem, we infer that
\begin{align*}
\left| \int_{\Omega} \left( \nabla \uu^T \nabla \phi \right) \cdot \nabla \partial_t \phi \, \d x \right|
&\leq \| \nabla \uu\|_{L^2(\Omega)} \| \nabla \phi\|_{L^\infty(\Omega)} \| \nabla \partial_t \phi\|_{L^2(\Omega)}\\
& \leq \frac{1}{24} \| \nabla \partial_t \phi\|_{L^2(\Omega)}^2 +
C \| \nabla \uu\|_{L^2(\Omega)}^2 \| \phi \|_{W^{2,4}(\Omega)}^2\\
& \leq \frac{1}{24} \| \nabla \partial_t \phi\|_{L^2(\Omega)}^2 +
C \| \nabla \uu\|_{L^2(\Omega)}^2 \left(1+ \alpha^2 \| \partial_t \phi\|_{L^4(\Omega)}^2 + \| \nabla \mu\|_{L^2(\Omega)}^2 \right) \\
& \leq \frac{1}{24} \| \nabla \partial_t \phi\|_{L^2(\Omega)}^2 +  \alpha^2 C \| \nabla \uu\|_{L^2(\Omega)}^2  \| \partial_t \phi\|_{L^2(\Omega)}^\frac12 \| \nabla \partial_t \phi\|_{L^2(\Omega)}^\frac32\\
&\quad + 
C \| \nabla \uu\|_{L^2(\Omega)}^2 \left( 1+ \| \nabla \mu \|_{L^2(\Omega)}^2 \right)\\
&\leq \frac{1}{12} \| \nabla \partial_t \phi\|_{L^2(\Omega)}^2 +  \alpha^8 C \| \nabla \uu\|_{L^2(\Omega)}^8  \| \partial_t \phi\|_{L^2(\Omega)}^2\\
&\quad + 
C \| \nabla \uu\|_{L^2(\Omega)}^2 \left( 1+ \| \nabla \mu \|_{L^2(\Omega)}^2 \right)
\end{align*}
and 
\begin{align*}
\left|
\int_{\Omega} \left( \nabla^2 \phi \, \uu \right) \cdot \nabla \partial_t \phi \, \d x \right|
&\leq \| \phi\|_{W^{2,3}(\Omega)} \| \uu\|_{L^6(\Omega)} \| \nabla \partial_t \phi\|_{L^2(\Omega)}\\
& \leq  C\left( 1+ \alpha \| \partial_t \phi\|_{L^3(\Omega)}+ \| \nabla \mu \|_{L^2(\Omega)} \right) \| \nabla \uu\|_{L^2(\Omega)} \| \nabla \partial_t \phi\|_{L^2(\Omega)}\\
&\leq \frac{1}{24} \| \nabla \partial_t \phi\|_{L^2(\Omega)}^2 +  \alpha^2 C \| \partial_t \phi\|_{L^2(\Omega)} \| \nabla \partial_t \phi\|_{L^2(\Omega)} \| \nabla \uu\|_{L^2(\Omega)}^2\\
&\quad +
C \| \nabla \uu\|_{L^2(\Omega)}^2 \left( 1+ \| \nabla \mu \|_{L^2(\Omega)}^2 \right)\\
&\leq \frac{1}{12} \| \nabla \partial_t \phi\|_{L^2(\Omega)}^2 +  \alpha^4 C \| \partial_t \phi\|_{L^2(\Omega)}^2 \| \nabla \uu\|_{L^2(\Omega)}^4+
C \| \nabla \uu\|_{L^2(\Omega)}^2 \left( 1+ \| \nabla \mu \|_{L^2(\Omega)}^2 \right),
\end{align*}
as well as
\begin{align*}
\left| \int_{\Omega} \uu \cdot \nabla \partial_t \phi  \, F'(\phi) \, \d x \right| 
&\leq \| \uu\|_{L^6(\Omega)} \| \nabla \partial_t \phi\|_{L^2(\Omega)} \| F'(\phi)\|_{L^3(\Omega)}\\
&\leq \frac{1}{24} \| \nabla \partial_t \phi\|_{L^2(\Omega)}^2 + C \| \nabla \uu\|_{L^2(\Omega)}^2 \left( 1+ \alpha^2 \| \partial_t \phi\|_{L^3(\Omega)}^2 + \| \nabla \mu \|_{L^2(\Omega)}^2 \right)\\
&\leq \frac{1}{24} \| \nabla \partial_t \phi\|_{L^2(\Omega)}^2 +
 \alpha^2 C \| \partial_t \phi\|_{L^2(\Omega)}  
 \| \nabla \partial_t \phi\|_{L^2(\Omega)}  \| \nabla \uu\|_{L^2(\Omega)}^2\\
&\quad + C \| \nabla \uu\|_{L^2(\Omega)}^2 \left( 1+ \| \nabla \mu \|_{L^2(\Omega)}^2 \right)\\
&\leq \frac{1}{12} \| \nabla \partial_t \phi\|_{L^2(\Omega)}^2 +  \alpha^4 C \| \partial_t \phi\|_{L^2(\Omega)}^2 \| \nabla \uu\|_{L^2(\Omega)}^4+
C \| \nabla \uu\|_{L^2(\Omega)}^2 \left( 1+ \| \nabla \mu \|_{L^2(\Omega)}^2 \right).
\end{align*}
By using \eqref{vCH-reg}, we find that
\begin{align*}
\left| \theta_0 \int_{\Omega} \uu \cdot \nabla \partial_t \phi \, \phi \, \d x\right|
&\leq C \| \uu\|_{L^2(\Omega)} \| \nabla \partial_t \phi\|_{L^2(\Omega)} \| \phi\|_{L^\infty(\Omega)} \\
&\leq \frac{1}{12} \| \nabla \partial_t \phi\|_{L^2(\Omega)}^2 +
C \| \uu\|_{L^2(\Omega)}^2.
\end{align*}
The positive constants $C$ in all the above estimates depend on the parameters of the system, such as $\theta_0$, $\theta$, $\Omega$ and $\overline{\phi_0}$, but are independent of $\alpha$, $\uu$, $T$ and the norms of the initial condition $\phi_0$.
They also are uniformly bounded for $\theta$ from a bounded interval. Therefore, collecting the above inequalities, we end up with the differential inequality
 \begin{equation}
 \label{DI-alpha}
\begin{split}
 &\ddt \left( \frac12 \| \nabla \mu\|_{L^2(\Omega)}^2 +\frac{\alpha}{2} \| \partial_t \phi\|_{L^2(\Omega)}^2 + \alpha \int_\Omega \uu \cdot \nabla \phi \, \partial_t \phi \, \d x \right) + \frac12 \int_{\Omega} |\nabla \partial_t \phi|^2 \, \d x \\
&\leq C \left( \| \nabla \uu\|_{L^2(\Omega)}^2 + \alpha^3  \| \nabla \uu\|_{L^2(\Omega)}^4 + \alpha^7 \| \nabla \uu\|_{L^2(\Omega)}^8 \right) \left(  \frac12 \| \nabla \mu\|_{L^2(\Omega)}^2 +\frac{\alpha}{4} \| \partial_t \phi\|_{L^2(\Omega)}^2 \right) \\
&\quad + \alpha C \| \partial_t \uu\|_{L^2(\Omega)}^2 + C  \| \nabla \uu\|_{L^2(\Omega)}^2.
\end{split}
\end{equation}
In light of \eqref{EE-3}, we observe that 
\begin{equation}
\label{alpha-term}
\begin{split}
\left| \alpha \int_{\Omega} \uu \cdot \nabla \phi \partial_t \phi \, \d x \right|
&\leq \alpha \| \uu\|_{L^\infty(\Omega)} \| \nabla \phi\|_{L^2(\Omega)}
\| \partial_t \phi\|_{L^2(\Omega)} \\
&\leq \frac{\alpha}{4} \|\partial_t \phi\|_{L^2(\Omega)}^2 +
\alpha C_0 \| \uu\|_{L^\infty(\Omega)}^2.
\end{split}
\end{equation}
Owing to this, we rewrite \eqref{DI-alpha} as follows
\begin{equation}
 \label{DI-alpha2}
\begin{split}
 &\ddt \left( \frac12 \| \nabla \mu\|_{L^2(\Omega)}^2 +\frac{\alpha}{2} \| \partial_t \phi\|_{L^2(\Omega)}^2 + \alpha \int_\Omega \uu \cdot \nabla \phi \, \partial_t \phi \, \d x \right) + \frac12 \int_{\Omega} |\nabla \partial_t \phi|^2 \, \d x \\
&\leq \mathcal{F}_1 \left(  \frac12 \| \nabla \mu\|_{L^2(\Omega)}^2 +\frac{\alpha}{2} \| \partial_t \phi\|_{L^2(\Omega)}^2+ \alpha \int_\Omega \uu \cdot \nabla \phi \, \partial_t \phi \, \d x  \right) + \mathcal{F}_2,
\end{split}
\end{equation}
where 
\begin{equation}
\label{F-def}
\begin{split}
&\mathcal{F}_1 =C \left( \| \nabla \uu \|_{L^2(\Omega)}^2 + \alpha^3  \| \nabla \uu\|_{L^2(\Omega)}^4 + \alpha^7 \| \nabla \uu\|_{L^2(\Omega)}^8 \right), 
\\
&\mathcal{F}_2= C \left( \alpha \| \partial_t \uu\|_{L^2(\Omega)}^2 +   \| \nabla \uu\|_{L^2(\Omega)}^2 \right) + \alpha C_0 \| \uu\|_{L^\infty(\Omega)}^2 \mathcal{F}_1. 
\end{split}
\end{equation}
Since $\uu \in {C_0^\infty}(0,T; \H_{0,\sigma}^1(\Omega)\cap \H^2(\Omega))$ by assumption, it is easily seen that $\mathcal{F}_1$, $\mathcal{F}_2 \in L^1(0,T)$. 
Thanks to the Gronwall lemma, we deduce that
\begin{equation*}
\begin{split}
&\sup_{t\in [0,T]} \left(  \frac12 \| \nabla \mu\|_{L^2(\Omega)}^2 +\frac{\alpha}{2} \| \partial_t \phi\|_{L^2(\Omega)}^2 + \alpha \int_\Omega \uu \cdot \nabla \phi \, \partial_t \phi \, \d x  \right)\\
&\leq \left( \frac12\| \nabla \mu(0)\|_{L^2(\Omega)}^2 + \frac{\alpha}{2} \|\partial_t \phi(0) \|_{L^2(\Omega)}^2 + \alpha 
\underbrace{\int_\Omega \uu(0)\cdot \nabla \phi(0) \,\partial_t \phi(0) \, \d x}_{=0} +  \int_0^T \mathcal{F}_2(s) \, \d s\right) \\
&\quad \times \mathrm{exp}\left( \int_0^T \mathcal{F}_1(s) \, \d s\right).
\end{split}
\end{equation*}
By exploiting \eqref{alpha-term}, we then arrive at
\begin{equation*}
\begin{split}
&\sup_{t\in [0,T]} \left(  \| \nabla \mu\|_{L^2(\Omega)}^2 +\alpha \| \partial_t \phi\|_{L^2(\Omega)}^2  \right)\\
&\leq \left( 2\| \nabla \mu(0)\|_{L^2(\Omega)}^2 + 2 \alpha \|\partial_t \phi(0) \|_{L^2(\Omega)}^2 +  4 \int_0^T \mathcal{F}_2(s) \, \d s\right) \mathrm{exp}\left( \int_0^T \mathcal{F}_1(s) \, \d s\right)  + 4 \alpha C_0 \| \uu\|_{L^\infty(\Omega)}^2.
\end{split}
\end{equation*}
We observe that 
$\partial_t \phi \in BC([0,T];H^1(\Omega))$ and $ \mu  \in BC([0,T];H^1(\Omega))$ due to \eqref{vCH-reg}.
By comparison in \eqref{vcCH}$_2$, it follows that $-\Delta \phi +\Psi'(\phi) \in BC([0,T];H^1(\Omega))$. Now, multiplying \eqref{vcCH}$_2$ by $\partial_t \phi$ and integrating over $\Omega$, we have
$$
\alpha \| \partial_t \phi\|_{L^2(\Omega)}^2+ \left( -\Delta \phi +\Psi'(\phi ), \partial_t \phi \right) = (\mu, \partial_t \phi).
$$
By using \eqref{vcCH}$_1$, we notice that
$$
\alpha \| \partial_t \phi\|_{L^2(\Omega)}^2
+ \left( -\Delta \phi +\Psi'(\phi), \Delta \mu - \uu \cdot \nabla \phi \right) 
= \left( \mu, \Delta \mu -\uu \cdot \nabla \phi \right).
$$
Integrating by parts and exploiting the boundary conditions of $\mu$ and $\uu$,  we deduce thatt
\begin{equation}
\label{ID-eq}
\alpha \| \partial_t \phi\|_{L^2(\Omega)}^2
+\| \nabla \mu\|_{L^2(\Omega)}^2
= \left( \nabla (-\Delta \phi +\Psi'(\phi)), \nabla \mu - \phi \, \uu \right)+
\left( \nabla \mu, \phi \, \uu \right).
\end{equation}
By continuity of each term in \eqref{ID-eq} and recalling that $\uu \in {C_0^\infty}(0,T; \H_{0,\sigma}^1(\Omega))$, we infer that
\begin{align*}
\alpha \| \partial_t \phi(0)\|_{L^2(\Omega)}^2
+\| \nabla \mu(0) \|_{L^2(\Omega)}^2
= \left( \nabla \left( -\Delta \phi_{0} +\Psi'(\phi_{0}) \right), \nabla \mu(0) \right),
\end{align*}
which, in turn, entails that 
\begin{equation}
\label{ID-bound}
\alpha \| \partial_t \phi(0)\|_{L^2(\Omega)}^2
+ \| \nabla \mu(0) \|_{L^2(\Omega)}^2
\leq 2 \left\| \nabla \left( -\Delta \phi_{0} +\Psi'(\phi_{0}) \right) \right\|_{L^2(\Omega)}^2.
\end{equation}
As a consequence, we find 
\begin{equation}
\label{E1-alpha}
\begin{split}
&\sup_{t\in [0,T]} \left( \| \nabla \mu(t)\|_{L^2(\Omega)}^2 
+ \alpha \|\partial_t \phi(t) \|_{L^2(\Omega)}^2 \right)\\
& \leq \left(  4 \left\| \nabla \left( -\Delta \phi_{0} +\Psi'(\phi_{0}) \right) \right\|_{L^2(\Omega)}^2
+ 4 \int_0^T \mathcal{F}_2(s) \, \d s\right) \mathrm{exp}\left( \int_0^T \mathcal{F}_1(s) \, \d s\right)  + 4 \alpha C_0 \| \uu\|_{L^\infty(\Omega)}^2.
\end{split}
\end{equation}
Besides, integrating \eqref{DI-alpha2} on $[0,T]$, and exploiting \eqref{alpha-term}, \eqref{ID-bound} and \eqref{E1-alpha}, we have
\begin{equation}
\label{E2-alpha}
\begin{split}
&\int_0^T \| \nabla \partial_t \phi(s)\|_{L^2(\Omega)}^2 \, \d s\\
&\leq  2 \left\| \nabla \left( -\Delta \phi_{0} +\Psi'(\phi_{0}) \right) \right\|_{L^2(\Omega)}^2 \\
&\quad +6 \left( \left\| \nabla \left(-\Delta \phi_{0} +\Psi'(\phi_{0}) \right) \right\|_{L^2(\Omega)}^2
+ \int_0^T \mathcal{F}_2(s) \, \d s\right) \left( \int_0^T \mathcal{F}_1(s)\, \d s \right) \mathrm{exp}\left( \int_0^T \mathcal{F}_1(s) \, \d s\right)  
\\
& \quad + 8 \alpha C_0 \| \uu\|_{L^\infty(\Omega)}^2  \int_0^T \mathcal{F}_1(s)\, \d s + 2\int_0^T \mathcal{F}_2 \, \d s.
\end{split}
\end{equation}
Since $\alpha \in (0,1]$, $\uu \in {C_0^\infty}(0,T; \H_{0,\sigma}^1(\Omega)\cap \H^2(\Omega))$ and $-\Delta \phi_0 + \Psi'(\phi_0) \in H^1(\Omega)$, we thus conclude that
\begin{equation}
\label{EST1}
\| \nabla \mu\|_{L^\infty(0,T; L^2(\Omega))} \leq K_0, \quad 
\sqrt{\alpha} \| \partial_t \phi\|_{L^\infty(0,T;L^2(\Omega))}\leq K_0, \quad
\| \nabla \partial_t \phi\|_{L^2(0,T;L^2(\Omega))}\leq K_0.
\end{equation}
It easily follows from \eqref{mu-H1e} and  \eqref{phi2} that
\begin{equation}
\label{EST2}
\|  \mu\|_{L^\infty(0,T; H^1(\Omega))} \leq K_1, 
\quad 
\| \phi\|_{L^\infty(0,T; H^2(\Omega))}\leq K_1, 
\end{equation}
and 
\begin{equation}
\label{EST3}
\| \phi\|_{L^2(0,T;W^{2,p}(\Omega))}+ \| F'(\phi)\|_{L^2(0,T;L^p(\Omega))} \leq K_p,
\end{equation}
where $2\leq p <\infty$ if $d=2$ and $p=6$ if $d=3$. The constants $K_0$, $K_1$ and $K_p$ (for $p$ as mentioned above) are independent of $\alpha$ and remain bounded for $\theta$ from a bounded interval.

Let us now consider a sequence of real numbers $\alpha_n \in (0,1]$ such that $\alpha_n \rightarrow 0$ as $n \rightarrow \infty$. Thanks to the above analysis, there exists a sequence of pairs $(\phi_{\alpha_n}, \mu_{\alpha_n})$ such that
\begin{equation}
\label{vcCH-alpha}
\begin{cases}
\partial_t \phi_{\alpha_n} +\uu \cdot \nabla \phi_{\alpha_n} = \Delta \mu_{\alpha_n}\\
\mu_{\alpha_n}= \alpha_n \partial_t \phi_{\alpha_n} -\Delta \phi_{\alpha_n}+F'(\phi_{\alpha_n})-\theta_0 \phi_{\alpha_n}
\end{cases}
\quad \text{a.e. in }\ \Omega \times (0,T),
\end{equation}
and  $\partial_\n \phi_{\alpha_n}=\partial_\n \mu_{\alpha_n}=0$ almost everywhere on $\partial \Omega \times (0,T)$, $ \phi_{\alpha_n}(\cdot,0)=\phi_{0}$ in $\Omega$. Each pair $(\phi_{\alpha_n}, \mu_{\alpha_n})$ satisfies \eqref{EST1}-\eqref{EST3} replacing $(\phi, \mu)$ with $(\phi_{\alpha_n}, \mu_{\alpha_n})$. Thus, there exists a subsequence (still denoted in the same way) $(\phi_{\alpha_n}, \mu_{\alpha_n})$ such that
\begin{equation}
\label{alpha-limit}
\begin{split}
\begin{aligned}
&\phi_{\alpha_n} \rightharpoonup \phi \quad  &&\text{weakly in } L^2(0,T; W^{2,p}(\Omega)),\\
&\phi_{\alpha_n} \rightharpoonup \phi \quad  &&\text{weakly in } 
W^{1,2}(0,T;H^1(\Omega)),\\
&\mu_{\alpha_n}  \rightharpoonup \mu \quad  &&\text{weak-star in } L^\infty(0,T;H^1(\Omega)),
\end{aligned}
\end{split}
\end{equation}
for any $2\leq p <\infty$ if $d=2$ and $p=6$ if $d=3$,
and, by the Aubin-Lions theorem, 
\begin{equation}
\label{sl-SS-alpha}
\begin{split}
\begin{aligned}
&\phi_{\alpha_n} \rightarrow \phi \quad  &&\text{strongly in } BC([0,T];W^{1,q}(\Omega)), 
\end{aligned}
\end{split}
\end{equation}
for all $2\leq q <\infty$ if $d=2$ and $2\leq q < 6$ if $d=3$. In order to pass to the limit in the logarithmic function $F'$, we recall from \eqref{vCH-reg} that 
$$
\phi_{\alpha_n} \in L^\infty(\Omega\times (0,T)) \ \text{ such that  } \ |\phi_{\alpha_n} (x,t)|<1 \ \text{ a.e. in } \  \Omega\times(0,T).
$$
Thanks to \eqref{sl-SS-alpha}, we infer that $\phi_{\alpha_n} \rightarrow \phi$ almost everywhere in $\Omega \times (0,T)$. As a consequence, we have that
$$
\phi \in L^\infty(\Omega\times (0,T)) \ \text{ with }\  |\phi(x,t)| \leq 1 \ \text{ a.e. in } \  \Omega\times(0,T).
$$
Then, $F'(\phi_{\alpha_n}) \rightarrow \widetilde{F'}(\phi)$ almost everywhere in $\Omega \times (0,T)$, where $\widetilde{F'}(s)=F'(s)$ if $s \in (-1,1)$ and $\widetilde{F'}(\pm 1)=\pm \infty$.
By the Fatou lemma and \eqref{EST3}, $\int_{\Omega \times (0,T)} |\widetilde{F'}(\phi)|^2 \, \d x \d s\leq K_2^2$, which implies that
$\widetilde{F'}(\phi)\in L^2(\Omega \times (0,T))$. This entails that 
$$
\phi \in L^\infty(\Omega\times (0,T)) \ \text{ such that } \ |\phi(x,t)| < 1 \ \text{ a.e. in } \  \Omega\times(0,T),
$$
and $\widetilde{F'}(\phi)= F'(\phi)$ almost everywhere in $\Omega \times (0,T)$.
Owing to this, and by \eqref{EST3}, we conclude that 
$$
F'(\phi_{\alpha_n}) \rightharpoonup F'(\phi) \quad  \text{weakly in } 
L^2(0,T;L^p(\Omega)),
$$
for any $p$ as above.
Thus, letting $n \rightarrow \infty$ in \eqref{vcCH-alpha}, we obtain that $(\phi, \mu)$ solves the Cahn-Hilliard system with divergence-free drift
\begin{equation}
\label{cCH}
\begin{cases}
\partial_t \phi +\uu \cdot \nabla \phi = \Delta \mu\\
\mu= -\Delta \phi+F'(\phi)-\theta_0 \phi
\end{cases}
\quad \text{a.e. in }\ \Omega \times (0,T).
\end{equation}
In addition, $\partial_\n \phi=\partial_\n \mu=0$ almost everywhere on $\partial \Omega \times (0,T)$ and $\phi(\cdot, 0)=\phi_{0}$ in $\Omega$. 
Furthermore, by the lower semi-continuity of the norm with respect to the weak convergence, passing to the limit in \eqref{E1-alpha} and \eqref{E2-alpha}, we have
\begin{equation}
\label{E1-f}
\begin{split}
\| \nabla \mu \|_{L^\infty(0,T;L^2(\Omega))} 
& \leq \left(  4 \left\| \nabla \left( -\Delta \phi_{0} +\Psi'(\phi_{0}) \right) \right\|_{L^2(\Omega)}^2+  4C \int_0^T \| \nabla \uu(s)\|_{L^2(\Omega)}^2 \, \d s\right)^\frac12 \\
&\quad \times \text{exp}\left( C \int_0^T \| \nabla \uu(s)\|_{L^2(\Omega)}^2 \, \d s \right)
\end{split}
\end{equation}
and 
 \begin{equation}
\label{E2-f}
\begin{split}
\int_0^T \| \nabla \partial_t \phi(s)\|_{L^2(\Omega)}^2 \, \d s
&\leq  6 \left( \left\| \nabla \left( -\Delta \phi_{0} +\Psi'(\phi_{0}) \right) \right\|_{L^2(\Omega)}^2
+ C\int_0^T \| \nabla \uu(s)\|_{L^2(\Omega)}^2 \, \d s\right) \\
&\quad \times \left( 1+ \left( C \int_0^T \| \nabla \uu(s)\|_{L^2(\Omega)}^2 \, \d s \right)  \mathrm{exp}\left( C \int_0^T \| \nabla \uu(s)\|_{L^2(\Omega)}^2 \, \d s\right) \right), 
\end{split}
\end{equation}
where the positive constant $C$ depends on $\theta_0$, $\theta$, $\Omega$ and $\overline{\phi_0}$, but is independent of  $\uu$, $T$ and the norms of the initial condition $\phi_0$. In particular, $C$ is bounded if $\theta$ belongs to a bounded interval.
Then, repeating the argument to obtain \eqref{F'-L1e} and \eqref{mu-H1e} without the ($\alpha$) viscous term, we easily recover that 
\begin{equation}
\label{mu-H1}
\| \mu\|_{L^\infty(0,T;H^1(\Omega))}\leq R_0.
\end{equation} 
Since $\partial_t \phi$ is mean-free, we also have
\begin{equation}
\label{phit-H1}
\| \partial_t \phi\|_{L^2(0,T;H^1(\Omega))}\leq R_1.
\end{equation}
By recalling the inequality proven in \cite[Lemma 2]{Abels2009} for \eqref{cCH}$_2$
\begin{equation}
\label{phi2p}
\| \phi\|_{W^{2,p}(\Omega)}+ \| F'(\phi)\|_{L^p(\Omega)}
\leq C_p \left( 1+\| \nabla \mu \|_{L^2(\Omega)} \right), 
\end{equation}
where $2\leq p <\infty$ if $d=2$ and $p=6$ if $d=3$, we obtain that
\begin{equation}
\label{phi-W2p}
\| \phi\|_{L^\infty(0,T; W^{2,p}(\Omega))} 
+ \| F'(\phi)\|_{L^\infty(0,T; L^p(\Omega))} \leq R_2(p).
\end{equation}
Finally, since 
$$
\| \uu \cdot \nabla \phi\|_{L^2(0,T;H^1(\Omega))}
\leq \| \uu \|_{L^2(0,T;H^1(\Omega))} 
\| \nabla \phi\|_{L^\infty(0,T; L^\infty(\Omega))}
+ \| \uu\|_{L^2(0,T; L^6(\Omega))} 
\| \phi\|_{L^\infty(0,T;W^{2,3}(\Omega))},
$$
we easily obtain by the  elliptic regularity theory applied to \eqref{cCH}$_1$ that
\begin{equation}
\label{muH3}
\| \mu\|_{L^2(0,T;H^3(\Omega))}\leq R_4(T).
\end{equation}
Here the positive constants $R_1$, $R_2$, $R_3$ and $R_4$ depend on $\theta_0$, $\theta$, $\Omega$, $\overline{\phi_0}$, $\left\| \nabla \left(-\Delta \phi_{0} +\Psi'(\phi_{0}) \right) \right\|_{L^2(\Omega)}$ and $\| \uu\|_{L^2(0,T;\H_{0,\sigma}^1(\Omega))}$. They all remain bounded
for $\theta$ from a bounded interval. Additionally, $R_4$ is bounded for $T$ bounded.

We are left to show that the existence of the unique solution $\phi$ to \eqref{CH}-\eqref{bcic} satisfying \eqref{REG} holds for any divergence-free velocity $\uu \in L^2(0,\infty;\H_{0,\sigma}^1(\Omega))$ and for any initial condition $\phi_0 \in H^2(\Omega)$ such that 
$\| \phi_0\|_{L^\infty(\Omega)}\leq 1$, $\left|\overline{\phi_0}\right|<1$, $\mu_0=-\Delta \phi_0+\Psi'(\phi_0) \in H^1(\Omega)$ and 
$\partial_\n \phi_0=0$ on $\partial \Omega$. For this purpose, since ${C_0^\infty}(0,T; \H_{0,\sigma}^1(\Omega)\cap \H^2(\Omega))$ is dense in $L^2(0,T ; \H_{0,\sigma}^1(\Omega))$, there exists a sequence $\lbrace \uu_n \rbrace \subset  {C_0^\infty}(0,T; \H_{0,\sigma}^1(\Omega)\cap \H^2(\Omega))$ such that $\uu_n \rightarrow \uu $ in $L^2(0,T;  \H^1_{0,\sigma}(\Omega))$ as $n\to\infty$. 
Besides, it was shown in \cite[proof of Theorem 4.1]{GMT2019} that there exists a sequence $ \lbrace \phi_{0,n}  \rbrace \subset H^3(\Omega)$ such that
\begin{itemize}
\item[1.] There exist $m\in (0,1)$ depending only on $\overline{\phi_0}$ and $\delta=\delta(n) \in (0,1)$ such that
$$
\left| \overline{\phi_{0,n}} \right| \leq m, \quad 
\| \phi_{0,n}\|_{L^\infty(\Omega)}\leq 1-\delta, \quad \forall \, n \in \mathbb{N}.
$$

\item[2.] $\phi_{0,n}\rightarrow \phi_0$ in $H^1(\Omega)$, $\phi_{0,n}\rightharpoonup \phi_0$ weakly in $H^2(\Omega)$ and $-\Delta \phi_{0,n}+ F'(\phi_{0,n}) \rightarrow -\Delta \phi_{0}+ F'(\phi_{0}) $ in $H^1(\Omega)$ with 
$$
\left\| -\Delta \phi_{0,n}+ F'(\phi_{0,n}) \right\|_{H^1(\Omega)} \leq \left\| -\Delta \phi_{0}+ F'(\phi_{0} \right\|_{H^1(\Omega)}.
$$

\item[3.] For any $n \in \mathbb{N}$, $\partial_\n \phi_{0,n}=0$ almost everywhere on $\partial \Omega$.
\end{itemize}
Now, owing to the first part of the proof, for any $n \in \mathbb{N}$, there exists a pair $(\phi_n, \mu_n)$ solving the Cahn-Hilliard system with divergence-free drift
\begin{equation}
\label{cCH-n}
\begin{cases}
\partial_t \phi_n +\uu_n \cdot \nabla \phi_n = \Delta \mu_n\\
\mu_n = -\Delta \phi_n+F'(\phi_n)-\theta_0 \phi_n
\end{cases}
\quad \text{a.e. in }\ \Omega \times (0,T),
\end{equation}
and $\partial_\n \phi_n=\partial_\n \mu_n=0$ almost everywhere on $\partial \Omega \times (0,T)$, as well as $\phi_n(\cdot, 0)=\phi_{0,n}$ in $\Omega$. 
It follows from \eqref{E1-f} and \eqref{E2-f} that 
\begin{equation}
\label{E1-n}
\begin{split}
\| \nabla \mu_n \|_{L^\infty(0,T;L^2(\Omega)} 
&\leq \left(  4 \left\| \nabla \left( -\Delta \phi_{0,n} +\Psi'(\phi_{0,n}) \right) \right\|_{L^2(\Omega)}^2+  4C \int_0^{T} \| \nabla \uu_n(s)\|_{L^2(\Omega)}^2 \, \d s\right)^\frac12\\
&\quad \times \text{exp}\left( C \int_0^{T} \| \nabla \uu_n(s)\|_{L^2(\Omega)}^2 \, \d s \right),
\end{split}
\end{equation}
and 
 \begin{equation}
\label{E2-n}
\begin{split}
\int_0^{T} \| \nabla \partial_t \phi_n(s)\|_{L^2(\Omega)}^2 \, \d s
&\leq 6 \left(  \left\| \nabla \left(-\Delta \phi_{0,n} +\Psi'(\phi_{0,n}) \right) \right\|_{L^2(\Omega)}^2+  C \int_0^{T} \| \nabla \uu_n(s)\|_{L^2(\Omega)}^2 \, \d s\right) \\
&\quad \times \left( 1+ \left( C \int_0^{T} \| \nabla \uu_n(s)\|_{L^2(\Omega)}^2 \, \d s\right) 
 \mathrm{exp}\left( C \int_0^{T} \| \nabla \uu_n(s)\|_{L^2(\Omega)}^2 \, \d s \right)\right).
\end{split}
\end{equation}
In light of the properties of $\uu_n$ and $\phi_{0,n}$, we simply obtain that
\begin{equation}
\| \nabla \mu_n \|_{L^\infty(0,T;L^2(\Omega))} \leq Q_1,\quad
\| \partial_t \phi_n\|_{L^2(0,T;H^1(\Omega))}\leq Q_2,
\end{equation}
where the positive constants $Q_1$ and $Q_2$ are independent of $n$. By reasoning as above to get \eqref{mu-H1}-\eqref{muH3}, it follows that
\begin{equation}
\begin{split}
&\| \mu_n\|_{L^\infty(0,T;H^1(\Omega))}\leq Q_3, \quad 
\| \mu_n\|_{L^2(0,T;H^3(\Omega))}\leq Q_4(T),\\
&\| \phi_n\|_{L^\infty(0,T; W^{2,p}(\Omega))}\leq Q_5(p),\quad  
 \| F'(\phi_n)\|_{L^\infty(0,T; L^p(\Omega))} \leq  Q_6,
\end{split}
\end{equation} 
where the positive constants $Q_3$, $Q_4(T)$, $Q_5(p)$ and $Q_6$ are independent of $n$. Thanks to these bounds, in a similar way as for the vanishing viscosity limit, we obtain by compactness the existence of two limit functions $\phi \in L^\infty(0,T;W^{2,p}(\Omega))\cap H^1(0,T;H^1(\Omega))$, such that $F'(\phi)\in L^\infty(0,T; L^p(\Omega))$ with $p$ as above, and 
$\mu \in L^\infty(0,T;H^1(\Omega))\cap L^2(0,T;H^3(\Omega))$,
which solve \eqref{CH} almost everywhere on $\Omega \times (0,T)$ and \eqref{bcic} almost everywhere on $\partial \Omega \times (0,T)$. The uniqueness of such solution is inferred from \cite[Theorem 6]{Abels2009} (see Theorem \ref{well-pos} and Remark \ref{remark-wellposs}). By the arbitrariness of $T$, we can extend the solution on $(0,\infty)$ and obtain the conclusion as stated in Theorem \ref{CH-strong}. The estimates \eqref{N-mu} and \eqref{N-phit} simply follows by passing to the limit as $n \rightarrow \infty$ in \eqref{E1-n}-\eqref{E2-n} and letting $T\rightarrow \infty$. 

\medskip
Finally, if $\uu \in L^\infty(0, \infty;\L^2_\sigma(\Omega))\cap L^2(0,\infty;\H^1_{0,\sigma}(\Omega))$, we deduce from 
\begin{equation}
\label{phit-est}
\| \partial_t \phi\|_{H^1(\Omega)'}\leq C \left( \|\nabla \mu \|_{L^2(\Omega)}+ \| \uu \|_{L^2(\Omega)} \right),
\end{equation}
that $\partial_t \phi \in L^\infty(0,\infty; H^1(\Omega)')$.
The proof is complete.
\end{proof}

In order to derive a further global estimate for the gradient of the chemical potential in \eqref{CH}, we report the following well-known result for the Neumann problem (see, e.g., \cite{McLean}). 

\begin{lemma}[Neumann problem]
\label{Neumann-prob}
Let $\Omega$ be a bounded domain in $\mathbb{R}^d$, $d=2,3$, with $C^3$ boundary. Given $f \in W^{k,p}(\Omega)$ and $g \in W^{k+1-1/p,p}(\partial \Omega)$, where $k=0,1$ and $p \in (1,\infty)$, such that
\begin{equation}
\label{comp-cond}
\int_\Omega f \, \d x = \int_{\partial \Omega} g \, \d \sigma,  
\end{equation}
there exists a $u \in W^{k+2,p}(\Omega)$ which is unique up to a constant satisfying
\begin{equation}
\begin{split}
\Delta u&=f \quad \text{in } \Omega,\\
\partial_\n u&=g\quad \text{on } \partial \Omega.
\end{split}
\end{equation}
In addition, 
\begin{equation}
\label{est-grad}
\| \nabla u \|_{W^{k+1,p}(\Omega)}\leq C \left( \| f\|_{W^{k,p}(\Omega)}+ \| g\|_{W^{k+1-1/p,p}(\partial \Omega)}\right),
\end{equation}
where the positive constant $C$ only depends on $k$, $p$ and $\Omega$.
\end{lemma}

Using this, we derive the following conclusion for the gradient of the chemical  potential. Beyond its intrinsic interest, it will play a crucial role to control the flux $\widetilde{\J}$ in the momentum equation \eqref{AGG}.

\begin{corollary}
\label{HREG-MU}
Let the assumptions of Theorem \ref{CH-strong}. Then, we have the following estimate for $k=1,2$
 \begin{equation}
\label{high-mu}
\begin{split}
\int_0^\infty \| \nabla \mu(s)\|_{H^k(\Omega)}^2 \, \d s
&\leq  
C \left(  \left\| \nabla \left( -\Delta \phi_{0} +\Psi'(\phi_{0}) \right) \right\|_{L^2(\Omega)}^2+ C \int_0^\infty \| \nabla \uu(s)\|_{L^2(\Omega)}^2 \, \d s\right) \\
&\quad \times \left( 1+\left( \int_0^\infty \| \nabla \uu(s)\|_{L^2(\Omega)}^2 \, \d s \right) \ \mathrm{exp}\left( 2C \int_0^\infty \| \nabla \uu(s)\|_{L^2(\Omega)}^2 \, \d s \right) \right).
\end{split}
\end{equation}
where the positive constant $C$ depends only on $k$, $\Omega$, $\theta$, $\theta_0$, $F$, and $\overline{\phi_0}$.
\end{corollary}

\begin{proof}
Since 
$$
\Delta \mu= \partial_t \phi + \uu\cdot \nabla \phi \ \text{ a.e. in } \Omega \times (0,\infty), \quad \partial_\n \mu= 0 \ \text{ a.e. on } \partial \Omega \times (0,\infty),
$$
and $\int_\Omega {\partial_t \phi + \uu\cdot \nabla \phi}(t) \, \d x=0$ for all $t \in (0,\infty)$, it follows from Lemma \ref{Neumann-prob}, the generalized Poincaré inequality and \eqref{phi2p} that
\begin{equation}
\begin{split}
\| \nabla \mu\|_{H^1(\Omega)}
&\leq C \left( \| \partial_t \phi\|_{L^2(\Omega)}+ \| \uu \cdot \nabla \phi\|_{L^2(\Omega)} \right) \\
& \leq C \left( \| \partial_t \phi\|_{L^2(\Omega)}+ \| \uu\|_{L^6(\Omega)} \| \nabla \phi\|_{L^3(\Omega)}\right)\\
& \leq C \left( \| \nabla \partial_t \phi\|_{L^2(\Omega)}+ \| \nabla \uu\|_{L^2(\Omega)} \| \phi\|_{H^2(\Omega)}\right)\\
& \leq C \left( \| \nabla \partial_t \phi\|_{L^2(\Omega)}+ \| \nabla \uu\|_{L^2(\Omega)} \left( 1+ \| \nabla \mu\|_{L^2(\Omega)} \right) \right).
\end{split}
\end{equation}
Similarly, we find 
\begin{equation}
\begin{split}
\| \nabla \mu\|_{H^2(\Omega)}
&\leq C \left( \| \partial_t \phi\|_{H^1(\Omega)}
+\| \uu \cdot \nabla \phi\|_{H^1(\Omega)} \right) \\
& \leq C \left( \| \partial_t \phi\|_{H^1(\Omega)}
+ \| \nabla \uu\|_{L^2(\Omega)} \left( 1+\| \nabla \mu\|_{L^2(\Omega)}\right)
+\| \nabla (\uu \cdot \nabla \phi )\|_{L^2(\Omega)} \right) 
\\
& \leq C \left( \| \nabla \partial_t \phi\|_{L^2(\Omega)}
+\| \nabla \uu\|_{L^2(\Omega)} \left( 1+\| \nabla \mu\|_{L^2(\Omega)}\right) \right. \\
&\quad \left. + \| \nabla \uu\|_{L^2(\Omega)} \| \nabla \phi\|_{L^\infty(\Omega)}
+ \| \uu\|_{L^6(\Omega)} \| \phi\|_{W^{2,3}(\Omega)} \right)\\
 & \leq C \left( \| \nabla \partial_t \phi\|_{L^2(\Omega)}+
\| \nabla \uu\|_{L^2(\Omega)} \left( 1+ \| \nabla \mu\|_{L^2(\Omega)} \right)\right).
\end{split}
\end{equation}
Thus, in light of \eqref{N-mu} and \eqref{N-phit}, for $k=1,2$, we conclude that 
\begin{equation}
\label{nmu-H1-e}
\begin{split}
\int_0^\infty \| \nabla \mu(s)\|_{H^k(\Omega)}^2 \, \d s
&\leq  
C \left(  \left\| \nabla \left( -\Delta \phi_{0} +\Psi'(\phi_{0}) \right) \right\|_{L^2(\Omega)}^2+ C \int_0^\infty \| \nabla \uu(s)\|_{L^2(\Omega)}^2 \, \d s\right) \\
&\quad \times \left( 1+\left( \int_0^\infty \| \nabla \uu(s)\|_{L^2(\Omega)}^2 \, \d s \right) \ \mathrm{exp}\left( 2C \int_0^\infty \| \nabla \uu(s)\|_{L^2(\Omega)}^2 \, \d s \right) \right).
\end{split}
\end{equation}
\end{proof}

We are now in the position to prove the first part of our main result 
\begin{proof}[Proof of Theorem \ref{MAIN} - (i) Global regularity of the concentration]\label{131}
Let $(\uu,\phi)$ be a global weak solution to \eqref{AGG}-\eqref{AGG-bc} given by Theorem \ref{WEAK-SOL}. Consider an arbitrary positive time $\tau$. Thanks to the energy inequality \eqref{energy-ineq} and the Korn inequality, we infer that
$$
\int_0^\infty \| \nabla \uu(s)\|_{L^2(\Omega)}^2 + \| \nabla \mu(s)\|_{L^2(\Omega)}^2 \, \d s < \infty.
$$
In addition, we have that $\mu \in L_{\uloc}^2([0,\infty); H^1(\Omega))$, $\phi \in L^4_{\uloc}([0,\infty; H^2(\Omega))$ and $\partial_\n \phi=0$ almost everywhere on $\partial \Omega \times (0,\infty)$ (cf. Remark \ref{rem-weak-sol}). Thus, there exists $\tau^\star \in (0,\tau]$ such that
$\phi(\tau^\star) \in H^2(\Omega)$ 
with $\| \phi(\tau^\star)\|_{L^\infty(\Omega)}\leq 1$, 
$|\overline{\phi(\tau^\star)}|<1$, $\mu(\tau^\star)=-\Delta \phi(\tau^\star)+\Psi'(\phi(\tau^\star)) \in H^1(\Omega)$ and
$\partial_\n \phi (\tau^\star)=0$ on $\partial \Omega$.
An application of Theorem \ref{CH-strong} on the interval $[\tau^\star,\infty)$, together with the uniqueness of weak solutions (cf. Remark \ref{remark-wellposs}), gives us the desired conclusion.
\end{proof}

\section{Large Time Behaviour and Strict Separation Property}\label{sec:LargeTimeBehaviour}
\setcounter{equation}{0}
For any $m \in \mathbb{R}$, let us define the function spaces
$$
L_{(m)}^2(\Omega)=\left\lbrace f \in L^2(\Omega): \frac1{|\Omega|} \int_\Omega f(x) \, \d x =m \right\rbrace,
\quad
H^1_{(m)}(\Omega)=  \left\lbrace f \in H^1(\Omega): \frac1{|\Omega|} \int_\Omega f(x) \, \d x =m \right\rbrace.
$$
Notice that $H^1_{(m)}(\Omega)$ is not a linear space if $m\neq 0$. Nevertheless, it can be identified with $H^1_{(0)}(\Omega)$ by simply translation with the constant $m$. Moreover, the tangent space of $H^1_{(m)}(\Omega)$ is $H^1_{(0)}(\Omega)$. Hence, if $G\colon H^1_{(m)}(\Omega)\to \R$ is differentiable, then $DG(f)\colon H^1_{(0)}(\Omega)\to \R$ is linear and bounded, i.e., $DG(f)\in H^{1}_{(0)}(\Omega)'$ for all $f\in H^1_{(m)}(\Omega)$.

We now introduce
\begin{equation*}
 E_0(\varphi) = \int_\Omega  \frac12 |\nabla \varphi|^2 + F(\varphi)  \, \d x
\end{equation*}
for $\varphi \in \dom (E_0)= \lbrace f\in H^1_{(m)}(\Omega): |f(x)|\leq 1 \text{ a.e. in } \Omega \rbrace$, where $F$ is the ``convex part'' of $\Psi$ (cf. \eqref{Log}).
We recall that $E_0$ is convex and lower semi-continuous. Thanks to \cite[Theorem~4.3]{AW2007}, the subgradient $\partial E_0 (\varphi)\in L^2_{(0)}(\Omega)$ for all $\varphi \in \mathcal{D}(\partial E_0)$, where 
\begin{equation*}
\mathcal{D}(\partial E_0) =\left\lbrace
 \varphi\in H^2(\Omega)\cap L^2_{(m)}(\Omega):F'(\varphi)\in L^2(\Omega),  F''(\varphi)|\nabla \varphi|^2 \in L^1(\Omega), \partial_\n \varphi|_{\partial\Omega} =0 
 \right\rbrace
  \end{equation*} 
is given as
$$
 \partial E_0 ( \varphi) = -\Delta \varphi + P_0 F'(\varphi), 
 \quad
 \text{where}
 \quad 
 P_0 f = f- \frac1{|\Omega|}\int_\Omega f(x) \, \d x.
$$
Let us now consider $\phi_\infty\in \mathcal{D}(\partial E_0)$, which solves the stationary Cahn-Hilliard equation 
\begin{alignat}{2}
\label{eq:StatCH1}
 -\Delta \phi + \Psi'(\phi)&=\text{const.} &\qquad&\text{in } \Omega,\\\label{eq:StatCH2}
 \partial_{\n} \phi &=0 & & \text{on } \partial\Omega,\\
 \label{eq:StatCH3}
 \frac{1}{|\Omega|}\int_\Omega \phi(x) \ \d x &= m.
  \end{alignat}
We recall that solutions to \eqref{eq:StatCH1}-\eqref{eq:StatCH3} are critical points of the functional $E_{\text{free}}$ on $H^1_{(m)}(\Omega)$. In addition, they are separated from the pure phases $\pm 1$ as stated in the next result, whose proof can be found in \cite[Proposition 6.1]{AW2007} and in \cite[Lemma A.1]{CG2020}.
\begin{proposition}\label{prop:1}
Let $\phi_\infty\in \mathcal{D}(\partial E_0)$ be a solution to (\ref{eq:StatCH1})-(\ref{eq:StatCH3}). Then, there exist two constants $M_j,\ j=1,2$, such that
\begin{equation}
  \label{eq:boundsOnU}
-1<M_1\le \phi_\infty(x)\le M_2<1\qquad \text{for all}\  x\in\overline{\Omega}. 
\end{equation}
\end{proposition}
 
The first important step towards the longtime stabilization of weak solutions to \eqref{AGG}-\eqref{AGG-bc} is to show
\begin{lemma}\label{lem:ConvU}
Let $(\uu,\phi)$ be a weak solution to \eqref{AGG}-\eqref{AGG-bc} in the sense of Theorem \ref{WEAK-SOL}. Then $\uu(t)\to {\bf 0}$ in $\L^2(\Omega)$ as $t\to\infty$.  
\end{lemma}
\begin{proof}
  First of all, we have by \eqref{energy-ineq} that the energy inequality for the full system
  \begin{equation}
\label{eq:EnergyInequality}
E(\uu(t), \phi(t)) +\int_\tau^t \int_{\Omega} \nu(\phi(s)) | D \uu(s)|^2 \, \d x \d s+ 
\int_\tau^t \int_{\Omega} |\nabla \mu(s)|^2 \, \d x \d s \leq E(\uu(\tau), \phi(\tau))
\end{equation}
holds for almost every $\tau\geq 0$ including $\tau=0$ and every $t\geq \tau$.
Subtracting the energy equality \eqref{EI} for the Cahn-Hilliard equation with divergence-free drift, we obtain
\begin{equation}
\label{Kin-en-eq} 
E_{\text{kin}}(\uu(t)) +\int_\tau^t \int_{\Omega} \nu(\phi(s)) | D \uu(s)|^2 \, \d x \d s 
 \leq E_{\text{kin}}(\uu(\tau)) -\int_{\tau}^t \int_\Omega \uu(s)\cdot  \nabla \mu(s) \, \phi(s) \, \d x \d s
\end{equation}
for almost every $\tau\geq 0$ including $\tau=0$ and every $t\geq \tau$.
Now, let $\eps>0$ be arbitrary. Since $\nabla \mu \in L^2(\Omega\times (0,\infty))$ by the energy estimate, there is some $T'>0$ such that $\|\nabla \mu\|_{L^2(\Omega\times (T,\infty))} \leq \eps$ for all $T\geq T'$. Moreover, since $\uu \in L^2(0,\infty;\H_{0,\sigma}^1(\Omega))$, the set of all $T\geq T'$ such that $\|\uu(T)\|_{L^2(\Omega)}\leq \eps$ has positive measure. Hence there is some $T\geq T'$ such that $\|\nabla \mu\|_{L^2(\Omega\times (T,\infty))}\leq \eps$ and $\|\uu(T)\|_{L^2(\Omega)}\leq \eps$ hold true. Because of \eqref{Kin-en-eq}, and exploiting the $L^\infty$ bound of $\phi$ (cf. \eqref{Weak-reg}), we deduce that
\begin{equation*}
\begin{split}
&\sup_{T\leq t <\infty} \int_\Omega \rho(\phi(t)) \frac{|\uu(t)|^2}2 \, \d x  +\int_T^\infty\int_\Omega \nu(\phi)|D\uu|^2\, \d x \, \d t \\
&\quad \leq \|\nabla \mu\|_{L^2(T,\infty;L^2(\Omega))} \|\uu\|_{L^2(T,\infty;L^2(\Omega))}+ \int_\Omega \rho(\phi(T))\frac{|\uu(T)|^2}2 \, \d x  
\end{split}
  \end{equation*}
for almost every $T>0$. 
Therefore, by the Cauchy-Schwarz, Korn  and Young inequalities we obtain
\begin{equation*}
  \|\uu\|_{L^\infty(T,\infty; L^2(\Omega))}^2 + \|\uu\|_{L^2(T,\infty;H^1(\Omega))}^2 \leq C\left(\|\nabla \mu \|_{L^2(T,\infty;L^2(\Omega))}^2+ \|\uu (T) \|_{L^2(\Omega)}^2\right)\leq 2C\eps,
  \end{equation*}
  where $C$ is independent of $T$ and $\eps$. This shows the claim.
\end{proof}
{We highlight that the proof of Lemma \ref{lem:ConvU} consists of a direct argument which only relies on the validity of the total energy inequality for (weak) solutions of \eqref{AGG} and the free energy equality for solutions to the Cahn-Hilliard equation with drift. In comparison with \cite[Lemma 11]{Abels2009}, no square-summability of the time derivative of $\uu$ is requested.}



Next, we proceed with the proof of the second part of our main result.

\begin{proof}[Proof of Theorem \ref{MAIN} - (ii) Separation property] 
Let us recall from the first part of the proof (cf. \eqref{REG-CONC}) that $\phi \in L^\infty(\tau,\infty;W^{2,p}(\Omega))$ for any $\tau>0$, and for arbitrary $2\leq p<\infty$ if $d=2$ and $p=6$ if $d=3$.
Now we define the $\omega$-limit set of $(\uu,\phi)$ as 
\begin{align*}
\omega(\uu,\phi) &= \left\lbrace (\uu',\phi')\in \L^2_\sigma(\Omega)\times W^{2-\eps ,p}(\Omega): \right.
\\
& \qquad 
\left. \exists \, t_n \nearrow\infty\ \text{ such that }\ (\uu(t_n),\phi(t_n))\to (\uu',\phi')\ \text{in}\  \L^2(\Omega)\times W^{2-\eps,p}(\Omega) \right\rbrace,
\end{align*}
where $\eps>0$. Since $\phi\in BUC([\tau,\infty); W^{2-\eps' ,p}(\Omega))$ for any $\eps'\in (0,\eps)$ and $\tau>0$ by interpolation, and $\uu(t)\to_{t\to\infty} \mathbf{0}$ by Lemma \ref{lem:ConvU}, we easily obtain that $\omega(\uu,\phi)$ is a non-empty, compact, and connected subset of $\L^2_\sigma(\Omega)\times W^{2-\eps ,p}(\Omega))$ {(cf. \cite[Definition 1.4.1 and Theorem 1.4.7]{C2015})}. In addition, owing to the fact that $E$ is a strict Lyapunov functional for \eqref{AGG} {and following \cite[Lemma 11]{Abels2009},} we are able to prove:
\begin{lemma}\label{lem:OmegaLimit}
Let $(\uu,\phi)$ be a weak solution to \eqref{AGG}-\eqref{AGG-bc} in the sense of Theorem \ref{WEAK-SOL}. Then, we have
\begin{equation*}
\omega(\uu,\phi)\subseteq \left\{ ({\bf 0},\phi'): \phi'\in W^{2,p}(\Omega)\cap H^1_{(m)}(\Omega)\
  \text{solves (\ref{eq:StatCH1})-(\ref{eq:StatCH3})}\right\},
  \quad \text{where} \quad m = \overline{\phi_0}.
  \end{equation*}
\end{lemma}
\begin{proof}
Thanks to \eqref{EI}, we recall that
  \begin{equation}
\label{eq:EnergyInequality-CH}
E_{\text{free}}(\phi(t)) + 
\int_0^t \int_{\Omega} |\nabla \mu|^2 \, \d x \d s = E_{\text{free}}(\phi_0) - \int_0^t \int_\Omega \uu\cdot \nabla \mu \, \phi\,\d x  \d s, \quad \forall \, t \geq 0.
\end{equation}
Since $\uu\cdot \nabla \mu \, \phi\in L^1(\Omega\times (0,\infty))$, the limit $\lim_{t\to \infty} E_{\text{free}}(\phi(t))$ exists. In addition, owing to $\uu (t)\to_{t\to\infty} \mathbf{0}$ in $\L^2_\sigma (\Omega)$, $\lim_{t\to\infty} E_{\text{kin}}(\uu(t))=0$, and thereby the limit
$ E_\infty:= \lim_{t\to \infty} E(\uu(t), \phi(t))$ exists. 

Let us consider a sequence $\lbrace t_n \rbrace_{n\in\N}\subset \mathbb{R}$ such that $0\leq t_n\nearrow\infty$. We set $\lim_{n\to \infty} (\uu(t_n),\phi(t_n))= (\uu'_0,\phi'_0)$. By Lemma~\ref{lem:ConvU}, it clearly follows that $\uu'_0={\bf 0}$. Now define $(\uu_n(t),\phi_n(t)):=(\uu(t+t_n),\phi(t+t_n))$ for $t\in[0,\infty)$.
Due to \cite[Theorem~6]{Abels2009}, $\lbrace \phi_n \rbrace_{n\in\N}$ converges weakly in $L^2(0,T;W^2_p(\Omega))\cap H^1(0,T;H^{1}(\Omega)')$, for every $T>0$, to a limit function $\phi'$, which is a weak solution (as defined in Theorem \ref{well-pos}) to \eqref{CH}-\eqref{bcic} with $\uu={\bf 0}$, chemical potential $\mu'$ and initial value $\phi'|_{t=0}= \phi_0'$. In particular, $\phi_n\to_{n\to \infty} \phi'$ in $L^2(0,T;H^1(\Omega))$ for every $T>0$. Hence, there is a subsequence such that
\begin{equation*}
 E_{\text{free}}(\phi_n(t))\to_{n\to \infty} E_{\text{free}}(\phi'(t))\quad \text{for a.e.}\ t\in [0,\infty). 
\end{equation*}
On the other hand, 
$
\lim_{n\to \infty}E(\uu_n(t),\phi_n(t))= \lim_{n\to \infty}E_{\text{free}}(\phi_n(t))=  E_\infty
$ since $\uu_n(t)\to_{n\to\infty} {\bf 0}$ in $\L^2(\Omega)$. 
Thus, $E_{\text{free}}(\phi'(t))=E_\infty$ for a.e.\ $t\in [0,\infty)$. 
As a consequence, by the energy identity for the convective Cahn-Hilliard system \eqref{EI} with $\uu\equiv {\bf 0}$, we deduce that $\nabla \mu'(t)=0$ for almost all $t\in [0,\infty)$. Therefore, $\partial_t \phi'(t)=0$, and $\phi'(t)\equiv \phi_0'$ solves the stationary Cahn-Hilliard equation \eqref{eq:StatCH1}-\eqref{eq:StatCH3} with $m=\overline{\phi_0}$.
\end{proof}

Finally, we are in the position to show the desired separation property  \eqref{separation}. 
In fact, thanks to the above characterization of $\omega(\uu,\phi)$, Proposition~\ref{prop:1}, the embedding $W^{2-\eps,p}(\Omega)\hookrightarrow C(\overline{\Omega})$ for $\eps>0$ sufficiently small, and the compactness of $\omega(\uu,\phi)$ in $\L_\sigma^2(\Omega)\times W^{2-\eps,p}(\Omega)$, there exists some $\delta'>0$ such that 
\begin{equation*}
 |\phi'(x)|\leq 1-\delta' \quad \forall \, x\in \overline{\Omega}, \ \text{ for any }  (\mathbf{0},\phi')\in \omega(\uu,\phi).
\end{equation*}
Observing that $\lim_{t\to\infty}\dist((\uu(t),\phi(t)),\omega(\uu,\phi))=0$ in the norm of $\L^2_\sigma(\Omega)\times W^{2-\eps,p}(\Omega)$, we conclude that, for every $\delta\in (0,\delta')$, there is some $T_{SP}>0$ such that
\begin{equation*}
 |\phi(x,t)|\leq 1-\delta \quad \forall \, x\in \overline{\Omega}, \, t\geq T_{SP}.
  \end{equation*}
\end{proof}

\section{Weak-strong uniqueness result}
\label{weak-strong-U}
\setcounter{equation}{0}

In this section we demonstrate a {\it weak-strong} uniqueness result for the system \eqref{AGG}-\eqref{AGG-bc} in both two and three dimensions. This will be essential to achieve the large time regularity of the velocity of each weak solution. {Due to the presence of the non-constant density, our approach is inspired by \cite{FN2012}. Moreover, the separation property plays a crucial role in our argument.} 

\begin{theorem}
\label{weak-strong}
Let $\Omega$ be a bounded domain in $\mathbb{R}^d$, $d=2,3$, of class $C^4$. Assume that $\uu_0 \in \H_{0,\sigma}^1(\Omega)$, $\phi_0 \in H^2(\Omega)$ be such that 
$\| \phi_0\|_{L^\infty(\Omega)}< 1$, $|\overline{\phi_0}|<1$, $\mu_0=-\Delta \phi_0+\Psi'(\phi_0) \in H^1(\Omega)$ and 
$\partial_\n \phi_0=0$ on $\partial \Omega$. 
In addition, we suppose that:
\begin{itemize}
\item $(\uu, \phi)$ is a solution on $[0,T]$ such that
\begin{equation}
\label{REG-weakS}
\begin{split}
&\uu \in BC_{\mathrm{w}}([0,T]; \L_\sigma^2(\Omega))\cap L^2(0,T;\H_{0,\sigma}^1(\Omega)),\\
&\phi \in L^\infty(0,T;H^3(\Omega))\cap L^2(0,T;H^4(\Omega)), \quad \partial_t \phi \in L^2(0,T;H^1(\Omega)),\\
&\phi \in C( \overline{\Omega}\times [0,T]),\quad\text{with}\quad \max_{t \in [0,T]} \|\phi(t)\|_{C(\overline{\Omega})}<1, \\
&\mu \in L^{\infty}(0,T; H^1(\Omega))\cap L^2(0,T;H^3(\Omega)),
\end{split}
\end{equation}
which satisfies
\begin{equation}
\label{weak-form}
\begin{split}
&\int_{\Omega} \rho(\phi(t)) \uu(t)\cdot \ww(t) \, \d x - \int_{\Omega} \rho(\phi_0) \uu_0 \cdot \ww(0) \, \d x- \int_0^t (\rho(\phi) \uu, \partial_t \ww) \, \d \tau 
-\int_0^t (\rho(\phi) \uu \otimes \uu, \nabla \ww) \, \d \tau 
\\ 
& \quad 
+ \int_0^t (\nu(\phi) D \uu, D \ww) \, \d \tau - \int_0^t (\uu \otimes \widetilde{\J}, \nabla \ww) \, \d \tau = \int_0^t (\mu \nabla \phi, \ww) \, \d \tau,
\end{split}
\end{equation}
where $\widetilde{\J}= -\rho'(\phi) \nabla \mu$, for any $t \in (0,T)$, for all $\ww \in W^{1,2}(0,T;\L_\sigma^2(\Omega))\cap L^4(0,T; \H^1_{0,\sigma}(\Omega))$, and 
\begin{equation}
\label{CH-weakS}
\partial_t \phi+ \uu \cdot \nabla \phi = \Delta \mu, \quad \mu= -\Delta \phi+ \Psi'(\phi)
\quad \text{a.e. in } \Omega\times (0,T).
\end{equation}
Moreover, $(\uu(0),\phi(0))=(\uu_0,\phi_0)$ 
and the energy inequality holds
 \begin{equation}
\label{E-ineq}
E(\uu(t), \phi(t)) +\int_0^t \int_{\Omega} \nu(\phi) | D \uu|^2 \, \d x \d \tau+ 
\int_0^t \int_{\Omega} |\nabla \mu|^2 \, \d x \d \tau \leq E(\uu_0, \phi_0)
\end{equation}
for any $t \in [0,T]$.
\smallskip

\item  $(\UU, \Pi, \Phi)$ is a strong solution on $[0,T]$ such that
\begin{equation}
\label{REG-strong}
\begin{split}
&\UU \in C([0,T]; \H^1_{0,\sigma}(\Omega))\cap L^2(0,T;\H^2(\Omega))\cap W^{1,2}(0,T;\L^2_\sigma(\Omega)),\quad \Pi \in L^2(0,T;H^1(\Omega)), \\
&\Phi \in L^\infty(0,T;H^3(\Omega))\cap L^2(0,T;H^4(\Omega)), \quad \partial_t \Phi \in L^2(0,T;H^1(\Omega)),\\
&\Phi \in C( \overline{\Omega}\times [0,T]),\quad\text{with}\quad \max_{t \in [0,T]} \|\Phi(t)\|_{C(\overline{\Omega})}<1, \\
&M \in L^{\infty}(0,T; H^1(\Omega))\cap L^2(0,T;H^3(\Omega)),
\end{split}
\end{equation}
which satisfies
\begin{equation}
\label{strong-form}
\begin{split}
\rho(\Phi) \partial_t \UU + \rho(\Phi) (\UU \cdot \nabla)\UU
+( \J^\star \cdot \nabla) \UU - \div (\nu(\Phi)D\UU) + \nabla \Pi= M \nabla \Phi \quad \text{a.e. in } \Omega\times (0,T),
\end{split}
\end{equation}
where $\J^\star= -\rho'(\Phi) \nabla M$, and 
\begin{equation}
\label{CH-strongS}
\partial_t \Phi+ \UU \cdot \nabla \Phi = \Delta M, \quad M= -\Delta \Phi+ \Psi'(\Phi)
\quad \text{a.e. in } \Omega\times (0,T),
\end{equation}
as well as $(\UU(0),\Phi(0))=(\uu_0,\phi_0)$. 
\end{itemize}
Then, $\uu=\UU$ and $\phi=\Phi$ on $[0,T]$.
\end{theorem}

\begin{remark}\label{rem:EnergyIdentity}
Notice that if $(\uu,\phi)$ is a weak solution to \eqref{AGG}-\eqref{AGG-bc} (as in Theorem \ref{WEAK-SOL}) with $\mu \in L^\infty(0,T; H^1(\Omega))$ and  $\uu\in L^2(0,T;\H^2(\Omega))$, then one can choose $\uu$ as a test function in \eqref{uu-weak}. Exploiting \eqref{EI}, we then obtain the energy identity
 \begin{equation}
\label{EE-strong}
E(\uu(t), \phi(t)) +\int_0^t \int_{\Omega} \nu(\phi) | D \uu|^2 \, \d x \d \tau+ 
\int_0^t \int_{\Omega} |\nabla \mu|^2 \, \d x \d \tau = E(\uu_0, \phi_0)
\end{equation}
for every $t \in [0,T]$.
\end{remark}

\begin{proof}
First of all, thanks to Remark~\ref{rem:EnergyIdentity}, we have the energy identity \eqref{EE-strong} for $(\UU,\Phi,M)$ (i.e. replacing $(\uu,\phi,\mu)$ with $(\UU,\Phi,M)$). In addition, Theorem \ref{well-pos} entails the validity of the following free energy equalities
$$
E_{\text{free}}(\phi(t))+ \int_0^t \int_{\Omega} |\nabla \mu|^2 \, \d x \d \tau 
+ \int_0^t (\uu \cdot \nabla \phi, \mu) \, \d \tau = E_{\text{free}}(\phi_0),
$$
and 
$$
E_{\text{free}}(\Phi(t))+ \int_0^t \int_{\Omega} |\nabla M|^2 \, \d x \d \tau 
+ \int_0^t (\UU \cdot \nabla \Phi, M) \, \d \tau = E_{\text{free}}(\phi_0),
$$
for all $t\in [0,T]$. Thus, we deduce from \eqref{E-ineq} and \eqref{EE-strong} (for $(\UU, \Phi, M)$) that
\begin{equation}
\label{Kin-Ineq}
\int_{\Omega}  \frac12 \rho(\phi(t)) |\uu(t)|^2 \, \d x +\int_0^t \int_{\Omega} \nu(\phi) | D \uu|^2 \, \d x \d \tau 
\leq \int_{\Omega}  \frac12 \rho(\phi_0) |\uu_0|^2 \, \d x 
+ \int_0^t (\mu \nabla \phi, \uu) \, \d \tau,
\end{equation}
and 
\begin{equation}
\label{Kin-Eq}
\int_{\Omega}  \frac12 \rho(\Phi(t)) |\UU(t)|^2 \, \d x +\int_0^t \int_{\Omega} \nu(\Phi) | D \UU|^2 \, \d x \d \tau 
= \int_{\Omega}  \frac12 \rho(\phi_0) |\uu_0|^2 \, \d x + \int_0^t (M \nabla \Phi, \UU) \, \d \tau,
\end{equation}
for all $t \in [0,T]$. Now, taking $\ww= - \UU$ in \eqref{weak-form}, we obtain
\begin{equation}
\label{UU-test}
\begin{split}
&-\int_{\Omega} \rho(\phi(t)) \uu(t)\cdot \UU(t)\, \d x + \int_{\Omega} \rho(\phi_0) |\uu_0|^2\, \d x + \int_0^t (\rho(\phi) \uu, \partial_t \UU) \, \d \tau 
+\int_0^t (\rho(\phi) \uu \otimes \uu, \nabla \UU) \, \d \tau 
\\ 
& \quad 
- \int_0^t (\nu(\phi) D \uu, D \UU) \, \d \tau + \int_0^t (\uu \otimes \widetilde{\J}, \nabla \UU) \, \d \tau = -\int_0^t (\mu \nabla \phi, \UU) \, \d \tau,
\end{split}
\end{equation}
for all $t \in [0,T]$. 
We recall the following basic relations
\begin{equation}
\label{rho1}
\begin{split}
\int_{\Omega}  \frac12 \rho(\phi) |\uu-\UU|^2 \, \d x
&=
\int_{\Omega}  \frac12 \rho(\phi) \left(  |\uu|^2 + |\UU|^2\right) \, \d x
- \int_{\Omega}  \rho(\phi) \uu \cdot \UU \, \d x
\\
&=
\int_{\Omega}  \frac12 \rho(\phi)  |\uu|^2 \, \d x  + \int_{\Omega} \rho(\Phi) |\UU|^2 \, \d x + \int_{\Omega} \frac12 \left( \rho(\phi)-\rho(\Phi)\right) |\UU|^2 \, \d x
\\
&\quad - \int_{\Omega}  \rho(\phi) \uu \cdot \UU \, \d x,
\end{split}
\end{equation}
and 
\begin{equation}
\label{nu1}
\begin{split}
\int_\Omega \nu(\phi) |D(\uu-\UU)|^2 \, \d x
&= \int_\Omega \nu(\phi) \left( |D \uu|^2 + | D \UU)|^2 - 2 D \uu: D \UU \right) \, \d x
\\
&=\int_\Omega \nu(\phi) |D \uu|^2 \, \d x +
\int_\Omega \nu(\Phi) | D \UU)|^2 \, \d x
+ \int_\Omega \left( \nu(\phi)-\nu(\Phi) \right)|D \UU|^2 \, \d x
\\
&\quad -\int_\Omega 2 \nu(\phi) D \uu: D \UU \, \d x.
\end{split}
\end{equation}
Summing \eqref{Kin-Ineq}, \eqref{Kin-Eq} and \eqref{UU-test}, and exploiting \eqref{rho1} and \eqref{nu1}, we find
\begin{equation}
\label{Kin-diff}
\begin{split}
&\int_{\Omega}  \frac12 \rho(\phi(t)) |\uu(t)-\UU(t)|^2 \, \d x + \int_0^t \int_\Omega \nu(\phi) |D(\uu-\UU)|^2 \, \d x \d \tau\\
&\quad
\leq 
\int_\Omega \frac12 \left( \rho(\phi(t))- \rho(\Phi(t))\right) |\UU(t)|^2 \, \d x
+ \int_0^t \left( (\nu(\phi)-\nu(\Phi)) D \UU, D \UU \right) \, \d \tau
\\
& \qquad
-\int_0^t (\nu(\phi) D \uu, D \UU) \, \d \tau  + \int_0^t (\mu \nabla \phi, \uu) \, \d \tau 
+ \int_0^t (M \nabla \Phi, \UU) \, \d \tau 
- \int_0^t (\mu \nabla \phi, \UU) \, \d \tau 
\\
&\qquad 
- \int_0^t (\rho(\phi)\uu, \partial_t \UU) \, \d \tau
-\int_0^t (\rho(\phi) \uu \otimes \uu, \nabla \UU) \, \d \tau
- \int_0^t (\uu \otimes \widetilde{\J}, \nabla \UU) \, \d \tau,
\end{split}
\end{equation}
for all $t \in [0,T]$. The conservation law of the densities reads as 
\begin{equation}
\partial_t \rho(\phi) + \div \left( \rho(\phi) \uu + \widetilde{\J}\right)=0, \quad
\partial_t \rho(\Phi) + \div \left( \rho(\Phi) \UU +\J^\star \right)=0
\quad \text{a.e. in } \Omega \times (0,T).
\end{equation}
For any function $v \in C^1([0,T]; C^1(\overline{\Omega}))$, we have
\begin{equation}
\label{rho-weak-form}
\begin{split}
&\int_\Omega \left( \rho(\phi(t))- \rho(\Phi(t)) \right) v(t) \, \d x - 
\underbrace{ \int_\Omega \left( \rho(\phi(0))- \rho(\Phi(0)) \right) v(0) \, \d x  }_{=0} - \int_0^t (\rho(\phi)-\rho(\Phi), \partial_t v) \, \d \tau\\
&\quad 
= \int_0^t \left( \rho(\phi)\uu- \rho(\Phi) \UU, \nabla v \right)\, \d \tau
+ \int_0^t \left( \widetilde{\J}- \J^\star, \nabla v\right) \, \d \tau,
\end{split}
\end{equation}
for all $t \in [0,T]$. By a density argument, we take $v= \frac12 |\UU|^2$ in \eqref{rho-weak-form} obtaining 
\begin{equation}
\label{rho-test}
\begin{split}
&\int_\Omega  \frac12 \left( \rho(\phi(t))- \rho(\Phi(t)) \right) |\UU(t)|^2 \, \d x \\
&\quad 
= \int_0^t (\rho(\phi)-\rho(\Phi), \UU \cdot \partial_t \UU) \, \d \tau
+ \int_0^t \left( \rho(\phi)\uu- \rho(\Phi) \UU, \nabla \left(\frac12 |\UU|^2\right) \right)\, \d \tau\\
&\qquad 
+ \int_0^t \left( \widetilde{\J}- \J^\star, \nabla \left( \frac12 |\UU|^2\right) \right) \, \d \tau,
\end{split}
\end{equation}
for all $t \in [0,T]$.
Substituting \eqref{rho-test} in \eqref{Kin-diff}, we deduce that
\begin{equation}
\label{Kin-diff-2}
\begin{split}
&\int_{\Omega}  \frac12 \rho(\phi(t)) |\uu(t)-\UU(t)|^2 \, \d x + \int_0^t \int_\Omega \nu(\phi) |D(\uu-\UU)|^2 \, \d x \d \tau\\
&\quad
\leq 
 \int_0^t (\rho(\phi)-\rho(\Phi), \UU \cdot \partial_t \UU) \, \d \tau
+ \int_0^t \left( \rho(\phi)\uu- \rho(\Phi) \UU, \nabla \left(\frac12 |\UU|^2\right) \right)\, \d \tau\\
&\qquad
+ \int_0^t \left( \widetilde{\J}- \J^\star, \nabla \left( \frac12 |\UU|^2\right) \right) \, \d \tau
+ \int_0^t \left( (\nu(\phi)-\nu(\Phi)) D \UU, D \UU \right) \, \d \tau
-\int_0^t (\nu(\phi) D \uu, D \UU) \, \d \tau  
\\
&\qquad 
+ \int_0^t (\mu \nabla \phi, \uu) \, \d \tau 
+ \int_0^t (M \nabla \Phi, \UU) \, \d \tau 
- \int_0^t (\mu \nabla \phi, \UU) \, \d \tau 
- \int_0^t (\rho(\phi)\uu, \partial_t \UU) \, \d \tau
\\
&\qquad 
-\int_0^t (\rho(\phi) \uu \otimes \uu, \nabla \UU) \, \d \tau
- \int_0^t (\uu \otimes \widetilde{\J}, \nabla \UU) \, \d \tau.
\end{split}
\end{equation}
On the other hand, multiplying \eqref{strong-form} by $\uu$ and integrating over $\Omega \times (0,t)$, we infer that
\begin{equation}
\label{strong-test}
\begin{split}
& \int_0^t (\rho(\Phi) \partial_t \UU, \uu) \, \d \tau 
+ \int_0^t  (\rho(\Phi) (\UU \cdot \nabla) \UU, \uu) \, \d \tau 
+ \int_0^t  ( (\J^\star \cdot \nabla ) \UU, \uu ) \, \d \tau 
\\ 
&\quad + \int_0^t  (\nu(\Phi) D \UU, D \uu) \, \d \tau 
 - \int_0^t  (M \nabla \Phi, \uu) \, \d \tau=0, 
 \end{split}
\end{equation}
for all $t \in [0,T]$. Summing \eqref{Kin-diff-2} and \eqref{strong-test}, we arrive at 
\begin{equation*}
\label{Kin-diff-3}
\begin{split}
&\int_{\Omega}  \frac12 \rho(\phi(t)) |\uu(t)-\UU(t)|^2 \, \d x + \int_0^t \int_\Omega \nu(\phi) |D(\uu-\UU)|^2 \, \d x \d \tau\\
&\quad
\leq 
 \underbrace{\int_0^t (\rho(\phi)-\rho(\Phi), \UU \cdot \partial_t \UU) \, \d \tau
 + \int_0^t (\rho(\Phi) \partial_t \UU, \uu) \, \d \tau  
 - \int_0^t (\rho(\phi)\uu, \partial_t \UU) \, \d \tau}_{I}
\\
&\qquad 
+\underbrace{ \int_0^t \left( \rho(\phi)\uu- \rho(\Phi) \UU, \nabla \left(\frac12 |\UU|^2\right) \right)\, \d \tau
+ \int_0^t  (\rho(\Phi) (\UU \cdot \nabla) \UU, \uu) \, \d \tau 
-\int_0^t (\rho(\phi) \uu \otimes \uu, \nabla \UU) \, \d \tau}_{II}
\\
&\qquad
+\underbrace{ \int_0^t  ( (\J^\star \cdot \nabla ) \UU, \uu ) \, \d \tau 
+ \int_0^t \left( \widetilde{\J}- \J^\star, \nabla \left( \frac12 |\UU|^2\right) \right) \, \d \tau
- \int_0^t (\uu \otimes \widetilde{\J}, \nabla \UU) \, \d \tau}_{III}
\\
&\qquad 
+ \underbrace{ \int_0^t \left( (\nu(\phi)-\nu(\Phi)) D \UU, D \UU \right) \, \d \tau
-\int_0^t (\nu(\phi) D \uu, D \UU) \, \d \tau  
+ \int_0^t  (\nu(\Phi) D \UU, D \uu) \, \d \tau }_{IV}
\\
&\qquad 
+ \underbrace{ \int_0^t (\mu \nabla \phi, \uu) \, \d \tau 
+ \int_0^t (M \nabla \Phi, \UU) \, \d \tau 
- \int_0^t (\mu \nabla \phi, \UU) \, \d \tau
- \int_0^t  (M \nabla \Phi, \uu) \, \d \tau}_{V},
\end{split}
\end{equation*}
for all $t \in [0,T]$. 
We notice that
\begin{align}
\label{I}
\begin{split}
&I
\begin{aligned}[t]
&= \int_0^t (\rho(\phi)-\rho(\Phi), \UU \cdot \partial_t \UU) \, \d \tau
 + \int_0^t (\rho(\Phi) -\rho(\phi) \partial_t \UU, \uu) \, \d \tau  
\\
&= - \int_0^t \left( \rho(\phi)-\rho(\Phi) \partial_t \UU,  \uu-\UU \right) \, \d \tau,
\end{aligned}
\end{split}
\\[15pt]
\label{II}
\begin{split}
&II
\begin{aligned}[t]
&=
\int_0^t \left( \left( \rho(\phi)\uu \cdot \nabla \right) \UU , \UU \right)\, \d \tau
- \int_0^t \left( \left( \rho(\Phi) \UU \cdot \nabla \right) \UU , \UU \right)\, \d \tau
+ \int_0^t  (\rho(\Phi) (\UU \cdot \nabla) \UU, \uu) \, \d \tau 
\\
&\quad 
-\int_0^t ( (\rho(\phi) \uu \cdot \nabla) \UU, \uu ) \, \d \tau
\\
&= \int_0^t \left( \left( \rho(\phi)\uu \cdot \nabla \right) \UU , (\UU-\uu) \right)\, \d \tau
- \int_0^t \left( \left( \rho(\Phi) \UU \cdot \nabla \right) \UU , (\UU-\uu) \right)\, \d \tau
\\
&= \int_0^t \left( \left( \left(\rho(\phi)\uu-\rho(\Phi) \UU \right) \cdot \nabla \right) \UU , (\UU-\uu) \right)\, \d \tau
\\
& = 
- \int_0^t \left( \left(  \rho(\phi)( \uu-\UU) \cdot \nabla \right) \UU , \uu-\UU \right)\, \d \tau
-\int_0^t \left( (\rho(\phi)-\rho(\Phi)) (\UU\cdot \nabla ) \UU, \uu-\UU \right) \, \d \tau,
\end{aligned}
\end{split}
\\[15pt]
\label{III}
\begin{split}
&III
\begin{aligned}[t]
&= - \frac{\rho_1-\rho_2}{2} \int_0^t  ( (\nabla M \cdot \nabla ) \UU, \uu ) \, \d \tau 
-\frac{\rho_1-\rho_2}{2}  \int_0^t \left(  ( (\nabla \mu- \nabla M) \cdot \nabla) \UU , \UU \right)  \, \d \tau\\
&\quad 
+ \frac{\rho_1-\rho_2}{2} 
\int_0^t ( (\nabla \mu \cdot \nabla ) \UU, \uu) \, \d \tau
\\
&= \frac{\rho_1-\rho_2}{2}  \int_0^t \left(  ( (\nabla \mu- \nabla M) \cdot \nabla) \UU , (\uu-\UU) \right)  \, \d \tau,
\end{aligned}
\end{split}
\\[15pt]
\label{IV}
\begin{split}
&IV
\begin{aligned}[t]
&= 
\int_0^t \left( (\nu(\phi)-\nu(\Phi)) D \UU, D \UU \right) \, \d \tau
-\int_0^t ( (\nu(\phi) - \nu(\Phi) ) D \uu, D \UU) \, \d \tau 
\\
&= -\int_0^t \left( (\nu(\phi)-\nu(\Phi)) D \UU, D (\uu-\UU)  \right) \, \d \tau,
\end{aligned}
\end{split}
\\[15pt]
\label{V}
\begin{split}
&V
\begin{aligned}[t]
&= 
\int_0^t (\mu \nabla \phi, \uu-\UU) \, \d \tau 
- \int_0^t (M \nabla \Phi, \uu - \UU) \, \d \tau 
= \int_0^t (\mu \nabla \phi - M \nabla \Phi, \uu-\UU) \, \d \tau. 
\end{aligned}
\end{split}
\end{align}
Thanks to \eqref{I}-\eqref{V}, we reach
\begin{equation}
\label{Kin-diff-4}
\begin{split}
&\int_{\Omega}  \frac12 \rho(\phi(t)) |\uu(t)-\UU(t)|^2 \, \d x + \int_0^t \int_\Omega \nu(\phi) |D(\uu-\UU)|^2 \, \d x \d \tau\\
&\quad
\leq  
- \int_0^t \left( \rho(\phi)-\rho(\Phi) \partial_t \UU,  \uu-\UU \right) \, \d \tau
- \int_0^t \left( \left(  \rho(\phi)( \uu-\UU) \cdot \nabla \right) \UU , \uu-\UU \right)\, \d \tau
\\
&\qquad
-\int_0^t \left( (\rho(\phi)-\rho(\Phi)) (\UU\cdot \nabla ) \UU, \uu-\UU \right) \, \d \tau
+\frac{\rho_1-\rho_2}{2}  \int_0^t \left(  ( (\nabla \mu- \nabla M) \cdot \nabla) \UU , (\uu-\UU) \right)  \, \d \tau
\\
&\qquad
-\int_0^t \left( (\nu(\phi)-\nu(\Phi)) D \UU, D (\uu-\UU)  \right) \, \d \tau
\int_0^t (\mu \nabla \phi - M \nabla \Phi, \uu-\UU) \, \d \tau,
\end{split}
\end{equation}
for all $t \in [0,T]$. The right-hand side can be estimated as in 
\cite[Section 6]{Gior2021} and in \cite[Section 4]{G2021}. As a result, we obtain
\begin{equation}
\label{Kin-diff-5}
\begin{split}
&\int_{\Omega}  \frac12 \rho(\phi(t)) |\uu(t)-\UU(t)|^2 \, \d x + \int_0^t \int_\Omega \nu(\phi) |D(\uu-\UU)|^2 \, \d x \d \tau\\
&
\leq  
\frac{\nu_\ast}{2} \int_0^t \| D (\uu-\UU)\|_{L^2(\Omega)}^2 \, \d \tau
+ \frac14 \int_0^t \| \Delta^2 (\phi-\Phi)\|_{L^2(\Omega)}^2 \, \d \tau\\
&\quad +C \int_0^t \left(1+ \| \UU\|_{H^2(\Omega)}^2+ \| \partial_t \UU\|_{L^2(\Omega)}^2 \right)
\left( \| (\uu-\UU) \|_{L^2(\Omega)}^2 + \| \Delta (\phi-\Phi) \|_{L^2(\Omega)}^2 \right) \, \d \tau.
\end{split}
\end{equation}
On the other hand, arguing as in \cite[Section 4]{G2021}, we have the differential equation
\begin{equation}
\label{phi-diff-1}
\begin{split}
&\frac12 \ddt \| \Delta (\phi-\Phi) \|_{L^2(\Omega)}^2 + \| \Delta^2 (\phi-\Phi)\|_{L^2(\Omega)}^2 \\
&= \int_{\Omega} \uu \cdot \nabla (\phi-\Phi) \Delta^2 (\phi-\Phi) \, \d x+
\int_{\Omega} (\uu-\UU) \cdot \nabla \Phi \Delta^2 (\phi-\Phi) \, \d x\\
&\quad + \int_{\Omega} \Delta( \Psi'(\phi)-\Psi'(\Phi)) \Delta^2 (\phi-\Phi) \, \d x.
\end{split}
\end{equation}
We observe that 
\begin{equation}
\label{corr1}
\begin{split}
\left| \int_{\Omega} \uu \cdot \nabla (\phi-\Phi) \, \Delta^2 (\phi-\Phi) \, \d x \right|
&\leq \| \uu\|_{L^3(\Omega)} \| \nabla (\phi-\Phi)\|_{L^6(\Omega)} 
\| \Delta^2 (\phi-\Phi)\|_{L^2(\Omega)}\\
&\leq \frac16 \|\Delta^2 (\phi-\Phi) \|_{L^2(\Omega)}^2+
C \| \uu\|_{L^3(\Omega)}^2 \| \Delta(\phi-\Phi)\|_{L^2(\Omega)}^2,
\end{split}
\end{equation}
and 
\begin{equation}
\label{corr2}
\begin{split}
\left| \int_{\Omega} (\uu-\UU) \cdot \nabla \Phi \, \Delta^2 (\phi-\Phi) \, \d x \right|
& \leq  \| \uu-\UU \|_{L^2(\Omega)} 
\| \nabla \Phi\|_{L^\infty(\Omega)} \| \Delta^2 (\phi-\Phi)\|_{L^2(\Omega)}\\
&\leq \frac16 \|\Delta^2 (\phi-\Phi)\|_{L^2(\Omega)}^2
+ C \| \uu-\UU \|_{L^2(\Omega)}^2.
\end{split}
\end{equation}
By exploiting \eqref{REG-weakS} and \eqref{REG-strong}, it follows that
\begin{equation}
\label{corr3}
\begin{split}
&\left| \int_{\Omega} \Delta( \Psi'(\phi)-\Psi'(\Phi)) \Delta^2 (\phi-\Phi) \, \d x\right|\\\
&\leq \int_{\Omega} \left| \left( \Psi''(\phi) \Delta (\phi-\Phi) 
+ \left( \Psi''(\phi)-\Psi''(\Phi) \right) \Delta \Phi \right) \Delta^2 (\phi-\Phi) \right| \, \d x \\
& \quad +\int_{\Omega} \left| \left( \Psi'''(\phi) \left( |\nabla \phi|^2- |\nabla \Phi|^2 \right) + \left( \Psi'''(\phi)-\Psi'''(\Phi) \right) |\nabla \Phi|^2  \right) \Delta^2 (\phi-\Phi) \right| \, \d x\\
&\leq C \| \Delta (\phi-\Phi)\|_{L^2(\Omega)} 
\| \Delta^2 (\phi-\Phi) \|_{L^2(\Omega)} \\
&\quad 
+ C \left( \| \Psi'''(\phi)\|_{L^\infty(\Omega)}
+ \| \Psi'''(\Phi)\|_{L^\infty(\Omega)}\right) 
\| (\phi-\Phi)\|_{L^\infty(\Omega)} 
\| \Delta \Phi\|_{L^2(\Omega)} 
\| \Delta^2 (\phi-\Phi)\|_{L^2(\Omega)}\\
&\quad + C \left( \| \nabla \phi\|_{L^\infty)(\Omega)}
+\| \nabla \Phi \|_{L^\infty(\Omega)} \right) \| \nabla (\phi-\Phi)\|_{L^2(\Omega)} \| \Delta^2 (\phi-\Phi) \|_{L^2(\Omega)} \\
&\quad + \left( \| \Psi''''(\phi)\|_{L^\infty(\Omega)}
+ \| \Psi''''(\Phi)\|_{L^\infty(\Omega)}\right) 
\| (\phi-\Phi)\|_{L^\infty(\Omega)} 
\|\nabla \Phi\|_{L^\infty(\Omega)}^2 \| \Delta^2 (\phi-\Phi)\|_{L^2(\Omega)}\\
&\leq \frac16 \|\Delta^2 (\phi-\Phi) \|_{L^2(\Omega)}^2
+ C \| \Delta(\phi-\Phi) \|_{L^2(\Omega)}^2.
\end{split}
\end{equation}
As a consequence, integrating \eqref{phi-diff-1} on $[0,t]$ and using \eqref{corr1}-\eqref{corr3}, we infer that
\begin{equation}
\label{phi-diff-2}
\begin{split}
&\frac12 \| \Delta (\phi(t)-\Phi(t)) \|_{L^2(\Omega)}^2 
+ \frac12 \int_0^t \| \Delta^2 (\phi-\Phi)\|_{L^2(\Omega)}^2 \, \d \tau \\
&\leq 
C \int_0^t \left( 1+ \| \uu\|_{L^3(\Omega)}^2\right) 
\left( \| \uu-\UU \|_{L^2(\Omega)}^2+\| \Delta(\phi-\Phi)\|_{L^2(\Omega)}^2\right) \, \d \tau.
\end{split}
\end{equation}
Then, summing \eqref{Kin-diff-5} and \eqref{phi-diff-2}, and exploiting the assumption on $\nu$, we end up with 
\begin{equation}
\label{DIFF-F}
\begin{split}
&\int_{\Omega}  \frac12 \rho(\phi(t)) |\uu(t)-\UU(t)|^2 \, \d x 
+\frac12 \| \Delta (\phi(t)-\Phi(t)) \|_{L^2(\Omega)}^2 \\
&\quad + \frac{\nu_\ast}{2}\int_0^t \|D(\uu-\UU)\|_{L^2(\Omega)}^2 \, \d \tau
+ \frac14 \int_0^t \| \Delta^2 (\phi-\Phi)\|_{L^2(\Omega)}^2 \, \d \tau \\
& \leq  
C \int_0^t  \mathcal{F}(\tau)
\left( \| (\uu-\UU) \|_{L^2(\Omega)}^2 + \| \Delta (\phi-\Phi) \|_{L^2(\Omega)}^2 \right) \, \d \tau,
\end{split}
\end{equation}
where
$$
\mathcal{F}(t)= 1+ \| \UU(t)\|_{H^2(\Omega)}^2+ \| \partial_t \UU(t)\|_{L^2(\Omega)}^2 + \| \uu(t)\|_{L^3(\Omega)}^2.
$$
Finally, since $\mathcal{F}(t)\in L^1(0,T)$ in light of \eqref{REG-weakS} and \eqref{REG-strong}, the desired conclusion $\uu=\UU$ and $\phi=\Phi$ on $[0,T]$ directly follows from the Gronwall lemma.
\end{proof}

\section{Large time regularity in three dimensions}
\label{Reg-3D}
\setcounter{equation}{0}

This section is devoted to the third part of our main result, namely the large time regularity of the velocity $\uu$. {The proof is based on a continuation argument for some Sobolev norms of $\uu$ when the initial data and forcing terms are {\it small}, for which we will make use of Corollary \ref{HREG-MU}. Notice that, in the case of matched densities, a general result has been formulated in \cite[Theorem 8]{Abels2009} for the Navier-Stokes equation with non-constant viscosity. Instead, for unmatched densities, we are forced to consider the full system.}

\begin{proof}[Proof of Theorem \ref{MAIN} - (iii) Large time regularity of the velocity] Let $\Omega$ be a bounded domain in $\R^d$, $d=2,3$ with $\partial \Omega$ of class $C^4$. Consider a global weak solution $(\uu,\phi)$ to \eqref{AGG}-\eqref{AGG-bc} given by Theorem \ref{WEAK-SOL}. We first summarize from the previous parts (i)-(ii) of Theorem \ref{MAIN} that, for any $\tau>0$,  
\begin{equation}
\label{REG-weak3D}
\begin{split}
&\phi \in L^\infty(\tau ,\infty;W^{2,p}(\Omega)), \quad \partial_t \phi \in L^2(\tau,\infty;H^1(\Omega)),\\
&\mu \in L^{\infty}(\tau,\infty; H^1(\Omega))\cap L_{\uloc}^2( [\tau,\infty) ;H^3(\Omega)),
\quad F'(\phi) \in L^\infty(\tau,\infty; L^p(\Omega)),
\end{split}
\end{equation}
where $2\leq p <\infty$ if $d=2$ and $p=6$ if $d=3$. Furthermore, the energy inequality \eqref{energy-ineq} entails that
\begin{equation}
\label{Diss}
\int_0^\infty \| \uu(s)\|_{\H^1_{0,\sigma}(\Omega)}^2 + \| \nabla \mu (s)\|_{L^2(\Omega)}^2 \, \d s \leq C_0,
\end{equation}
for some positive constant $C_0$ depending only on $E(\uu_0,\phi_0)$ and $\Omega$, and the strict separation property
\begin{equation}
\label{SP}
\max_{x \in \overline{\Omega}} \left| \phi(x,t) \right|\leq 1-\delta, \quad \forall \, t \geq T_{SP}
\end{equation}
holds for some $\delta \in (0,1)$ and $T_{SP}>0$.
It immediately follows from \eqref{SP} that $F''(\phi)$ and $F'''(\phi)$ are globally bounded in $[T_{SP}, \infty)$. As a consequence, it is easily seen from \eqref{AGG}$_4$ and \eqref{REG-weak3D} that $\partial_t \mu \in L^2(T_{SP},\infty; H^1(\Omega)')$, which, in turn, gives $\mu \in BC([T_{SP}, \infty); H^1(\Omega))$. Then, we infer from the regularity theory of the Laplace equation with homogeneous Neumann boundary condition that $\phi \in L^\infty(T_{SP},\infty; H^3(\Omega))\cap L_{\uloc}^2([T_{SP};\infty); H^4(\Omega))$. Finally, we observe from the achieved regularity that $\partial_\n \phi(t)=0$ on $\partial \Omega$ for any $t\geq T_{SP}$.

Let us now fix $0<\widetilde{\varepsilon}<\varepsilon<1$, whose specific (small) values will be chosen later on. Thanks to \eqref{Diss}, there exists a positive time $\TT=\TT(\widetilde{\varepsilon},\varepsilon)$  (which can be taken, without loss of generality, larger than $T_{SP}$) such that 
\begin{equation}
\label{Small-u0}
\| \uu(\TT)\|_{\H_{0,\sigma}^1(\Omega)}\leq \varepsilon \quad \text{and} \quad
\left( \int_\TT^\infty  \| \uu(s)\|_{\H^1_{0,\sigma}(\Omega)}^2 \, \d s \right)^\frac12 \leq \widetilde{\varepsilon}.
\end{equation}
Furthermore, we infer from the above regularity properties that $\phi(\TT) \in H^2(\Omega)$,   $\| \phi(\TT)\|_{L^\infty(\Omega)}< 1$, $|\overline{\phi(\TT)}|<1$, $\mu(\TT)=-\Delta \phi(\TT)+\Psi'(\phi(\TT)) \in H^1(\Omega)$ and $\partial_\n \phi(\TT)=0$ on $\partial \Omega$. Exploiting once again \eqref{Diss}, up to a possible redefinition of $\TT$, we also have
\begin{equation}
\label{Small-mu0}
\| \nabla \mu(\TT)\|_{L^2(\Omega)}\leq \widetilde{\varepsilon} \quad \text{and} \quad
\left( \int_\TT^\infty \| \nabla \mu(s)\|_{L^2(\Omega)}^2 \, \d s \right)^\frac12  \leq \widetilde{\varepsilon}.
\end{equation}
Thanks to \eqref{N-phit} of Theorem \ref{CH-strong} and Corollary \ref{HREG-MU}, we obtain the following estimates
 \begin{equation}
 \label{int-phit-sm}
\begin{split}
\int_\TT^\infty \| \nabla \partial_t \phi(s)\|_{L^2(\Omega)}^2 \, \d s
&\leq 
C \left(  \| \nabla \mu(\TT)\|_{L^2(\Omega)}^2+ \int_\TT^\infty \| \nabla \uu(s)\|_{L^2(\Omega)}^2 \, \d s\right) \\
&\quad \times \left( 1+\left( \int_\TT^\infty \| \nabla \uu(s)\|_{L^2(\Omega)}^2 \, \d s \right) \ \mathrm{exp}\left( 2C \int_\TT^\infty \| \nabla \uu(s)\|_{L^2(\Omega)}^2 \, \d s \right) \right)\\
&\leq 2 C \widetilde{\varepsilon}^2 (1+\widetilde{\varepsilon}^2) \mathrm{e}^{2C \widetilde{\varepsilon}}
= \widetilde{K_1} \widetilde{\varepsilon}^2
\end{split}
\end{equation}
and 
\begin{equation}
\label{int-mu-sm}
\begin{split}
\int_\TT^\infty \| \nabla \mu(s)\|_{H^2(\Omega)}^2 \, \d s
&\leq 
C \left(  \| \nabla \mu(\TT)\|_{L^2(\Omega)}^2+ \int_\TT^\infty \| \nabla \uu(s)\|_{L^2(\Omega)}^2 \, \d s +  \int_\TT^\infty \| \nabla \mu(s)\|_{L^2(\Omega)}^2 \, \d s \right) \\
&\quad \times \left( 1+\left( \int_\TT^\infty \| \nabla \uu(s)\|_{L^2(\Omega)}^2 \, \d s \right) \ \mathrm{exp}\left( 2C \int_\TT^\infty \| \nabla \uu(s)\|_{L^2(\Omega)}^2 \, \d s \right) \right)\\
&\leq 3 C \widetilde{\varepsilon} \left( 1+\widetilde{\varepsilon} \right) \mathrm{e}^{2C \widetilde{\varepsilon}}
= \widetilde{K_2} \widetilde{\varepsilon}^2,
\end{split}
\end{equation}
where $\widetilde{K_1}$ and $\widetilde{K_2}$ are two positive constants depending only on $\Omega$, $\theta$, $\theta_0$, $F$, and $\overline{\phi_0}$.

Next, by \cite[Theorem 1.1]{G2021} there exists a local strong solution $(\UU, \Pi, \Phi)$ defined on the maximal interval $[\TT, \TT_{\rm{max}})$ such that
\begin{equation}
\label{REG-strong-2}
\begin{split}
&\UU \in C([\TT,\wTT]; \H^1_{0,\sigma}(\Omega)) \cap 
L^2(\TT,\wTT;\H^2(\Omega))\cap W^{1,2}(\TT,\wTT;\L^2_\sigma(\Omega)),\quad \Pi \in L^2(\TT,\wTT;H^1(\Omega)), \\
&\Phi \in L^\infty(\TT,\wTT;H^3(\Omega))\cap L^2(\TT,\wTT ; H^4(\Omega)), \quad \partial_t \Phi \in L^2(\TT,\wTT;H^1(\Omega)),\\
&\Phi \in C( \overline{\Omega}\times [\TT,\wTT]) \quad\text{with}\quad \max_{t \in [\TT,\wTT]} \|\Phi(t)\|_{C(\overline{\Omega})}<1, \\
&M= -\Delta \Phi+\Psi'(\Phi) \in C([\TT,\wTT]; H^1(\Omega))\cap L^2(\TT,\wTT;H^3(\Omega)),
\end{split}
\end{equation}
for any $\TT\leq \wTT <\TT_{\rm{max}}$,
which satisfies \eqref{AGG} almost everywhere in $(\TT,\TT_{\rm{max}})$ and the energy identity in $[\TT,\wTT]$ for any $\TT\leq \wTT <\TT_{\rm{max}}$ (cf. Remark \ref{rem:EnergyIdentity}). On the other hand, in light of Theorem \ref{weak-strong}, we deduce that $(\uu,\phi)=(\UU,\Phi)$ (and so $P=\Phi$) in the interval $[\TT,\TT_{\rm{max}})$. This implies that 
\begin{equation}
\label{u-strong-loc}
\uu \in C([\TT,\wTT]; \H^1_{0,\sigma}(\Omega))\cap L^2(\TT,\wTT;\H^2(\Omega))\cap W^{1,2}(\TT,\wTT;\L^2_\sigma(\Omega)),\quad \forall \, \TT\leq \wTT<\TT_{\rm{max}},
\end{equation}
and \eqref{AGG}$_1$ is solved almost everywhere in $(\TT,\TT_{\rm{max}})$.

We are left to show that $\TT_{\rm{max}}= \infty$ provided that $\varepsilon$, $\widetilde{\varepsilon}$ and $\TT$ are suitably chosen. To this end, we notice that the regularity properties in \eqref{REG-weak3D} and the separation property in \eqref{SP} hold globally in time (from $T_{SP}$ on), and so they are  independent of $\TT_{\rm{max}}$. Thus, we only need to prove that the norms of the spaces in \eqref{u-strong-loc} do not blow up as $\wTT$ is replaced by $\TT_{\rm{max}}$. We now write the momentum equation \eqref{AGG}$_1$ as follows
\begin{equation}
 \label{NS}
\rho (\phi) \partial_t \uu + \rho(\phi) \left( \uu \cdot \nabla \right) \uu -
 \rho'(\phi) \left( \nabla \mu \cdot \nabla \right) \uu 
- \div \left( \nu(\phi) D \uu \right) +\nabla P^\ast= \f, \quad \text{in } \, \Omega \times (\TT,\TT_{\rm{max}}),
\end{equation}
where 
$$
P^\ast= P + \frac12 |\nabla \phi|^2 + \Psi(\phi) -\mu \phi
\quad \text{and} \quad \f= - \phi \nabla \mu.
$$
Multiplying \eqref{NS} by $\partial_t \uu$ and integrating over $\Omega$, we find
\begin{equation}
\label{v-test1}
\begin{split}
&\ddt \int_{\Omega} \frac{\nu(\phi)}{2} | D \uu|^2 \, \d x + \frac{\rho_\ast}{2} \| \partial_t \uu\|_{L^2(\Omega)}^2 \\
&\quad \leq C \left( \| \f\|_{L^2(\Omega)}^2+ \| \left( \uu \cdot \nabla \right) \uu \|_{L^2(\Omega)}^2 + \| \left( \nabla \mu \cdot \nabla \right) \uu\|_{L^2(\Omega)}^2 + \int_{\Omega} |\partial_t \phi| | D \uu|^2 \, \d x \right),
\end{split}
\end{equation}
for some positive constant $C$ depending only on $\rho_1$, $\rho_2$, $\nu_1$ and $\nu_2$. 
Besides, by the regularity theory of the Stokes operator with variable viscosity (see \cite[Lemma 4]{Abels2009}), we have 
\begin{equation}
\label{v-H2}
\| \uu\|_{H^2(\Omega)}^2 \leq C \left( \| \f\|_{L^2(\Omega)}^2 + \| \left( \uu \cdot \nabla \right) \uu \|_{L^2(\Omega)}^2 + \| \left( \nabla \mu \cdot \nabla \right) \uu\|_{L^2(\Omega)}^2 +\| \partial_t\uu\|_{L^2(\Omega)}^2 \right),
\end{equation}
for some positive constant $C$ depending on $\rho_1$, $\rho_2$, $\nu_1$, $\nu_2$, $\Omega$ and $\| \phi\|_{BC([\TT,\infty); W^{1,4}(\Omega))}$. 
For $\vartheta>0$, we observe that 
\begin{equation}
\begin{split}
\| \left( \uu \cdot \nabla \right) \uu \|_{L^2(\Omega)}^2 
&\leq \| \uu\|_{L^6(\Omega)}^2 \| \nabla \uu\|_{L^3(\Omega)}^2\\
&\leq C \| \nabla \uu\|_{L^2(\Omega)}^3 \| \uu\|_{H^2(\Omega)}\\
&\leq \vartheta \| \uu\|_{H^2(\Omega)}^2 + C_\vartheta \| \nabla \uu\|_{L^2(\Omega)}^6,
\end{split}
\end{equation}
and
\begin{equation}
\begin{split}
\| \left( \nabla \mu \cdot \nabla \right) \uu\|_{L^2(\Omega)}^2
& \leq \| \nabla \uu\|_{L^3(\Omega)}^2 \| \nabla \mu\|_{L^6(\Omega)}^2\\
&\leq C \| \nabla \uu\|_{L^2(\Omega)} \| \uu\|_{H^2(\Omega)} \| \nabla \mu\|_{L^2(\Omega)} \| \nabla \mu\|_{H^2(\Omega)}\\
&\leq  \vartheta \| \uu\|_{H^2(\Omega)}^2 + C_\vartheta \| \nabla \mu \|_{L^2(\Omega)}^2 \| \nabla \mu \|_{H^2(\Omega)}^2 \| D \uu \|_{L^2(\Omega)}^2,
\end{split}
\end{equation}
as well as
\begin{equation}
\begin{split}
\int_{\Omega} |\partial_t \phi| | D \uu|^2 \, \d x 
& \leq \| \partial_t \phi \|_{L^3(\Omega)} \| D \uu\|_{L^3(\Omega)}^2 \\
&\leq C \| \nabla \partial_t \phi\|_{L^2(\Omega)} \| D \uu \|_{L^2(\Omega)} \| \uu\|_{H^2(\Omega)}\\
& \leq \vartheta \| \uu\|_{H^2(\Omega)}^2 + C_\vartheta \| \nabla \partial_t \phi\|_{L^2(\Omega)}^2 \| D \uu \|_{L^2(\Omega)}^2,
\end{split}
\end{equation}
for some positive constant $C_\vartheta$ depending only on $\vartheta$ and $\Omega$. 
Combining the above estimates and choosing $\vartheta$ sufficiently small, we arrive at the differential inequality
\begin{equation}
\label{v-di}
\begin{split}
&\ddt \int_{\Omega} \frac{\nu(\phi)}{2} | D \uu|^2 \, \d x + \frac{\rho_\ast}{4} \| \partial_t \uu\|_{L^2(\Omega)}^2 + \varpi \| \uu\|_{H^2(\Omega)}^2 \\
&\quad \leq C_R \| \f\|_{L^2(\Omega)}^2+ C_R \left( \| \nabla \uu\|_{L^2(\Omega)}^4+ \| \nabla \partial_t \phi\|_{L^2(\Omega)}^2 + \| \nabla \mu \|_{H^2(\Omega)}^2  \right) \| \nabla \uu\|_{L^2(\Omega)}^2,
\end{split}
\end{equation}
where $\varpi$ is a (possibly small) positive constant and $C_R$ is a (possibly large) positive constant. They depends on $\| \phi\|_{BC( [\TT,\infty); W^{1,4}(\Omega))}$, $\| \nabla \mu\|_{L^\infty(\TT, \infty; L^2(\Omega))}$, $\rho_1$, $\rho_2$, $\nu_1$, $\nu_2$ and $\Omega$, but are independent of $\varepsilon$, $\widetilde{\varepsilon}$ and $\TT$.
Let us now define 
\begin{equation}
\label{T_star}
\mathcal{T}_\ast= \max \left\lbrace t\in [\TT,\TT_{\rm{max}}): \| \uu\|_{C([\TT,t]; \H_{0,\sigma}^1(\Omega))} \leq  \sqrt{\frac{ 4 \nu^\ast}{\nu_\ast}} \varepsilon \right\rbrace.
\end{equation}
Since $\sqrt{\frac{4 \nu^\ast}{\nu_\ast}}>1$, it is easily seen that $\mathcal{T}_\ast> \TT$. We choose $\varepsilon\in (0,1) $ such that
\begin{equation}
\label{eps-def}
C_R \left(  \sqrt{\frac{ 4 \nu^\ast}{\nu_\ast}} \varepsilon \right)^4 \leq \frac{\varpi}{4}.
\end{equation}
Recalling that 
\begin{equation}
\label{U-H1-eq}
\frac{\nu_\ast}{4} \| \nabla \uu\|_{L^2(\Omega)}^2 \leq \int_{\Omega} \frac{\nu(\phi)}{2} | D \uu|^2 \, \d x \leq \frac{\nu^\ast}{2} \| \nabla \uu\|_{L^2(\Omega)}^2,
\end{equation}
we infer from \eqref{v-di} and \eqref{eps-def} that, for almost all $t\in [\TT,\mathcal{T}_\ast]$,
\begin{equation}
\label{v-di-2}
\begin{split}
&\ddt \left( \int_{\Omega} \frac{\nu(\phi)}{2} | D \uu|^2 \, \d x \right)
+ \frac{\varpi}{4 \nu^\ast} \left( \int_{\Omega} \frac{\nu(\phi)}{2} | D \uu|^2 \, \d x \right)
+ \frac{\rho_\ast}{4} \| \partial_t \uu\|_{L^2(\Omega)}^2+\frac{\varpi}{2} \| \uu\|_{H^2(\Omega)}^2 \\
&\quad \leq C_R \| \f\|_{L^2(\Omega)}^2+ C_R \left( \| \nabla \partial_t \phi\|_{L^2(\Omega)}^2 + \| \nabla \mu \|_{H^2(\Omega)}^2  \right) \left( \sqrt{\frac{4 \nu^\ast}{\nu_\ast}} \varepsilon\right)^2.
\end{split}
\end{equation}
Then, integrating in time \eqref{v-di-2}, multiplying by $\frac{4}{\nu_\ast}$ and using \eqref{Small-mu0}, \eqref{int-phit-sm}, \eqref{int-mu-sm} and \eqref{U-H1-eq}, we obtain that
\begin{equation}
\label{main-arg}
\begin{split}
&\| \nabla \uu (t)\|_{L^2(\Omega))}^2 
+ \frac{\rho_\ast}{\nu_\ast} \int_\TT^t \| \partial_t \uu(s)\|_{L^2(\Omega)}^2\, \d s
+ \frac{2 \varpi}{\nu_\ast} \int_\TT^t \| \uu(s)\|_{H^2(\Omega)}^2\, \d s\\
&\quad \leq 
\frac{2 \nu^\ast}{\nu_\ast} \| \nabla \uu(\TT)\|_{L^2(\Omega)}^2 
+ \frac{4 C_R}{\nu_\ast} \| \f \|_{L^2(\TT,\infty; L^2(\Omega))}^2 \\
&\qquad + \frac{16 C_R \nu^\ast}{\nu_\ast^2}  \varepsilon^2 \left( 
\int_\TT^\infty \| \nabla \partial_t \phi(s)\|_{L^2(\Omega)}^2 \, \d s + \int_\TT^\infty \| \nabla \mu (s) \|_{H^2(\Omega)}^2 \, \d s \right)
\\
&\quad \leq \frac{2 \nu^\ast}{ \nu_\ast}\varepsilon^2 + \frac{4 C_R}{\nu_\ast} \int_\TT^\infty \| \nabla \mu(s)\|_{L^2(\Omega)}^2 \, \d s \\
&\qquad + \frac{16 C_R \nu^\ast}{\nu_\ast^2}  \varepsilon^2 \left( 
\int_\TT^\infty \| \nabla \partial_t \phi(s)\|_{L^2(\Omega)}^2 \, \d s + \int_\TT^\infty \| \nabla \mu (s) \|_{H^2(\Omega)}^2 \, \d s \right)
\\
&\quad \leq \frac{2 \nu^\ast}{ \nu_\ast}\varepsilon^2 
+ \frac{4 C_R}{\nu_\ast} \widetilde{\varepsilon}^2 
+ \frac{16 C_R \nu^\ast}{\nu_\ast^2}  (\widetilde{K_1}+\widetilde{K_2}) \widetilde{\varepsilon}^2,
\end{split}
\end{equation}
for any $t \in [\TT, \TT_\ast]$.
Setting now $\widetilde{\varepsilon}\in (0,\varepsilon)$ such that
\begin{equation}
\label{eps2-def}
\left[ \frac{4 C_R}{\nu_\ast}+ \frac{16 C_R \nu^\ast}{\nu_\ast^2}(\widetilde{K_1}+\widetilde{K_2})\right] \widetilde{\varepsilon}^2< \frac{2\nu^\ast}{\nu_\ast} \varepsilon^2,
\end{equation}
we eventually infer that
\begin{equation}
\max_{t \in [\TT, \TT_\ast]} \| \uu(t)\|_{\H^1_{0,\sigma}(\Omega)} <  \sqrt{\frac{ 4 \nu^\ast}{\nu_\ast}} \varepsilon.
\end{equation}
In particular, $\| \uu (\mathcal{T}_\ast)\|_{\H^1_{0,\sigma}(\Omega)} < \sqrt{\frac{ 4 \nu^\ast}{\nu_\ast}} \varepsilon$. However, by continuity, there exists 
$\widetilde{\tau}>0$ such that $\| \uu (t)\|_{\H^1_{0,\sigma}(\Omega)} \leq  \sqrt{\frac{ 4 \nu^\ast}{\nu_\ast}} \varepsilon$ in $[\mathcal{T}_\ast, \mathcal{T}_\ast+\widetilde{\tau}]$. Thus, we found a contradiction with the definition of $\mathcal{T}_\ast$ in \eqref{T_star}. We conclude that 
\begin{equation}
\label{u-H1-small}
\| \uu\|_{C([\TT,\widetilde{\TT}]; \H_{0,\sigma}^1(\Omega))} \leq  \sqrt{\frac{ 4 \nu^\ast}{\nu_\ast}} \varepsilon, \quad \forall \, \TT\leq \wTT<\TT_{\rm{max}}.
\end{equation}
In addition, exploiting \eqref{main-arg} and \eqref{eps2-def}, we are led to
\begin{equation}
\label{u-H2-small}
 \| \uu\|_{L^2(\TT,\widetilde{\TT};\H^2(\Omega))}
+ \| \partial_t \uu \|_{L^2(\TT,\widetilde{\TT} ;\L_\sigma^2(\Omega))} \leq \frac{\sqrt{\frac{4\nu^\ast}{\nu_\ast} \varepsilon^2}}{\min \left\lbrace \frac{\rho_\ast}{\nu_\ast}, \frac{2\varpi}{\nu_\ast} \right\rbrace}, \quad \forall \, \, \TT\leq \wTT<\TT_{\rm{max}}.
\end{equation}
In light of \eqref{u-H1-small} and \eqref{u-H2-small}, we have that
$$
\limsup_{t \rightarrow \TT_{\rm{max}}} \left( \| \uu\|_{C([\TT,t ]; \H_{0,\sigma}^1(\Omega))} + \| \uu\|_{L^2(\TT,t ;\H^2(\Omega))} +  \| \partial_t \uu \|_{L^2(\TT,t ;\L_{\sigma}^2(\Omega))} \right)<\infty.
$$
As a result, by a classical argument, it is possible to define $\uu(\TT_{\rm{max}}) \in \H^1_{0,\sigma}$. Since $\phi(\TT_{\rm{max}}) \in H^2(\Omega)$,   $\| \phi(\TT_{\rm{max}})\|_{L^\infty(\Omega)}< 1$, $|\overline{\phi(\TT_{\rm{max}})}|<1$, $\mu(\TT_{\rm{max}}) \in H^1(\Omega)$ and $\partial_\n \phi(\TT_{\rm{max}})=0$ on $\partial \Omega$, a further application of \cite[Theorem 1.1]{G2021} and Theorem \ref{weak-strong} ensure the existence of a strong solution beyond the maximal time $\TT_{\rm{max}}$, which contradicts with the definition of $\TT_{\rm{max}}$. Hence, $(\uu, P, \phi)$ is a strong solution defined on the whole interval $[\TT,\infty)$ such that
\begin{equation}
\| \uu\|_{L^\infty(\TT,\infty; \H^1_{0,\sigma}(\Omega))}
+ \| \uu\|_{L^2(\TT,\infty;\H^2(\Omega))}
+ \| \partial_t \uu \|_{L^2(\TT,\infty ;\L_\sigma^2(\Omega))}
\leq 2 \max \left\lbrace \sqrt{\frac{ 4 \nu^\ast}{\nu_\ast}} \varepsilon, 
\frac{\sqrt{\frac{4\nu^\ast}{\nu_\ast} \varepsilon^2}}{\min \left\lbrace \frac{\rho_\ast}{\nu_\ast}, \frac{2\varpi}{\nu_\ast} \right\rbrace} \right\rbrace.
\end{equation}
Therefore, the desired claim in Theorem \ref{MAIN} - (iii) follows with $T_R= \TT(\varepsilon)$ where $\varepsilon \in (0,1)$ is given by \eqref{eps-def}. 
\end{proof}

\section{Convergence to equilibrium}
\label{Conv-eq}
\setcounter{equation}{0}

In this section, we complete the proof of Theorem \ref{MAIN} by showing that each weak solution converges to an equilibrium (stationary state) of the system \eqref{AGG}-\eqref{AGG-bc}.

\begin{proof}[Proof of Theorem \ref{MAIN} - (iv) Convergence to a stationary solution] 
Let us first recall (cf. Theorem \ref{MAIN} - (ii)) that 
\begin{equation*}
|\phi(x,t)|\leq 1-\delta, \quad \forall \, (x,t) \in \overline{\Omega}\times [T_{SP},\infty).
\end{equation*}
Now let $\widetilde{\Psi}$ be the smooth and bounded function such that
$\widetilde{\Psi}|_{[1+\delta,1-\delta]}=\Psi|_{[-1+\delta,1-\delta]}$, and
\begin{equation*}
\widetilde{E}_{\text{free}}(\varphi):= \int_\Omega \frac12 |\nabla \varphi|^2 +  \widetilde{\Psi}(\varphi)\, \d x,  
\end{equation*}
for all $\varphi \in H^1_{(m)}(\Omega)$ with $|\varphi(x)|\leq 1$ for almost every $x\in\Omega$. Then, we report the following result whose proof can be found in \cite[Proposition 6.1]{AW2007}.

\begin{proposition}[Lojasiewicz-Simon inequality]\label{prop:LS}
Let $\phi'\in \mathcal{D}(\partial E_0)$ be a solution to (\ref{eq:StatCH1})-(\ref{eq:StatCH3}).
Then, there exist three constants $\theta\in \left(0,\frac{1}{2}\right], C, \kappa>0$ such that
\begin{equation}
  \label{eq:LS}
\left|\widetilde{E}_{\mathrm{free}}(\varphi)-\widetilde{E}_{\mathrm{free}}(\phi')
\right|^{1-\theta}\le C
\left\| D\widetilde{E}_{\mathrm{free}}(\varphi) \right\|_{H^1_{(0)}(\Omega)'}, \quad \forall \, \varphi \in H^1_{(m)}(\Omega):  \|\varphi-\phi'\|_{\Hone(\Omega)}\leq\kappa,
\end{equation}
where $D\widetilde{E}_{\mathrm{free}}\colon H^1_{(m)}(\Omega)\to 
{H^1_{(0)}(\Omega)'}$ denotes the Frech\'et derivative of $\widetilde{E}_{\mathrm{free}}\colon H^1_{(m)}(\Omega)\to \R$.
\end{proposition}

Next, we note that the only critical points of the energy
\begin{equation*}
E_{\text{kin}}(\vv,\varphi)= \int_\Omega \frac12 \rho(\varphi)|\vv(x)|^2\, \d x, \quad \forall \, \vv\in \L^2_\sigma(\Omega), \, \varphi \in  L^\infty(\Omega) \, 
\text{ with } \, \|\varphi\|_{L^\infty(\Omega)}<1
\end{equation*}
are $(\vv,\varphi)$ with $\vv =\mathbf{0}$ and $\varphi$ arbitrary. Moreover,  obviously
\begin{equation*}
\left| E_{\text{kin}}(\vv,\varphi)-E_{\text{kin}}(\mathbf{0},\phi') \right|^\frac12\leq \sqrt{\frac{\rho^\ast}{2}} \left\|\vv\right\|_{L^2(\Omega)}, \quad \forall \, \vv\in \L^2_\sigma(\Omega), \, \varphi \in  L^\infty(\Omega) \, 
\text{ with } \, \|\varphi\|_{L^\infty(\Omega)}<1.
\end{equation*}
Hence, by Proposition \ref{prop:LS}, for any $\phi'\in \mathcal{D}(\partial E_0)$ solution to \eqref{eq:StatCH1}-\eqref{eq:StatCH3} and any $R>0$, there exist three constants $\theta\in \left( 0,\frac{1}{2}\right], C, \kappa>0$ such that $\widetilde{E}(\vv,\varphi)=E_{\text{kin}}(\vv,\varphi)+ \widetilde{E}_{\text{free}}(\varphi)$ satisfies
\begin{eqnarray}
\left| \widetilde{E}(\vv,\varphi)-\widetilde{E}(\mathbf{0},\phi') \right|^{1-\theta}&\leq& C\left( \left\| D\widetilde{E}_{\text{free}}(\varphi) \right\|_{{H^1_{(0)}(\Omega)'}}+ \|\vv\|_{L^2(\Omega)}\right)   \label{eq:LS2}
\end{eqnarray}
for all $\|\varphi-\phi'\|_{\Hone}\leq\kappa$,  and $\vv\in \L^2_\sigma(\Omega)$ with $\|\vv\|_{L^2(\Omega)}\leq R$. 

Since $\omega(\uu,\phi)$ is a compact subset of 
$\L^2_\sigma(\Omega)\times W^{2-\eps,p}(\Omega)$ for every $\eps>0$, and because of Lemma~\ref{lem:OmegaLimit}, for every $R>0$ there exist three universal constants $\theta\in \left( 0,\frac{1}{2}\right], C, \kappa>0$ such that \eqref{eq:LS2} holds for all $(\mathbf{0},\phi') \in \omega(\uu,\phi)$, $\varphi \in 
H_{(m)}^1 (\Omega)$ such that $\| \varphi-\phi'\|_{\Hone}\leq\kappa$, and $\vv\in \L^2_\sigma(\Omega)$ with $\|\vv\|_{L^2(\Omega)}\leq R$. Furthermore, by choosing $T_C\geq \max \lbrace T_{SP}, T_R \rbrace$ sufficiently large, we obtain $\dist((\uu(t),\phi(t)), \omega(\uu,\phi))\leq \kappa$ with respect to the norm of $\L^2_\sigma(\Omega)\times H^1(\Omega)$ for all $t\geq T_C$. Thus, we infer that
\begin{equation*}
\left| \widetilde{E}(\uu(t),\phi(t))-\widetilde{E}(\mathbf{0},\phi') \right|^{1-\theta}\leq C\left( \left\| D\widetilde{E}_{\text{free}}(\phi(t)) \right\|_{{H^1_{(0)}(\Omega)'}}+ \|\uu(t)\|_{L^2(\Omega)}\right), \ \forall \, t \geq T_C, \forall \, (\mathbf{0},\phi') \in \omega(\uu,\phi). 
\end{equation*}
On the other hand, since $\omega(\uu,\phi)$ is connected, $\widetilde{E}(\mathbf{0},\phi')$ is independent of $\phi'$ with $(\mathbf{0},\phi')\in \omega(\uu,\phi))$. Thus, we set such a value as $E_\infty$, and we conclude that
\begin{equation}
\label{LS4}
\left|\widetilde{E}(\uu(t),\phi(t))-E_\infty \right|^{1-\theta}\leq C\left( \left\| D\widetilde{E}_{\text{free}}(\phi(t)) \right\|_{H_{(0)}(\Omega)'}+ \|\uu(t)\|_{L^2(\Omega)}\right), \quad \forall \, t \geq T_C.
\end{equation}

Lastly, in order to prove the convergence as $t\to \infty$, we consider
\begin{equation*}
  H(t):= \left(E(\uu(t),\phi(t))-E_\infty\right)^\theta,
\end{equation*}
where $\theta$ is as in \eqref{LS4}. By the energy identity \eqref{EE-strong}, which holds due to Remark~\ref{rem:EnergyIdentity} and Theorem~\ref{MAIN} - (iii), $H(t)$ is non-increasing and
\begin{align*}
- \ddt H(t)
&= \theta \left( \|\nabla \mu(t)\|_{L^2(\Omega)}^2+\int_\Omega \nu(\phi(t))|D\uu(t)|^2 \, \d x\right) \left(E(\uu(t),\phi(t))-E_\infty\right)^{\theta-1}
  \\
& \geq \theta C^{-1} \left(\|\nabla \mu(t)\|_2^2+\int_\Omega \nu(\phi(t))|D\uu(t)|^2 \, \d x\right)\left( \left\| D\widetilde{E}_{\text{free}}(\phi(t)) \right\|_{{H^1_{(0)}(\Omega)'}} + \|\uu(t)\|_{L^2(\Omega)} \right)^{-1}
\end{align*}
for all $t\geq T_C$. 
Now we use that
\begin{equation*}
  D\widetilde{E}_{\text{free}}(\phi(t)) = -\Delta \phi(t) + P_0 \widetilde{\Psi}'(\phi(t)) =
 -\Delta \phi(t) + P_0 \Psi'(\phi(t)) = P_0\mu(t), \quad \forall \, t \geq T_C,
\end{equation*}
since $|\phi(t)|\leq 1-\delta$. In light of
$\|P_0\mu(t)\|_{{H^1_{(0)}(\Omega)'}}\leq C\|\nabla \mu(t)\|_{L^2(\Omega)}$ and $\|\uu(t)\|_{L^2(\Omega)}\leq C \|D\uu(t)\|_{L^2(\Omega)}$ by Korn's inequality,
$
- \ddt H(t) \geq C \left(\|\nabla \mu(t)\|_{L^2(\Omega)}+\|D\uu(t)\|_{L^2(\Omega)}\right),
$
for some positive constant $C$.
In turn, this implies that
\begin{equation*}
\int_{T_C}^\infty\|\nabla \mu(t)\|_{L^2(\Omega)}\, \d t 
+   \int_{T_C}^\infty\|D\uu(t)\|_{L^2(\Omega)}\, \d t \leq CH(T_C)<\infty.
\end{equation*}
Therefore, by \eqref{phit-est}, we deduce that
\begin{equation*}
\int_T^\infty  \|\partial_t \phi(t)\|_{H^1(\Omega)'} \, \d t 
\leq C H(T_C)<\infty,
\end{equation*}
which entails that $\partial_t \phi \in L^1(T_C,\infty;H^1(\Omega)')$.
Thus, we infer that
\begin{equation*}
\phi(t) = \phi(T_C)+ \int_{T_C}^t \partial_t \phi(\tau)\, \d \tau\to_{t\to\infty} \phi_\infty \qquad \text{in}\ H^1(\Omega)'.  
\end{equation*}
In particular, $\omega(\uu,\phi)= \{(0,\phi_\infty)\}$ and $\phi_\infty$ solves the stationary Cahn-Hilliard equation (\ref{eq:StatCH1}) -(\ref{eq:StatCH3}) thanks to Lemma~\ref{lem:OmegaLimit}.
Since $(\uu,\phi)\in BC( [T_R, \infty); \H_{0,\sigma}^1(\Omega) \times W^{2,p}(\Omega)$, for any $2\leq p<\infty$ if $d=2$ and $p=6$ if $d=3$, we conclude that $(\uu(t),\phi(t))$ converges weakly to
$(0,\phi_\infty)$ in $\H_{0,\sigma}^{1}(\Omega) \times W^{2,p}(\Omega)$.
This finally proves Theorem~\ref{MAIN} - (iv).
\end{proof}

\section{Global regularity for the AGG model in two dimensions}
\label{Reg-2D}
\setcounter{equation}{0}

In this section we prove the global well-posedness of the AGG model in any generic two dimensional bounded domain. We recall that the local well-posedness has been proven in \cite[Theorem 3.1]{Gior2021}. We show herein that the local strong solutions to system \eqref{AGG}-\eqref{AGG-bc} are in fact globally defined.

\begin{theorem}
\label{Glob-2D}
Let $\Omega$ be a bounded domain of class $C^3$ in $\mathbb{R}^2$. Assume that $\uu_0 \in \H^1_{0,\sigma}(\Omega)$ and $\phi_0 \in H^2(\Omega)$ such that $\| \phi_0\|_{L^\infty(\Omega)}\leq 1$, $|\overline{\phi_0}|<1$, $\mu_0= -\Delta \phi_0+ \Psi'(\phi_0) \in H^1(\Omega)$, and $\partial_\n \phi_0=0$ on $\partial \Omega$. Then, there exists a unique global strong solution $(\uu, P, \phi)$ to system \eqref{AGG}-\eqref{AGG-bc} defined on $\Omega \times [0,\infty)$ in the following sense:
\begin{itemize}
\item[(i)] The solution $(\uu, P, \phi)$ satisfies 
\begin{align}
\label{reg-SS}
\begin{split}
&\uu \in BC([0,\infty); \H^1_{0,\sigma}(\Omega)) 
\cap L^2(0,\infty;\H^2(\Omega))
\cap W^{1,2}(0,\infty;\L^2_\sigma(\Omega)), 
\\
& P \in L_{\uloc}^2([0,\infty);H^1(\Omega)),\\
&\phi \in L^\infty(0,\infty;H^3(\Omega)), \
\partial_t \phi \in L^\infty(0,\infty; H^1(\Omega)')
\cap L^2(0,\infty;H^1(\Omega)),\\
&\phi \in L^\infty(\Omega\times (0,\infty)) \text{ with } |\phi(x,t)|<1 \ \text{a.e. in } \  \Omega\times(0,\infty),\\
&\mu \in BC([0,\infty);H^1(\Omega))\cap L_{\uloc}^2([0,\infty);H^3(\Omega))\cap W_{\uloc}^{1,2}([0,\infty); H^1(\Omega)'), \\
&F'(\phi), F''(\phi), F'''(\phi) \in L^\infty(0,\infty;L^p(\Omega)), \ \forall \, p \in [1,\infty).
\end{split}
\end{align}

\item[(ii)] The solution $(\uu, P, \phi)$ fulfills the system \eqref{AGG} almost everywhere in $\Omega \times (0,\infty)$ and the boundary conditions $\partial_\n \phi=\partial_\n \mu=0$ almost everywhere in $\partial \Omega \times (0,\infty)$. 

\item[(iii)] The solution $(\uu, P, \phi)$ is such that $\uu(\cdot, 0)=\uu_0$ and $\phi(\cdot, 0)=\phi_0$ in $\Omega$.
\end{itemize}
\end{theorem}

\begin{proof}
Given an initial condition $(\uu_0, \phi_0)$ satisfying the required assumptions, the result in \cite[Theorem 3.1]{Gior2021} guarantees the existence and uniqueness of a local strong solution $(\uu, P, \phi)$ to system \eqref{AGG}-\eqref{AGG-bc} originating from $(\uu_0, \phi_0)$. We consider the maximal interval of existence $[0,T_{\rm{max}})$ of such solution. That is, the solution $(\uu, P, \phi)$ satisfies \eqref{reg-SS} in the interval $[0,T]$ for any $T<T_{\rm{max}}$, the system \eqref{AGG} almost everywhere in $\Omega \times (0,T_{\rm{max}})$ and the boundary conditions $\partial_\n \phi=\partial_\n \mu=0$ almost everywhere in $\partial \Omega \times (0,T_{\rm{max}})$. Furthermore, $\uu(0)=\uu_0$ and $\phi(0)=\phi_0$ in $\Omega$. Our goal is to show that $T_{\rm{max}}=\infty.$

We assume by contradiction that $T_{\rm{max}}<\infty$. First, integrating \eqref{AGG}$_3$ over $\Omega \times (0,t)$ with $t<T_{\rm{max}}$, we obtain
\begin{equation}
\label{CMphi}
\int_{\Omega} \phi(t) \, \d x= \int_{\Omega} \phi_0 \, \d x, \quad \forall \, t \in [0,T_{\rm{max}}).
\end{equation}
Multiplying \eqref{AGG}$_1$ and \eqref{AGG}$_3$ by $\uu$ and $\mu$, respectively, integrating over $\Omega$ and summing the resulting equation, we find the energy identity
$$
E(\uu(t), \phi(t))+ \int_0^t \int_{\Omega} \nu(\phi)|D \uu|^2 + |\nabla \mu|^2 \, \d x \, \d \tau= E(\uu_0,\phi_0), \quad \forall \, 0\leq t<T_{\rm{max}}.
$$
Since $E(\uu_0,\phi_0) <\infty$, we infer that, for all $0<t<T_{\rm{max}}$,
\begin{align}
\label{uL2H1-2D}
&\| \uu\|_{L^\infty(0,t;L^2(\Omega))} \leq \widetilde{C_0}, \quad  
 \| \uu\|_{L^2(0,t;H^1(\Omega))}\leq \widetilde{C_0},\\
\label{phiH1-2D}
&\| \phi\|_{L^\infty(0,t;H^1(\Omega))}\leq \widetilde{C_0}, \quad 
 \| \nabla \mu\|_{L^2(0,t;L^2(\Omega))}\leq \widetilde{C_0}.
\end{align}
Here the positive constant $\widetilde{C_0}$ depends on $E(\uu_0,\phi_0)$, but it is independent of $t$ and $T_{\rm{max}}$. In light of \eqref{uL2H1-2D}, owing to Theorem \eqref{CH-strong} and Corollary \ref{HREG-MU}, it follows that 
\begin{align}
\label{muH1-2D}
&\| \mu\|_{L^\infty(0,t;H^1(\Omega))}\leq \widetilde{C_1}, \quad 
 \| \nabla \mu\|_{L^2(0,t;H^2(\Omega))}\leq \widetilde{C_1},\\
\label{phitH1-2D}
&\| \partial_t \phi\|_{L^\infty(0,t;H^1(\Omega)')}\leq \widetilde{C_1}, \quad 
 \| \partial_t \phi\|_{L^2(0,t;H^1(\Omega))}\leq \widetilde{C_1},\\
 \label{phiW2p-2D}
 &\| \phi\|_{L^\infty(0,t;W^{2,p}(\Omega))}\leq \widetilde{C_1}(p), \quad 
 \|F'(\phi)\|_{L^\infty(0,t;L^p(\Omega))}\leq \widetilde{C_1}(p),
\end{align}
for all $0<t<T_{\rm{max}}$ and $2\leq p<\infty$, where the positive constants $\widetilde{C_1}$, $\widetilde{C_1}(p)$ are independent of $T_{\rm{\max}}$. Furthermore, by \cite[Lemma A.6]{CG2020} (see also\cite{MZ}) and \eqref{muH1-2D}-\eqref{phiW2p-2D}, we get
\begin{equation}
\label{F2-F3-2D}
\| F''(\phi) \|_{L^\infty(0,t; L^p(\Omega))}
+ \| F'''(\phi) \|_{L^\infty(0,t; L^p(\Omega))}\leq \widetilde{C_2}(p),
\end{equation}
for $p$ as above and some $\widetilde{C_2}(p)$ independent of $T_{\rm{max}}$.

Next, exploiting \eqref{AGG}$_3$, we rewrite \eqref{AGG}$_1$ in non-conservative form (cf. \eqref{NS}) as follows
\begin{equation}
\label{NS-nc}
\rho(\phi) \partial_t \uu + \rho(\phi) (\uu \cdot \nabla)\uu
-\rho'(\phi) (\nabla \mu\cdot \nabla) \uu - \div (\nu(\phi)D\uu) + \nabla P^\ast= - \phi \nabla \mu.
\end{equation}
Multiplying \eqref{NS-nc} by $\partial_t \uu$ and integrating over $\Omega$, we obtain
\begin{equation}
\label{NS-test1}
\begin{split}
\ddt \int_{\Omega} \frac{\nu(\phi)}{2} | D \uu|^2 \, \d x + \rho_\ast \| \partial_t \uu\|_{L^2(\Omega)}^2 
& \leq \left| \int_{\Omega} \rho(\phi) (\uu \cdot \nabla)\uu \cdot \partial_t \uu \, \d x\right|
+ \left| \int_{\Omega} \rho'(\phi) (\nabla \mu\cdot \nabla) \uu \cdot \partial_t \uu\, \d x \right| 
\\
&\quad +\left|  \int_{\Omega} \nu'(\phi) \partial_t \phi \, |D \uu|^2 \, \d x \right| 
+ \left| \int_{\Omega} - \phi \nabla \mu \cdot \partial_t \uu \, \d x \right|.
\end{split}
\end{equation}
Using the Ladyzhenskaya inequality, the Korn inequality and  \eqref{uL2H1-2D}, we find
\begin{equation}
\label{T1}
\begin{split}
 \left| \int_{\Omega} \rho(\phi) (\uu \cdot \nabla)\uu \cdot \partial_t \uu\, \d x\right|
 &\leq \rho^\ast \| \uu\|_{L^4(\Omega)} \| \nabla \uu\|_{L^4(\Omega)} \| \partial_t \uu\|_{L^2(\Omega)}\\
 &\leq C \widetilde{C_0}^\frac12 \| \nabla \uu\|_{L^2(\Omega)} \|  \uu\|_{H^2(\Omega)}^\frac12 \| \partial_t \uu\|_{L^2(\Omega)}\\
 &\leq \frac{\rho_\ast}{6} \| \partial_t \uu\|_{L^2(\Omega)}^2 
 + \varpi \| \uu\|_{H^2(\Omega)}^2 
 + \widetilde{C} \| D \uu\|_{L^2(\Omega)}^4,
\end{split}
\end{equation}
where $\widetilde{C}$ stands for a generic positive constant, whose value may change from line to line, which depends on the norms of the initial data, the parameters of the system and $\varpi>0$. The (small) value of $\varpi$ will be chosen subsequently. 
Similarly, we have
\begin{equation}
\label{T2}
\begin{split}
\left|  \int_{\Omega} \nu'(\phi) \partial_t \phi \, |D \uu|^2 \, \d x \right| 
& \leq  
C \| \partial_t \phi\|_{L^2(\Omega)} \| D \uu\|_{L^4(\Omega)}^2 \\
& \leq  
C \| \nabla \partial_t \phi\|_{L^2(\Omega)} \| D \uu\|_{L^2(\Omega)} \| \uu\|_{H^2(\Omega)} \\
&\leq \varpi \| \uu\|_{H^2(\Omega)}^2 
 + \widetilde{C} \| \nabla \partial_t \phi\|_{L^2(\Omega)}^2 \| D \uu\|_{L^2(\Omega)}^2.
\end{split}
\end{equation}
Besides, by the Sobolev embedding and \eqref{phiW2p-2D}
\begin{equation}
\label{T3}
\begin{split}
\left| \int_{\Omega} \rho'(\phi) (\nabla \mu\cdot \nabla) \uu \cdot \partial_t \uu\, \d x \right| 
& \leq \left| \frac{\rho_1-\rho_2}{2} \right| \| \nabla \mu\|_{L^\infty(\Omega)} \| \nabla \uu\|_{L^2(\Omega)} \| \partial_t \uu\|_{L^2(\Omega)}\\
&\leq \frac{\rho_\ast}{6} \| \partial_t \uu\|_{L^2(\Omega)}^2 
 + \widetilde{C} \| \nabla \mu\|_{H^2(\Omega)}^2 \| \nabla \uu\|_{L^2(\Omega)}^2,
\end{split}
\end{equation}
and 
\begin{equation}
\label{T4}
\begin{split}
\left| \int_{\Omega} -\phi \nabla \mu \cdot \partial_t \uu \, \d x \right|
&\leq  \| \phi\|_{L^\infty(\Omega)} \| \nabla \phi\|_{L^2(\Omega)}  \| \partial_t \uu\|_{L^2(\Omega)} \\
&\leq \frac{\rho_\ast}{6} \| \partial_t \uu\|_{L^2(\Omega)}^2 
+ \widetilde{C} \| \nabla \mu\|_{L^2(\Omega)}^2.
\end{split}
\end{equation}
On the other hand, exploiting the regularity theory of the Stokes equation with concentration-depending viscosity in \cite[Lemma 4]{Abels2009} and owing to \eqref{phiW2p-2D}, we infer that
\begin{equation}
\begin{split}
\| \uu\|_{H^2(\Omega)}^2
&\leq \widetilde{C} \left( \| \rho(\phi) \partial_t \uu \|_{L^2(\Omega)}^2
+\| \rho(\phi) (\uu \cdot \nabla)\uu\|_{L^2(\Omega)}^2 \right.
\\
& \qquad \quad \left. + \| \rho'(\phi) (\nabla \mu\cdot \nabla) \uu \|_{L^2(\Omega)}^2
+\| \phi \nabla \mu\|_{L^2(\Omega)}^2 \right).
\end{split}
\end{equation}
Arguing in a similar way as above, we deduce that
\begin{equation}
\begin{split}
\| \uu\|_{H^2(\Omega)}^2
&\leq \widetilde{C} \bigg( \rho^\ast \| \partial_t \uu \|_{L^2(\Omega)}^2
+\rho^\ast \| \uu \|_{L^4(\Omega)}^2 \| \nabla \uu\|_{L^4(\Omega)}^2 \bigg.\\
&\quad \bigg. + \left| \frac{\rho_1-\rho_2}{2}\right|\| \nabla \mu \|_{L^\infty(\Omega)}^2 \| \nabla \uu \|_{L^2(\Omega)}^2
+\| \phi\|_{L^\infty}^2 \| \nabla \mu\|_{L^2(\Omega)}^2 \bigg)\\
& \leq \widetilde{C} \left( \rho^\ast \| \partial_t \uu \|_{L^2(\Omega)}^2
+\rho^\ast \widetilde{C_0} C \| \nabla \uu\|_{L^2(\Omega)}^2 \| \uu\|_{H^2(\Omega)} \right.\\
&\quad \left. + \left| \frac{\rho_1-\rho_2}{2}\right|\| \nabla \mu \|_{H^2(\Omega)}^2 \| \nabla \uu \|_{L^2(\Omega)}^2
+ \| \nabla \mu\|_{L^2(\Omega)}^2  \right).
\end{split}
\end{equation}
Thus, using the Young inequality, we arrive at
\begin{equation}
\label{T5}
\begin{split}
\| \uu\|_{H^2(\Omega)}^2
& \leq \widetilde{C}  \| \partial_t \uu \|_{L^2(\Omega)}^2
+ \widetilde{C} \| \nabla \uu\|_{L^2(\Omega)}^4 
+ \widetilde{C} \| \nabla \mu \|_{H^2(\Omega)}^2 
\| \nabla \uu \|_{L^2(\Omega)}^2
+ \widetilde{C}\|\nabla \mu\|_{L^2(\Omega)}^2.
\end{split}
\end{equation}
Now, combining \eqref{NS-test1} with \eqref{T1}-\eqref{T4}, summing the resulting equation by $3 \varpi \times$\eqref{T5}, and setting $\varpi=\frac{\rho_\ast}{12 \widetilde{C}}$, we eventually reach the differential inequality
\begin{equation*}
 \ddt \left( \frac12 \int_{\Omega} \nu(\phi) |D \uu|^2 \, \d x\right)
+ \frac{\rho_\ast}{4} \|\partial_t \uu\|_{L^2(\Omega)}^2 
+ \frac{\rho_\ast}{12 \widetilde{C}} \| \uu\|_{H^2(\Omega)}^2
\leq G(t)
\left( \frac12 \int_{\Omega} \nu(\phi) |D \uu|^2 \, \d x\right) +
\widetilde{C} \|\nabla \mu \|_{L^2(\Omega)}^2,
\end{equation*}
where
$$
G(t)= \widetilde{C} \left( \| D \uu(t)\|_{L^2(\Omega)}^2+ \| \nabla \partial_t \phi(t)\|_{L^2(\Omega)}^2
+ \| \nabla \mu(t) \|_{H^2(\Omega)}^2 \right).
$$
The Gronwall lemma yields
\begin{equation}
\label{uH1-2D}
\begin{split}
&\max_{\tau \in[0,t]} \| \uu(\tau)\|_{H^1(\Omega)}^2
 + \int_0^{t}\|\partial_t \uu(\tau)\|_{L^2(\Omega)}^2 
+ \| \uu (\tau)\|_{H^2(\Omega)}^2 \, \d \tau\\
&\leq \widetilde{C}
\left( \| \uu_0\|_{H^1(\Omega)}^2 + \int_0^t \| \nabla \mu(s)\|_{L^2(\Omega)}^2 \right) \left( 1+ \int_0^t G(s) \, \d s\right) \mathrm{e}^{\int_0^t G(s) \, \d s}=: \widetilde{G}(t),
 \quad \forall \, t <T_{\rm{max}}.
\end{split}
\end{equation}
In light of \eqref{uL2H1-2D}, \eqref{muH1-2D} and \eqref{phitH1-2D},
$\widetilde{G} \in L^\infty(0,T_{\rm{max}})$.
On the other hand,
by the assumption $T_{\rm{max}}< \infty$, it is easily seen that
$$
\limsup_{t\rightarrow T_{\rm{max}}^-} \left( \| \uu(t)\|_{\H_\sigma^1} +  \| \mu(t)\|_{H^1(\Omega)}\right)= \infty.
$$
Otherwise, the solution could be extended beyond the time 
$T_{\rm{max}}$ thanks to \cite[Theorem 3.1]{Gior2021}. This contradicts \eqref{muH1-2D} and \eqref{uH1-2D}.
Thus, we conclude that $T_{\rm{max}}=\infty$, and the solution $(\uu, P, \phi)$ exists on $[0,\infty)$. In particular, since the estimates \eqref{uL2H1-2D}-\eqref{F2-F3-2D} holds on $[0,\infty)$, it follows that $\| \widetilde{G}\|_{L^\infty(0,\infty)}< \infty$ and thereby
$$
\uu \in BC([0,\infty); \H^1_{0,\sigma}(\Omega)) 
\cap L^2(0,\infty;\H^2(\Omega))
\cap W^{1,2}(0,\infty;\L^2_\sigma(\Omega)).
$$
In addition, since $F''(\phi) \in L^\infty(0,\infty; L^p(\Omega))$ for any $2\leq p < \infty$, by comparison with terms in \eqref{AGG}$_4$, we also deduce that $\partial \mu \in L_{\uloc}^2([0,\infty); H^1(\Omega)')$, which, in turn, implies that $\mu \in BC([0,\infty); H^1(\Omega))$.
\end{proof}

\section{Double obstacle potential: the limit $\theta \searrow 0$}
\label{obstacle}
\setcounter{equation}{0}

In this final section we study the double obstacle version of the system \eqref{AGG}-\eqref{AGG-bc} which is obtained by passing to the limit
$\theta \searrow 0$ in the Flory--Huggins potential $\Psi$, cf.  \eqref{Log}.
The limiting free energy is then given by 
$$
E^{\text{do}}_{\text{free}}(\phi)= 
 \int_{\Omega} \frac12 |\nabla \phi|^2 + I_{[-1,1]} (\phi) - \frac{\theta_0}2 \phi^2   \, \d x.
 $$
 Here $I_{[-1,1]}$ is the indicator function of the interval $[-1,1]$ given by
 \begin{equation*}
 I_{[-1,1]} (s) 	=
 	\begin{cases} 0&\mbox{ if } s \in[-1,1] \,,\\
 		\infty &\mbox{ if } s\notin  [-1,1] \,.
 	\end{cases}
 \end{equation*}
In this case equation \eqref{AGG}$_4$ has to be replaced by
\begin{equation}
	\label{DOincl}
\mu+\Delta \phi+\theta_0 \phi  \in \partial I_{[-1,1]}(\phi)
\end{equation}
almost everywhere in $\Omega \times (0,\infty)$ where $\partial I_{[-1,1]}$ is the subdifferential of
$I_{[-1,1]}$.
The inclusion \eqref{DOincl}  is equivalent to the variational inequality
$$ -\left( \mu, \zeta-\phi \right) +\left( \nabla \phi, \nabla \zeta-\nabla \phi\right) -\theta_0\left( \phi, \zeta- \phi\right) \geq 0$$
which has to hold for almost all $t$ and all $\zeta \in H^1(\Omega)$ which fulfill $|\zeta(x)|\leq 1$ almost everywhere in $\Omega$, see
\cite{BE1991}.

We first state  a result on the double obstacle limit for the convective Cahn-Hilliard equation which has been shown in
\cite[Theorem 3.10]{
Abels2011}.
\begin{theorem}[Double obstacle limit of the convective Cahn-Hilliard equation]
	\label{BasicThmDO}
	 Let $\Omega$ be a bounded domain in $\mathbb{R}^d$, $d=2,3$, with $C^2$ boundary and  $0<\theta_k\leq 1$, $ k\in \mathbb{N}$, be such that $\lim_{k\rightarrow \infty}\theta_k = 0$. Moreover, assume
	$ \phi_0, \phi_{0,k} \in H^1(\Omega)$ with $\| \phi_{0,k}\|_{L^\infty(\Omega)}\leq 1$ and $
		\sup_{k\in \mathbb{N}} \left| \overline{\phi_{0,k}} \right|<1$, and  $\uu, \uu_k \in  L^2(0,\infty;\H^1_{0,\sigma}(\Omega))$ be such that as $k$ tends to infinity
	$$ 
	\phi_{0,k} 
	\rightarrow \phi_0 \quad \text{in} \quad H^1(\Omega)
	\quad \text{and} \quad
	\uu_k \rightarrow  \uu \quad 
	\text{in}\quad L^2(0,T;\H^1_{0,\sigma}(\Omega))
	 $$ for all $T\in (0,\infty)$. Furthermore, let $(\phi_k, \mu_k)$ be the sequence of weak solutions to \eqref{CH}-\eqref{bcic} with $(\uu,\phi_0, F)$ replaced by  $(\uu_k,\phi_{0,k}, F_k)$ where $F$ is defined with $\theta_k$ instead of $\theta$.
	Then, it holds in the limit  $k$ tending to $\infty$ 
	\begin{equation}
		\label{obstacle-limit}
		\begin{split}
			\begin{aligned}
				\phi_{k} &\rightharpoonup \phi \quad  &&\text{weakly in } L^2(0,T; W^{2,p}(\Omega)),\\
				\nabla \mu_{k} & \rightharpoonup \nabla \mu \quad  &&\text{weakly  in } L^2(0,T;L^2(\Omega)),\\
				F_k'(\phi_k)&\rightharpoonup F^* \quad  &&\text{weakly in } 
				L^{2}(0,T;L^p(\Omega)),
			\end{aligned}
		\end{split}
	\end{equation}
\noindent for all $T\in(0,\infty)$, where $F^* = \mu +\Delta \phi +\theta_c \phi \in \partial  I_{[-1,1]} (\phi)$  almost 
everywhere in 
$\Omega \times (0,\infty)$, $p=6$ if $d=3$, $p\in [2,\infty)$ is arbitrary if $d=2$ 
and $(\phi,\mu)\in
C( [0,\infty); H^1(\Omega))\cap L_{\uloc}^4([0,\infty);H^2(\Omega))\cap L_{\uloc}^2([0,\infty);W^{2,p}(\Omega))\times 
L_{\uloc}^2([0,\infty);H^1(\Omega))$ is the unique weak solution of 
\begin{equation}
	\begin{cases}
		\label{CHDO}
		\partial_t \phi+ \uu \cdot \nabla \phi = \Delta \mu \\
		\mu+\Delta \phi+\theta_0 \phi \in \partial I_{[-1,1]}(\phi)
	\end{cases}
	\quad \text{ in } \Omega\times (0,\infty)
\end{equation}
completed  with the following boundary and initial conditions
\begin{equation}
	\begin{cases}
		 \partial_\n\phi=\partial_\n \mu=0 \quad &\text{on }  \partial \Omega \times (0,T),\\
		\phi(\cdot, 0)=\phi_0 \quad &\text{in } \Omega.
	\end{cases}
\end{equation}
The weak solution satisfies the free energy equality
\begin{equation}
	\label{EI-do}
	E_{\mathrm{free}}^{\mathrm{do}}(\phi(t))+ \int_\tau^t 
	\| \nabla \mu(s)\|_{L^2(\Omega)}^2 \, \d s =  E_{\mathrm{free}}^{\mathrm{do}}(\phi(\tau)) - \int_{\tau}^t (\uu \cdot \nabla \phi, \mu) \, \d s,
\end{equation}
for every $0\leq \tau< t \leq \infty$. 
\end{theorem}
The formulation in \cite[Theorem 3.10]{
	Abels2011} is slightly different. However, e.g., the variable mean values $\overline{\phi_{0,k}}$ do not effect the arguments substantially.

We  now formulate our result stating the higher regularity for the  Cahn-Hilliard equation with divergence-free drift in the double obstacle case.
\begin{theorem}
\label{CH-strong-do}
Let $\Omega$ be a bounded domain in $\mathbb{R}^d$, $d=2,3$, with $C^3$ boundary and the initial condition $\phi_0 \in H^2(\Omega)$ be such that 
$\| \phi_0\|_{L^\infty(\Omega)}\leq 1$, $\left| \overline{\phi_0} \right|<1$.
Furthermore, we assume that a function $\mu_0\in H^1(\Omega)$ exists such that
\begin{equation}
\label{init-do}
\mu_0+\theta_0 \phi_0 +\Delta \phi_0 \in  \partial I_{[-1,1]}(\phi_0)  \quad \text{almost everywhere in } \Omega
\end{equation}
and 
$\partial_\n \phi_0=0$ on $\partial \Omega$. 
Assume that $\uu \in L^2(0,\infty;\H^1_{0,\sigma}(\Omega))$.
Then, there exists a unique global  solution to 
\begin{equation}
\begin{cases}
	\label{CHDO-reg}
		\partial_t \phi+ \uu \cdot \nabla \phi = \Delta \mu \\
			\mu+\Delta \phi+\theta_0 \phi \in \partial I_{[-1,1]}(\phi)
\end{cases}
\quad \text{ almost everywhere in } \Omega\times (0,\infty),
\end{equation}
such that 
\begin{equation}
\label{REGdo}
\begin{split}
&\phi \in L^\infty(0,\infty;W^{2,p}(\Omega)), \quad \partial_t \phi \in L^2(0,\infty;H^1(\Omega)),\\
&\phi \in L^{\infty}(\Omega\times (0,\infty)) \text{ with }  |\phi(x,t)|\leq 1 
\ \text{a.e. in } \Omega\times (0,\infty),\\
&\mu \in L^{\infty}(0,\infty; H^1(\Omega))\cap L_{\uloc}^2([0,\infty);H^3(\Omega)),
		\end{split}
	\end{equation}
	for any $2\leq p <\infty$ if $d=2$ and $p=6$ if $d=3$. The solution satisfies \eqref{CHDO-reg} almost everywhere in $\Omega \times (0,\infty)$ and  $\partial_\n \mu=0$ almost everywhere on 
	$\partial\Omega\times(0,\infty)$. Moreover, there exists a positive constant $C$ depending only on $\Omega$,  $\theta_0$,
	and $\overline{\phi_0}$ such that
\begin{equation}
\label{N-mu-do}
\begin{split}
\| \nabla \mu \|_{L^\infty(0,\infty;L^2(\Omega))}
			& \leq \left(  4 \| \nabla \mu_0\|_{L^2(\Omega)}^2+  4C \int_0^\infty \| \nabla \uu(s)\|_{L^2(\Omega)}^2 \, \d s\right)^\frac12 
			\\
&\quad \times\mathrm{exp}\left( C \int_0^\infty \| \nabla \uu(s)\|_{L^2(\Omega)}^2 \, \d s \right),
		\end{split}
	\end{equation}
	and
	\begin{equation}
		\label{N-phit-do}
		\begin{split}
			\int_0^\infty \| \nabla \partial_t \phi(s)\|_{L^2(\Omega)}^2 \, \d s
			&\leq 
			6 \left(  \| \nabla \mu_0\|_{L^2(\Omega)}^2+ C \int_0^\infty \| \nabla \uu(s)\|_{L^2(\Omega)}^2 \, \d s\right) \\
			&\quad \times \left( 1+\left( \int_0^\infty \| \nabla \uu(s)\|_{L^2(\Omega)}^2 \, \d s \right) \ \mathrm{exp}\left( 2C \int_0^\infty \| \nabla \uu(s)\|_{L^2(\Omega)}^2 \, \d s \right) \right). 
		\end{split}
	\end{equation}
In addition, if $\uu \in  L^\infty(0, \infty;\L^2_\sigma(\Omega))\cap L^2(0,\infty;\H^1_{0,\sigma}(\Omega))$, then $\partial_t \phi \in L^\infty(0,\infty; H^1(\Omega)')$.
\end{theorem} 
\begin{proof}
The argument relies on the limit $\theta =1/k \searrow 0$, with $k\in \mathbb{N}$, in the problem \eqref{CH}-\eqref{bcic}. To this end, we will use Theorem \ref{CH-strong} with appropriate initial data $\phi_0^k$. To fulfill the assumptions on the initial data in Theorem \ref{CH-strong}, we solve the elliptic boundary value problem 
\begin{equation}
\label{DO-apr-in}
\begin{cases}
-\Delta \phi_0^k+\frac 1k F_0^\prime (\phi_0^k)+\phi_0^k=\mu_0+\theta_0 \phi_0 +\phi_0=:\tilde \mu_0 \quad &\text{in } \Omega,\\
\partial_\n \phi_0^k =0\quad &\text{on } \partial \Omega,
\end{cases}
\end{equation}
where we choose $F_0(s)= \frac{1}{2} [ (1+s)\log(1+s)+(1-s)\log(1-s)]$.
Some of the following estimates are formal but can be justified by approximating $F_0$ by smooth funtions with quadratic growth.
Testing \eqref{DO-apr-in} with $\frac 1k F_0^\prime (\phi_0^k)$ and using the fact that $F_0$ is monotone gives
that $\frac 1k F_0^\prime (\phi_0^k)$ is uniformly in $k$ bounded in $L^2(\Omega)$. Now
elliptic regularity theory gives that $\phi_0^k$ is uniformly in $k$ bounded in $H^2(\Omega)$. 
As $F_0^\prime (\phi_0^k) \in L^2(\Omega)$, we obtain that $|\phi_0^k(x)|<1$ almost everywhere in $\Omega$. 
For a subsequence, we obtain that $\phi_{0}^k$ converges to $\phi^*$ weakly in $H^2(\Omega)$ and strongly in
  $H^1(\Omega)$.
  We now choose $\widetilde{\eta}\in H^1(\Omega)$ with $|\widetilde{\eta}(x)|\le 1$ almost everywhere in $\Omega$ and use
  $\eta=\widetilde{\eta}-\phi_0^k$ as test function in the weak formulation of  \eqref{DO-apr-in} obtaining
\begin{equation}
\label{eq:theta10}
\begin{split}
&
-\int_{\Omega} \tilde \mu_0 \left( \widetilde{\eta}-\phi_0^k \right)\, \d x +
 \int_{\Omega} \nabla\phi_0^k \cdot \nabla \left(\widetilde{\eta}-\phi_0^k\right)\, \d x +  \int_{\Omega} \phi_0^k \left(\widetilde{\eta}-\phi_0^k \right)\, \d x \\
&\quad 
= -\int_{\Omega} \frac 1k  F_0'(\phi_0^k)
  	\left(\widetilde{\eta}-\phi_0^k \right)\, \d x.
 \end{split}
 \end{equation}
As $F_0$ is convex, we have
  \begin{equation*}
  	F_0(\widetilde{\eta})\ge F_0(\phi_0^k) +
  	F_0'(\phi_0^k)(\widetilde{\eta}-\phi_0^k).
  \end{equation*}
Then, the right hand side in \eqref{eq:theta10} is greater or equal than
  \begin{equation*}
  	\int_{\Omega} \frac 1k \left(F_0(\phi_0^k )-F_0(\widetilde{\eta})\right) \, \d x 
  \end{equation*}
which converges to zero for $k$ tending to infinity.
The above convergence properties now give the inequality
 \begin{equation*}
 \label{eq:theta11}
 -\int_{\Omega} \tilde \mu_0 \left(\widetilde{\eta}-\phi^* \right)\, \d x +
  	\int_{\Omega} \nabla\phi^* \cdot
  	\nabla \left(\widetilde{\eta}-\phi^* \right)\, \d x + 
  		\int_{\Omega} \phi^* \cdot
  	\left(\widetilde{\eta}-\phi^* \right)\, \d x
  	\geq 0,
  \end{equation*}
which gives that
$$\tilde \mu_0+\Delta \phi^* -\phi^* = \mu_0 +\theta_0 \phi_0 + \phi_0+\Delta \phi^*-\phi^* \in \partial I_{[-1,1]} (\phi^*)$$ 
together with zero Neumann boundary conditions. As the operator $\phi\mapsto -\Delta \phi +\phi   + \partial I_{[-1,1]}(\phi) $ is strictly monotone
we obtain from \eqref{init-do} that $\phi_0=\phi^*$. 
 
We now consider the sequence of strong solutions $\lbrace (\phi^k,\mu^k) \rbrace$ given by Theorem \ref{CH-strong} with $\theta=\frac 1k$ and $\phi_0^k$ instead of $\phi_0$. In order to derive estimates, which are uniform in $k$, from \eqref{N-mu} and\eqref{N-phit}, we need to control
$  -\Delta \phi_0^k+\frac 1k F_0^\prime (\phi_0^k)-\theta_0 \phi_0^k $ in $H^1(\Omega)$. In fact, by construction in \eqref{DO-apr-in}, we notice that
$$ 
-\Delta \phi_0^k+\frac 1k F_0^\prime (\phi_0^k)-\theta_0 \phi_0^k  =
-\Delta \phi_0^k+
\frac 1k F'_0(\phi_0^k)+\phi_0^k-\theta_0 \phi_0^k-\phi_0^k=  
\mu_0+\theta_0 (\phi_0-\phi_0^k) +(\phi_0-\phi_0^k)  =:\mu_0^k
$$
and we observe that $\mu_0^k$ is  uniformly bounded in $H^1(\Omega)$ in $k$. 
Setting $\Psi_{\frac 1k} (s) =\frac 1k \frac{1}{2}\big[ (1+s)\log(1+s)+(1-s)\log(1-s)\big]
 -\frac{\theta_0}{2}s^2 =\frac 1k F_0(s) - \frac{\theta_0}{2}s^2 $, we infer that
 \begin{equation}
 \label{N-mu-proof-do}
 	\begin{split}
 \| \nabla \mu^k \|_{L^\infty(0,\infty;L^2(\Omega))} 
 		& \leq \left(  4 \left\| \nabla \left(-\Delta \phi^k_{0} +\Psi'_{\frac 1k} (\phi_{0}^k) \right) \right\|_{L^2(\Omega)}^2+  4C \int_0^\infty \| \nabla \uu(s)\|_{L^2(\Omega)}^2 \, \d s\right)^\frac12 \\
 		&\quad \times
 		\mathrm{exp}\left( C \int_0^\infty \| \nabla \uu(s)\|_{L^2(\Omega)}^2 \, \d s \right)
 		\\
 		& \leq \left(  4 \left\| \nabla \left( \mu_0 +(\theta_0+1)(\phi_0-\phi^k_0) \right) \right\|_{L^2(\Omega)}^2+  4C \int_0^\infty \| \nabla \uu(s)\|_{L^2(\Omega)}^2 \, \d s\right)^\frac12 \\
 		&\quad \times
 		\mathrm{exp}\left( C \int_0^\infty \| \nabla \uu(s)\|_{L^2(\Omega)}^2 \, \d s \right)
 	\end{split}
 \end{equation}
 and
 \begin{equation}
 	\label{N-phitproof-do}
 	\begin{split}
 		\int_0^\infty \| \nabla \partial_t \phi^k(s)\|_{L^2(\Omega)}^2 \, \d s
 		&\leq 
 		6 \left(  \left\| \nabla \left(-\Delta \phi^k_{0} +\Psi'_{\frac 1k} (\phi_{0}^k) \right) \right\|_{L^2(\Omega)}^2+ C \int_0^\infty \| \nabla \uu(s)\|_{L^2(\Omega)}^2 \, \d s\right) \\
 		&\quad \times \left( 1+\left( \int_0^\infty \| \nabla \uu(s)\|_{L^2(\Omega)}^2 \, \d s \right) \ \mathrm{exp}\left( 2C \int_0^\infty \| \nabla \uu(s)\|_{L^2(\Omega)}^2 \, \d s \right) \right)
 		\\
 		&\leq 6 \left(   \left\| \nabla \left( \mu_0 +(\theta_0+1)(\phi_0-\phi^k_0) \right) \right\|_{L^2(\Omega)}^2+ C \int_0^\infty \| \nabla \uu(s)\|_{L^2(\Omega)}^2 \, \d s\right) \\
 		&\quad \times \left( 1+\left( \int_0^\infty \| \nabla \uu(s)\|_{L^2(\Omega)}^2 \, \d s \right) \ \mathrm{exp}\left( 2C \int_0^\infty \| \nabla \uu(s)\|_{L^2(\Omega)}^2 \, \d s \right) \right).
 	\end{split}
 \end{equation}
In particular, the constant $C$ depends on $\Omega$, $\theta_0$, and $\overline{\phi_0} $, but is independent of $k$.
In the limit as $k$ tends to $\infty$, we obtain that $\nabla \mu\in L^\infty(0,\infty;L^2(\Omega))$
and $\nabla \partial_t \phi \in L^2(0,\infty;L^2(\Omega))$. In addition, thanks to \cite[Proposition 3.3]{Abels2011}, we deduce that $\phi \in L^\infty(0, \infty; W^{2,p}(\Omega))$ for $2\leq p < \infty$ if $d=2$ and $p=6$ if $d=3$. Also,   the estimates \eqref{N-mu-do} and \eqref{N-phit-do} hold due to the fact that
$\phi_0^k \rightarrow \phi_0$ in $H^1(\Omega)$ and the lower semicontinuity of the norm. The rest of the proof follows by arguing similarly as in Theorem \ref{CH-strong} (see also \cite{Abels2011}).
\end{proof}

We now study the double obstacle limit of the full Navier-Stokes/Cahn-Hilliard system \eqref{AGG} with the
boundary conditions \eqref{AGG-bc}.  The limiting system is given as
\begin{equation}
	\label{AGG-DO}
\begin{cases}
	\partial_t \left( \rho(\phi)\uu\right) + \div \left( \uu \otimes \left( \rho(\phi) \uu + \widetilde{\J} \right) \right) - \div \left( \nu(\phi) D \uu \right) + \nabla P= - \div \left( \nabla \phi \otimes \nabla \phi \right),\\
	\div \, \uu=0,\\
	\partial_t \phi +\uu\cdot \nabla \phi = \Delta \mu,\\
	\mu+\Delta \phi+\theta_0 \phi  \in \partial I_{[-1,1]}(\phi),
\end{cases}
\end{equation}  
in $\Omega \times (0,\infty)$,
subject to the initial and boundary conditions
\begin{equation}
	\label{AGG-DO-bc}
	\begin{cases}
		\uu=\mathbf{0}, \quad \partial_\n\phi=\partial_\n \mu=0 \quad &\text{on }  \partial \Omega \times (0,T),\\
		\uu|_{t=0}=\uu_0, \quad \phi|_{t=0}=\phi_0 \quad &\text{in } \Omega.
	\end{cases}
\end{equation}
Combining the arguments of \cite{Abels2011}, \cite{ADG2013} and the methods of this paper we obtain the following result.
\begin{theorem}[Weak solutions in the double obstacle case]
\label{AGG-DO-Thm}
Let $\Omega$ be a bounded domain in $\mathbb{R}^d$, $d=2,3$, with boundary $\partial \Omega$ of class $C^2$. Moreover, let  $0<\theta_k\leq 1$, $ k\in \mathbb{N}$, be such that $\lim_{k\rightarrow \infty}\theta_k = 0$. 
Assume that $\phi_0, \phi_{0,k} \in H^1(\Omega)$ with $\| \phi_{0,k}\|_{L^\infty(\Omega)}\leq 1$ and $ \sup_{k\in \mathbb{N}}|\overline{\phi_{0,k}}|<1$, and  $\uu_0, \uu_{0,k} \in  \H^1_{0,\sigma}(\Omega)$ be such that 
$$ 
\uu_k \rightarrow  \uu \quad \text{in}\quad \H^1_{0,\sigma}(\Omega), \quad \phi_{0,k}  \rightarrow \phi_0 \quad \text{in} \quad H^1(\Omega) \quad \text{as } \, k \rightarrow \infty.
$$ 
	  In addition, let 
	$(\uu_k, \phi_k,\mu_k)$ be weak solutions to \eqref{AGG}-\eqref{AGG-bc} with initial values $(\uu_{0,k}, \phi_{0,k})$. Then, there exists a subsequence $k_j$, $j\in\mathbb{N}$, with $k_j\rightarrow \infty$ as $j\rightarrow \infty$,
such that as $j$ tends to infinity
\begin{equation}
\label{obstacle-limit-AGG}
\begin{split}
\begin{aligned}
(\uu_{k_j}, \nabla \mu_{k_j})	 &\rightharpoonup (\uu, \nabla \mu) \quad  &&\text{weakly in } L^2(0,\infty; \H_{0,\sigma}^1(\Omega)\times L^2(\Omega)),\\
(\uu_{k_j}, \phi_{k_j}) & \rightharpoonup (\uu, \phi) \quad  &&\text{weak-star in  } L^\infty(0,\infty;\L_\sigma^2(\Omega)\times H^1(\Omega)),\\ (\phi_{k_j}, \mu_{k_j})&\rightharpoonup (\phi, \mu)  &&\text{weakly in } 
L^{2}(0,T;W^{2,p}(\Omega)\times L^2(\Omega)),
\end{aligned}
\end{split}
\end{equation}
for all $T\in(0,\infty)$, with $p\in [2,\infty)$ arbitrary if $d=2$ and $p=6$ if $d=3$, and $(\uu, \phi,\mu)$ is a weak solution to \eqref{AGG-DO}--\eqref{AGG-DO-bc}.
Furthermore, the energy inequality
	\begin{equation}
		\label{energy-do-ineq}
		E^{\mathrm{do}}(\uu(t), \phi(t))+ \int_s^t \left\| \sqrt{\nu(\phi(\tau))} D \uu(\tau)\right\|_{L^2(\Omega)}^2 + \| \nabla \mu(\tau)\|_{L^2(\Omega)}^2 \, \d \tau \leq  E^{\mathrm{do}}(\uu(s),\phi(s))
	\end{equation}
	holds for all $t \in [s, \infty)$ and almost all $s \in [0,\infty)$ (including $s=0$). Here, the total energy $E^{\mathrm{do}}(\uu,\phi)$ is given by
	$$E^{\mathrm{do}}(\uu,\phi)= E_{\mathrm{kin}}(\uu, \phi) + E^{\mathrm{do}}_{\mathrm{free}}(\phi).
	$$
\end{theorem}

The following result now states additional regularity for the concentration of any weak solution obtained in the previous Theorem \ref{AGG-DO-Thm}. The proof of which can be performed exactly in the same way as in the proof of 
Theorem \ref{MAIN} (i), see the end of Section \ref{131}.

\begin{theorem}[Regularity  of weak solutions in the double obstacle case]
\label{MAIN-DO}
Let $\Omega$ be a bounded domain in $\mathbb{R}^d$, $d=2,3$, with boundary $\partial \Omega$ of class $C^3$. Consider a global weak solution $(\uu,\phi)$ given by Theorem \ref{AGG-DO-Thm}. Then, for any $\tau >0$, we have
		\begin{equation}
			\label{REG-CONC-DO}
			\begin{split}
				&\phi \in L^\infty( \tau,\infty; W^{2,p}(\Omega)), \quad \partial_t \phi \in L^2( \tau,\infty;H^1(\Omega)),\\
				&\mu \in L^{\infty}(\tau,\infty; H^1(\Omega))\cap L_{\uloc}^2([\tau,\infty);H^3(\Omega)),
			\end{split}
\end{equation}
where $2\leq p <\infty$ if $d=2$ and $p=6$ if $d=3$. Moreover, the equations \eqref{AGG-DO}$_{3-4}$ are satisfied almost everywhere in $\Omega \times (0,\infty)$ and  the boundary conditions $\partial_\n \phi=\partial_\n \mu=0$ holds almost everywhere on $\partial\Omega\times(0,\infty)$. 
\end{theorem}

Let us finally comment about the longtime behavior. In the double obstacle case, we cannot obtain enough large time regularity to show the same results as in the case of a logarithmic potential. This is mainly due to the fact that a separation property does not hold in this case. In addition, as typical for obstacle problems, the variable $\phi$ in space will not lie in $H^3(\Omega)$.
This is already true for simple stationary solutions, see, e.g., \cite{BE1991}.
In addition, no Lojasiewicz-Simon inequality is known so far in the double obstacle case and hence it is not possible to show that the solution converges for large time to a single stationary solution. However, we can still characterize  the $\omega$-limit set of a weak solution. As in Lemma \ref{lem:ConvU} one can show that the velocity for large times converges to $0$.
Concerning the concentration $\phi$, its long time limits consist of critical points of the energy $E^{\text{do}}_{\text{free}}$.
Defining the convex part of  $E^{\text{do}}_{\text{free}}$ as 
$$
E_0^{\text{do}} (\phi)= 
\int_{\Omega} \frac12 |\nabla \phi|^2 + I_{[-1,1]} (\phi)\, \d x,
$$ 
we now recall the characterization of
the subgradient $\partial E_0^{\text{do}}$ given in \cite{Abels2011, KenmochiNP1995}. 
\begin{theorem}
The domain of definition of the subgradient $\partial E_0^{\operatorname{do}}$ with respect to $L^2_{(0)}(\Omega)$ of $ E_0^{\operatorname{do}}$ is given as
\begin{equation*}
\mathcal{D}(\partial E_{0}^{\operatorname{do}})= 
\left \lbrace \phi\in H^2(\Omega)\cap L^2_{(m)}(\Omega)\,:\,\phi(x) \in [-1,1] \  \text{a.e. in } \Omega, \ \partial_\n \phi|_{\partial\Omega} =0 \right\rbrace.
\end{equation*} 
For any $\phi_0\in \mathcal{D}(\partial E_{0}^{\operatorname{do}})$, we have   $\mu_0\in \partial E_{0}^{\operatorname{do}}(\phi)$
if and only if
\begin{equation*}
	\mu_0(x) +\Delta \phi (x) +\overline \mu \in \partial I_{[-1,1]}(\phi(x)) \quad \text{a.e. in } \Omega
\end{equation*} 
for some constant $\overline \mu$.
\end{theorem}

The stationary solutions $\phi_\infty \in \mathcal{D}(\partial E_{0}^{\operatorname{do}})$   of the double obstacle version of the Cahn-Hilliard equation fulfill (with a suitable constant $\overline \mu$)  
\begin{alignat}{2}
\label{eq:Stat-do-CH1}
\overline \mu +\Delta \phi + \theta_0 \phi&\in \partial I_{[-1,1]}(\phi) &\qquad&\text{in } \Omega,\\
\label{eq:Stat-do-CH2}
\partial_{\n} \phi &=0 & & \text{on } \partial\Omega,\\
\label{eq:Stat-do-CH3}
\frac{1}{|\Omega|}\int_\Omega \phi(x) \, \d x &= m.
\end{alignat}
We notice that the solutions to \eqref{eq:Stat-do-CH1}-\eqref{eq:Stat-do-CH3} are critical points of the functional $E^{\operatorname{do}}_{\operatorname{free}}$ on $H^1_{(m)}(\Omega)$. For the long time behavior, we find with the help of the energy inequality \eqref{energy-do-ineq} and the free energy identity
\eqref{EI-do} similar as in Lemma \ref{lem:ConvU} and Lemma \ref{lem:OmegaLimit} the following result 

\begin{theorem}\label{lem:OmegaLimit-do}
Let $(\uu,\phi)$ be a weak solution to \eqref{AGG-DO}-\eqref{AGG-DO-bc} in the sense of Theorem \ref{AGG-DO-Thm}. Then, we have
	
	$$\uu(t)\to {\bf 0} \ \text{ in } \ \L_{\sigma}^2(\Omega), \quad \text{ as } t\to\infty$$
	and
	\begin{equation*}
		\omega(\uu,\phi)\subseteq \left\{ ({\bf 0},\phi'): \phi'\in W^{2,p}(\Omega)\cap H^1_{(m)}(\Omega)\
		\text{solves (\ref{eq:Stat-do-CH1})-(\ref{eq:Stat-do-CH3})}\right\},
	\end{equation*}
	where $m = \overline{\phi_0}$ and $\omega(\uu,\phi)$ is the $\omega$-limit set with respect to the norm of $\L^2_\sigma\times W^{2-\eps,p}(\Omega)$ for an arbitrary $\eps>0$  and $p$ as before.  
\end{theorem}

\section*{Acknowledgement}
\noindent
Part of this work was done while AG was visiting HA and HG at the 
Department of Mathematics of the University of Regensburg whose hospitality is gratefully acknowledged. AG is a member of Gruppo Nazionale per l'Analisi Matematica, la Probabilit\'a e le loro Applicazioni (GNAMPA) of Istituto Nazionale per l'Alta Matematica (INdAM). {Finally, we are grateful to the anonymous referees for their helpful comments, which improved the manuscript.} 

\section*{Conflict of Interests and Data Availability Statement}

There is no conflict of interests.

There is no associated data to the manuscript.

\end{document}